\documentclass[10pt, reqno]{amsart}
\usepackage{amssymb}
\usepackage{hyperref}

\numberwithin{equation}{section}

\newtheorem{theorem}{Theorem}[section]

\newtheorem{lemma}[theorem]{Lemma}

\newtheorem{corollary}[theorem]{Corollary}

\newtheorem{proposition}[theorem]{Proposition}

\theoremstyle{definition}

\newtheorem{remark}[theorem]{Remark}

\theoremstyle{remark}
\newcommand{\R}{\mathbb{R}}
\newcommand{\C}{\mathbb{C}}

\newcommand{\E}{\mathcal{E}}
\newcommand{\M}{\mathcal{M}}
\newcommand{\N}{\mathcal{N}}

\let\Re=\undefined\DeclareMathOperator*{\Re}{Re}
\let\Im=\undefined\DeclareMathOperator*{\Im}{Im}

\newcommand{\PLo}{P_{\text{lo}}}
\newcommand{\PHi}{P_{\text{hi}}}
\newcommand{\ulin}{{u_\text{\upshape lin}}}

\newcommand{\eps}{\varepsilon}
\newcommand{\wh}{\widehat}
\DeclareMathOperator*{\supp}{supp}

\newcommand{\Hr}{H^1_{\text{\upshape real}}+i\dot{H}^1_{\text{\upshape real}}}

\newcommand{\jb}{\langle\nabla\rangle}

\newcommand{\qtq}[1]{\quad\text{#1}\quad}

\newcounter{smalllist}
\newenvironment{SL}{\begin{list}{{\rm(\roman{smalllist})\hss}}{%
\setlength{\topsep}{0mm}\setlength{\parsep}{0mm}\setlength{\itemsep}{0mm}%
\setlength{\labelwidth}{2.0em}\setlength{\itemindent}{2.5em}\setlength{\leftmargin}{0em}\usecounter{smalllist}%
}}{\end{list}}

\newenvironment{CI}{\begin{list}{{\ $\bullet$\ }}{%
\setlength{\topsep}{0mm}\setlength{\parsep}{0mm}\setlength{\itemsep}{0mm}%
\setlength{\labelwidth}{0mm}\setlength{\itemindent}{-1.5em}\setlength{\leftmargin}{1.5em}%
\setlength{\labelsep}{0mm} }}{\end{list}}

\begin{document}

\title[Cubic-quintic NLS]{The final-state problem for the cubic-quintic NLS with non-vanishing boundary conditions}
\begin{abstract}
We construct solutions with prescribed scattering state to the cubic-quintic NLS
$$
(i\partial_t+\Delta)\psi=\alpha_1 \psi-\alpha_{3}\vert \psi\vert^2 \psi+\alpha_5\vert \psi\vert^4 \psi
$$
in three spatial dimensions in the class of solutions with $|\psi(x)|\to c >0$ as $|x|\to\infty$.  This models disturbances in an infinite expanse of (quantum) fluid in its quiescent state ---  the limiting modulus $c$ corresponds to a local minimum in the energy density.

Our arguments build on work of Gustafson, Nakanishi, and Tsai on the (defocusing) Gross--Pitaevskii equation.  The presence of an energy-critical nonlinearity and changes in the geometry of the energy functional add several new complexities.  One new ingredient in our argument is a demonstration that solutions of such (perturbed) energy-critical equations exhibit continuous dependence on the initial data with respect to the \emph{weak} topology on $H^1_x$.
\end{abstract}
							
							\author{Rowan Killip}
							\address{Department of Mathematics,
							UCLA, Los Angeles, USA}
							\email{killip@math.ucla.edu}
							\author{Jason Murphy}
							\address{Department of Mathematics,
							University of California, Berkeley, USA}
							\email{murphy@math.berkeley.edu}
							\author{Monica Visan}
							\address{Department of Mathematics,
							UCLA, Los Angeles, USA}
							\email{visan@math.ucla.edu}
								
							\maketitle

								%%%%%%%%%%
								\section{Introduction}
								%%%%%%%%%%

We study the cubic-quintic nonlinear Schr\"odinger equation (NLS) with non-vanishing boundary conditions in three space dimensions:	
	\begin{equation}
	\label{eq:cq}
	\left\{\begin{array}{ll}
	(i\partial_t+\Delta)\psi=\alpha_1 \psi-\alpha_{3}\vert \psi\vert^2 \psi+\alpha_5\vert \psi\vert^4 \psi,\quad(t,x)\in\R\times\R^3,\\
	 \psi(0)=\psi_0.
	\end{array}\right.
	\end{equation}
We consider parameters $\alpha_1,\alpha_{3},\alpha_5>0$ so that $\alpha_{3}^2-4\alpha_1\alpha_5>0$, which guarantees that the polynomial $\alpha_1-\alpha_{3} x+\alpha_5 x^2$ has two distinct positive roots $r_0^2>r_1^2>0$. The boundary condition is given by
\begin{align}\label{eq:bc0}
\lim_{\vert x\vert\to\infty}\vert \psi(t,x)\vert=r_0.
\end{align}
The choice of the larger root guarantees the energetic stability of the constant solution; it constitutes a local minimum of the energy functional \eqref{energy}. 

The equation \eqref{eq:cq} appears in a great variety of physical problems.  It is a model in superfluidity
\cite{Ginsburg1958, Ginsburg1976}, descriptions of bosons \cite{Barashenkov} and  of defectons \cite{Pushakarov1978},
the theory of ferromagnetic and molecular chains \cite{Pushakarov1984, Pushakarov1986}, and in nuclear hydrodynamics \cite{Kartavenko}.  The popularity of this model can be explained by its simplicity combined with the fact that it captures an important phenomenology: the constituents of most fluids experience an attractive interaction at low densities and a repulsion at high densities.  The recent paper \cite{KOPV:35} focuses on the analogous problem with data decaying at infinity, which constitutes a model for the dynamics of a finite body of fluid; the model \eqref{eq:cq} describes the behavior of a localized disturbance in an infinite expanse of fluid that is otherwise quiescent.

By rescaling both space-time and the values of $\psi$, it suffices to consider the case $r_0^2=1$ and $\alpha_5=1$. This leaves one free parameter 
	\begin{equation}\label{def:gamma}
	\gamma:=1-r_1^2\in(0,1),
	\end{equation}
in terms of which equation \eqref{eq:cq} becomes
	\begin{equation}
	\label{eq:cq2}
	\left\{\begin{array}{ll}
	(i\partial_t+\Delta)\psi=(\vert \psi\vert^2-1)(\vert \psi\vert^2-1+\gamma)\psi,\\
	 \psi(0)=\psi_0,
	\end{array}\right.
	\end{equation}
with the boundary condition
	\begin{equation}
	\label{eq:bc}
	\lim_{\vert x\vert\to\infty} \psi(t,x)=1.
	\end{equation}

As discussed in \cite{Gerard} (albeit in the context of the Gross--Pitaevskii equation), finite energy functions obeying \eqref{eq:bc0} have a limiting phase as $|x|\to\infty$, which we can normalize to be zero, yielding \eqref{eq:bc}. Furthermore, the dynamics of \eqref{eq:cq} preserve the value of this phase, so that the boundary condition is independent of time, as well. This breaks the gauge invariance of \eqref{eq:cq} and prohibits using a phase factor to remove the linear term in this equation. The presence of the linear term leads to weaker dispersion at low frequencies, which presents a key challenge in understanding the long-time behavior of solutions.  

We are interested in perturbations of the constant solution $\psi\equiv 1$, and thus it is natural to introduce the function $u=u_1+iu_2$ defined via $\psi=1+u$. Using \eqref{eq:cq2}, we arrive at the following equation for $u$:
	\begin{equation}
	\label{eq:cq3}
	\left\{\begin{array}{ll}
	(i\partial_t+\Delta)u-2\gamma u_1=N(u),
	\\ u(0)=u_0,
	\end{array}\right.
	\end{equation}
where $N(u)=\sum_{j=2}^5 N_j(u)$, with
\begin{align*}
N_2(u)&=(3\gamma+4)u_1^2+\gamma u_2^2+2i\gamma u_1u_2, \\
N_3(u)&=(\gamma+8)u_1^3+(\gamma+4)u_1u_2^2+i[(\gamma+4)u_1^2u_2+\gamma u_2^3], \\
N_4(u)&=5u_1^4+6u_1^2u_2^2+u_2^4+i[4u_1^3u_2+4u_1u_2^3],\\
N_5(u)&= |u|^4u=u_1^5+2u_1^3u_2^2+u_2^4u_1+i[u_1^4u_2+2u_1^2u_2^3+u_2^5] .
\end{align*}
	
The Hamiltonian for \eqref{eq:cq2} is given by
\begin{align}\label{energy}
E(\psi)=\tfrac12\int_{\R^3}\vert\nabla\psi\vert^2\,dx+\tfrac{\gamma}{4}\int_{\R^3}(\vert\psi\vert^2-1)^2\,dx+\tfrac16\int_{\R^3}(\vert\psi\vert^2-1)^3\,dx.
\end{align}
Introducing the notation
	$$q(u):=\vert\psi\vert^2-1=2u_1+\vert u\vert^2,$$
we may write
	$$2\gamma u_1+ N(u)=[\gamma q(u)+q(u)^2](1+u)$$
and
\begin{equation}\label{eq:ham2}
E(1+u)=\tfrac12\int_{\R^3}\vert\nabla u \vert^2\,dx+\tfrac\gamma4\int_{\R^3}q(u)^2\,dx+\tfrac16\int_{\R^3}q(u)^3\,dx.
\end{equation}	
In the sequel we will write $E(u)$ for $E(1+u)$; when there is no risk of confusion we will simply write $q(u)=q$.  Note that $q$ represents density fluctuations relative to the constant background.  The quantity $\int q(t,x)\,dx$, which represents the matter/mass, is conserved in time; in this work we do not rely on this conservation law.

We define the energy space for \eqref{eq:cq3} to be
	\begin{equation}\label{eq:energyspace}
	\E:=\{u\in\dot{H}_x^1(\R^3):q(u)\in L_x^2(\R^3)\},
	\end{equation}
with associated metric
$$
[d_\E(u, v)]^2:= \|u-v\|_{\dot H^1_x}^2 + \|q(u)-q(v)\|_{L_x^2}^2,
$$
and we let $\|u\|_{\E}:=d_\E(u,0)$ denote the energy-norm.
	  
To justify our choice of energy space, we first note that functions with finite energy-norm have finite energy.  Indeed, using Sobolev embedding and the fact that $(L^3_x+L^6_x)\cap L^2_x\subset L^3_x$, it is not hard to see that if $u\in\E$ then $q(u)\in L_x^3$, and so $\vert E(u)\vert<\infty$. In fact,
$$
\vert E(u)\vert\lesssim \|u\|_{\E}^2+\|u\|_{\E}^3.
$$ 
On the other hand, in Lemma~\ref{lemma:coercive1} we will show that for $\gamma\in[\tfrac23,1)$, functions with finite energy have finite energy-norm. When $\gamma\in(0,\frac23)$, the energy is not coercive unless we impose an additional smallness assumption (see Lemma~\ref{lemma:coercive2}). 

When the energy is not coercive, there is no unique candidate for the name `energy space'.  The authors of \cite{KOPV} worked with the following notion of energy space:
$$
\E_{\text{KOPV}}:=\{u\in \dot{H}_x^1(\R^3)\cap L^4_x(\R^3): \Re u\in L^2_x(\R^3)\}.
$$
Note that $\E_{\text{KOPV}}\subset \E$.  In \cite{KOPV}, it was proved that \eqref{eq:cq3} is globally wellposed for data $u_0\in \E_{\text{KOPV}}$; in particular, solutions are unconditionally unique in $C(\R;\E_{\text{KOPV}})$.

In Section~\ref{section:gwp}, we prove global well-posedness and unconditional uniqueness for \eqref{eq:cq3} in the energy space $\E$ (see Theorem~\ref{thm:gwp}). As in \cite{KOPV, Matador, Zhang}, our approach is to regard the equation as a perturbation of the defocusing energy-critical NLS
\begin{equation}\label{E:gopher}
(i\partial_t + \Delta) u = |u|^4 u,
\end{equation}
which was proven to be globally wellposed, first in the radial case and then for general data in the celebrated papers \cite{Bourg,CKSTT:gwp}. Proving well-posedness for a Schr\"odinger equation in three dimensions that contains a quintic nonlinearity requires control over the $\dot{H}_x^1$-norm of the solution. As the energy \eqref{eq:ham2} is not necessarily coercive for $\gamma\in(0,\frac23)$, conservation of the Hamiltonian does not supply the requisite \emph{a priori} bound. To resolve this issue we require that both the energy and the kinetic energy of the data are small when $\gamma\in(0,\frac23).$

The stability of the equilibrium solution $\psi\equiv 1$ to \eqref{eq:cq2} is equivalent to the small-data problem for \eqref{eq:cq3}.  In this direction, there are two natural problems to consider, namely, the initial-value and the final-state problems for \eqref{eq:cq3}.  For the former, the question is whether small and localized initial data lead to solutions that are global and decay as $|t|\to \infty$.  For the latter, the question is whether one can construct a solution that scatters to a prescribed asymptotic state. In this paper we prove two results related to the final-state problem.  We will address the initial-value problem in a forthcoming work.

To fit \eqref{eq:cq3} into the standard framework of dispersive equations it is convenient to diagonalize the equation. Setting
$$
U=\vert\nabla\vert\jb^{-1}\ \text{ and }\ H=\vert\nabla\vert\jb,\ \text{ with }\ \jb:=\sqrt{2\gamma-\Delta}\ \text{ and }\ \vert\nabla\vert=(-\Delta)^{1/2},
$$
we arrive at the following equation for $v:=Vu:=u_1+iUu_2$:
	\begin{equation} \label{eq:cq v}
	\left\{\begin{array}{ll}
	(i\partial_t-H)v=N_v(u):=U\Re [N(u)]+i\Im [N(u)], \\
	v(0)=Vu_0.
	\end{array}\right.
	\end{equation}

Note that $\ulin(t):=V^{-1}e^{-itH}Vu_+$ solves the equation 
\begin{align}
\label{eq:lin}
(i\partial_t+\Delta)\ulin-2\gamma\Re \ulin=0 \qtq{with} \ulin(0)=u_+;
\end{align}
this is the linearization of \eqref{eq:cq3} about $u=0$. 

Our main result in this paper is the following theorem: 

\begin{theorem}\label{thm:wave ops1}
Suppose $\gamma\in[\tfrac23,1)$. For any $u_+\in \Hr$, there exists a global solution $u\in C(\R;\E)$ to \eqref{eq:cq3} such that
\begin{align}\label{E:T:H1}
\lim_{t\to\infty}\|u(t)- \ulin(t)\|_{\dot H^1_x}=0,
\end{align}
where $\ulin(t):=V^{-1}e^{-itH}Vu_+$.  Moreover, we have modified asymptotics in the energy space, in the sense that this same solution $u$ obeys
\begin{align}\label{E:T:E}
\lim_{t\to\infty}d_\E\bigl(u(t), \ulin(t) -\gamma \langle\nabla\rangle^{-2} |\ulin(t)|^2 \bigr)=0.
\end{align}
In the case $\gamma\in(0,\tfrac23)$, both conclusions still hold if additionally $\|u_+\|_{\Hr}$ is sufficiently small. 
\end{theorem}	

\begin{remark}
The hypotheses on $u_+$ are not sufficient to guarantee that $\ulin(t)\in\E$ at any time $t$; correspondingly, one cannot hope to say that $u$ is close to $\ulin$ in the energy space.  Nonetheless, \eqref{E:T:H1} does show that the modification in \eqref{E:T:E} only plays a role at very low frequencies.  Indeed, simple computations show that the modification can be omitted, for example, when $u_+$ is a Schwartz function.   

We do not guarantee uniqueness of the solution $u$ in Theorem~\ref{thm:wave ops1}.  Later, we will show uniqueness within a restricted class of solutions $u$ for suitable scattering states $u_+$; see Theorem~\ref{thm:wave ops2} and Corollary~\ref{cor}  below.
\end{remark}

To give proper context to our work, we need to discuss prior work of Gustafson, Nakanishi, and Tsai \cite{GNT:dd, GNT:2d, GNT:3d} on the Gross--Pitaevskii equation
\begin{equation}\label{E:GP}
\begin{cases}
(i\partial_t+\Delta)\psi=(\vert\psi\vert^2-1)\psi \\
\psi(0)=\psi_0 \\
\lim_{\vert x\vert\to\infty} \psi(t,x) =1. 
\end{cases}
\end{equation}
Note that unlike in \eqref{eq:cq2}, the cubic nonlinearity here is defocusing.  Writing $\psi=1+u$, this equation preserves the energy
\begin{equation}\label{E:GPE}
E_{\text{GP}}(u)=\tfrac12\int_{\R^3}\vert\nabla u\vert^2\,dx+\tfrac14\int_{\R^3} q(u)^2\,dx.
\end{equation}
In contrast to \eqref{eq:ham2}, this energy density is lacking the sign-indefinite $q(u)^3$ term.  Correspondingly, the energy is coercive and the nonlinearity is energy-subcritical.  

The final-state problem for Gross--Pitaevskii was addressed by Gustafson, Nakanishi, and Tsai in \cite{GNT:2d, GNT:3d} in two and three dimensions and in \cite{GNT:dd} in higher dimensions.  They also considered the initial-value problem in dimensions $d\geq3$ in \cite{GNT:dd, GNT:3d}.

The jumping-off point for Theorem~\ref{thm:wave ops1} is an analogous result appearing in \cite{GNT:3d} for the Gross--Pitaevskii equation, which in turn builds on earlier work of Nakanishi \cite{Nak} on the (gauge-invariant) NLS.   Let us discuss a particular result from \cite{Nak}, the one that is most closely connected to the problem studied in this paper: Given $u_+ \in H_x^1(\R^3)$ and $\frac23<p<\frac43$, there is a solution to
\begin{align}\label{GONLS}
(i \partial_t  +\Delta) u = |u|^p u
\end{align}
that obeys $e^{-it\Delta} u(t) \to u_+$ in $H_x^1(\R^3)$.   To prove this, Nakanishi first defines solutions $u^T$ to \eqref{GONLS} with $u^T(T) = e^{iT\Delta} u_+$.  As the problem is mass-subcritical, these solutions are easily seen to be global with uniformly bounded $H^1_x$-norm (even in the focusing case).

By writing \eqref{GONLS} in Duhamel form and exploiting the dispersive estimate \eqref{NLS disp}, it is not difficult to show that for each $\phi\in C^\infty_c(\R^3)$, the collection of functions
\begin{equation}\label{E:equiF}
\bigl\{ t \mapsto \langle \phi, e^{-it\Delta} u^T(t)\rangle  : T\in \R\bigr\}
\end{equation}
forms an equicontinuous family on a compactification $\R\cup\{\pm\infty\}$ of the real line.  In particular, each such function has limiting values as $t\to\pm\infty$.  Applying Arzel\`a--Ascoli and the Cantor diagonal argument ($H_x^1$ is separable), we may find a sequence $T_n\to\infty$ and a function $u^\infty\in L^\infty_t H^1_x$ so that
$$
e^{-it\Delta} u^{T_n}(t) \rightharpoonup e^{-it\Delta} u^\infty(t) \quad\text{weakly in $H^1_x$ for each $t\in \R$.}
$$   
This construction guarantees that $u^\infty$ has two further properties.  First, the function $u^\infty:\R\to H^1_x$ is weakly continuous on $\R\cup\{\pm\infty\}$, that is, when $H_x^1$ is endowed with the weak topology.  Secondly, for any $\phi\in C^\infty_c(\R^3)$,
\begin{equation*}%\label{TestWeak}
\langle \phi,  e^{-it\Delta}  u^{T_n}(t)\rangle \to \langle \phi, e^{-it\Delta}  u^\infty(t)\rangle \quad\text{as $n\to\infty$, uniformly in $t\in\R$.}
\end{equation*}
Using these properties it is elementary to verify that $e^{-it\Delta} u^\infty(t) \rightharpoonup u_+$ as $t\to\infty$.  This leaves two obligations: first one must show that $u^\infty$ is actually a solution to \eqref{GONLS} and secondly, one must upgrade weak convergence to norm convergence.

Due to the $H^1_x$-subcriticality of the nonlinearity, the Rellich--Kondrashov theorem allows one to show that $u^\infty$ is a weak solution to \eqref{GONLS}.  For this problem, weak solutions with values in $H^1_x$ are necessarily strong solutions and so we may conclude that $u^\infty$ is a solution to \eqref{GONLS}.

Lastly, to upgrade weak convergence to strong convergence, one exploits conservation of mass and energy and the Radon--Riesz Theorem.  For example, one may argue as follows:  The quantity
\begin{align}
F(u) := \int_{\R^3} |\nabla u|^2 + \tfrac{2}{p+2} |u|^{p+2} + |u|^2\,dx 
\end{align}
is conserved under the flow \eqref{GONLS}.  Exploiting this, dispersion of the linear flow, and weak lower-semicontinuity of norms, we deduce that
$$
\varlimsup_{t\to\infty} \| e^{-it\Delta} u^\infty(t) \|_{H_x^1}^2 \leq F(u^\infty) \leq \varliminf_{n\to\infty} F(u^{T_n}(0)) = \varliminf_{n\to\infty} F(u^{T_n}(T_n)) = \| u_+ \|_{H^1_x}^2.
$$
Given that $e^{-it\Delta} u^\infty(t) \rightharpoonup  u_+$, we deduce that $e^{-it\Delta} u^\infty(t) \to  u_+$ in $H^1_x$.

In order to adapt this beautiful argument to the Gross--Pitaevskii setting, the authors of \cite{GNT:3d} had to overcome two significant obstacles: First, one needs to make the (conserved) energy \eqref{E:GPE} associated to \eqref{E:GP} play the role of $F$ in the above.  It is far from obvious that this has the requisite convexity.  Secondly, the simple arguments used to prove equicontinuity of the family \eqref{E:equiF} no longer work.  This failure stems from  lower-power terms in the nonlinearity combined with the fact that energy conservation gives poor \emph{a priori} spatial decay of solutions; while it guarantees $q(u)\in L^2_x$, it only yields $u_1\in L_x^3$ and no better than $u_2\in L_x^6$.  This is not sufficient decay to allow direct access to any of the integrable-in-time dispersive estimates obeyed by the propagator.

The key to obtaining equicontinuity of the analogue of the family \eqref{E:equiF} in the Gross--Pitaevskii setting is to exploit certain non-resonant structures in the nonlinearity that allow one to integrate by parts in time.  In implementing this approach, one sees that it is necessary to exhibit such non-resonance in both the quadratic and cubic terms of the nonlinearity.  Such a brute force attack is rather messy.  The burden can be significantly reduced by using test functions whose Fourier support excludes the origin.  We will demonstrate this (primarily expository) improvement over the arguments from \cite{GNT:3d} in the proof of Proposition~\ref{P:weak convergence}.  One particular virtue of this approach is that it makes clear from the start that the argument is inherently completely immune to the poor dispersion manifested by the propagator \eqref{E:matrix prop} at low frequencies.

In \cite{GNT:3d}, Gustafson, Nakanishi, and Tsai exploit the quadratic non-resonant structure in a more elegant way through the use of a normal form transformation
\begin{equation}\label{GNT_NF}
z = \bigl[u_1 + (2-\Delta)^{-1} |u|^2\bigr] + i \sqrt{-\Delta/(2-\Delta)} \,u_2  .
\end{equation}
In this work they also observe (and then utilize) the further non-resonant structure at the cubic level (akin to \eqref{404}).  There is some flexibility in the choice of normal form that witnesses the requisite non-resonance; however, the particular one employed in \cite{GNT:3d} has the dramatic additional benefit of overcoming the first obstacle described above.  The necessary convexity of the energy functional becomes clearer when written in their new variables: With $u$ and $z$ related by \eqref{GNT_NF},
\begin{equation}\label{E:GPEident}
E_{\text{GP}} (u) = \tfrac12 \bigl\| \sqrt{2-\Delta}\, z\|_{L^2_x}^2 + \tfrac14 \bigl\| \sqrt{-\Delta/(2-\Delta)}\, |u|^2 \|_{L^2_x}^2 .
\end{equation}
The virtue of this identity is best understood in the context of \eqref{957}.  Because the right-most term in \eqref{E:GPEident} is non-negative, combining \eqref{E:GPEident} with \eqref{957} yields
$$
\varlimsup_{t\to\infty} \tfrac12 \bigl\| \sqrt{2-\Delta}\, z(t) \|_{L^2_x}^2 = \tfrac12 \bigl\| \sqrt{2-\Delta}\, z_+ \|_{L^2_x}^2,
$$
where $z(t)$ and $z_+$ represent a particular solution and its putative scattering state, both in terms of the normal form variable.  This is just what is needed as input for the Radon--Riesz Theorem.

In order to prove Theorem~\ref{thm:wave ops1} we will need to capitalize on all the ideas introduced in \cite{GNT:3d} to prove the analogous result for the Gross--Pitaevskii equation.  In particular, we will exploit a normal form transformation modeled closely on \eqref{GNT_NF}, namely,
\begin{equation}\label{1st M}
z = M(u) := \bigl[u_1 + \gamma (2\gamma-\Delta)^{-1} |u|^2\bigr] + i \sqrt{-\Delta/(2\gamma-\Delta)} \, u_2 .
\end{equation}

However, several new difficulties arise above and beyond those overcome in \cite{GNT:3d}.  The first group of new difficulties is associated to the presence of energy-critical terms in the nonlinearity.  This already manifests in the proof that \eqref{eq:cq3} admits global solutions for initial data in the energy space $\E$, as discussed earlier in this introduction.

In the Nakanishi argument described above, it was used that weak limits (in the $H^1_x$ topology pointwise in time) of strong solutions to \eqref{GONLS} are themselves strong solutions.  In the subcritical case, one sees relatively easily that such limits are weak solutions (via Rellich--Kondrashov) and can then exploit earlier work (cf. \cite[Ch. 3--4]{Cazenave}) showing that weak solutions are strong solutions.  In particular, solutions converging weakly to zero (in $H^1_x$) by concentrating will actually converge to zero in the space-time norms used to construct such solutions.  In a similar way, we see that increasingly concentrated parts of a solution (which will drop out under taking a weak limit) do not affect parts of the solution living at unit scale.

In the critical case, the norm is invariant under scaling.  Correspondingly, highly concentrated parts of a solution may have large norm and so, naively at least, have a non-trivial effect on the remainder of the solution.  Thus, it is not clear that weak limits of solutions should even be weak solutions!  The key to escaping this nightmare is to show that two parts of a solution have little effect on one another if they live at widely separated scales.  Naturally, even decomposing a solution into its constituent parts is a highly non-trivial business; however, such technology has been rather extensively developed due to its role in the induction on energy paradigm, which already underlies the proof of the global well-posedness of \eqref{E:gopher} in \cite{CKSTT:gwp}.

Before tackling \eqref{eq:cq3}, a reasonable first step is to prove that weak limits of solutions are themselves solutions in the case of the equation \eqref{E:gopher}.  This is a very natural problem, which does not seem to have been studied before.  As it turns out, it is possible to verify this statement with relatively little labour by harnessing the full power of the concentration compactness ideas already developed in that setting.  Specifically, one starts with a nonlinear profile decomposition and then further exploits some of the decoupling ideas used in its proof.  In this paper, we will implement this strategy in the setting of \eqref{eq:cq3}; this is ample guidance for anyone seeking to reconstruct our argument for \eqref{E:gopher}.

We begin by developing a linear profile decomposition adapted to \eqref{eq:cq3}; see Proposition~\ref{P:LPD}.  Despite the fact that the linear equation underlying \eqref{eq:cq3} differs from that underlying \eqref{E:gopher}, we are able to adapt the profile decomposition for the linear Schr\"odinger equation to our setting, rather than proceeding \emph{ab initio}.  To develop a nonlinear profile decomposition, we need to construct solutions to \eqref{eq:cq3} associated to each linear profile.  For profiles living at unit scale, existence of these solutions (and all requisite bounds) follow from Theorem~\ref{thm:gwp}.  Profiles whose characteristic length scale diverges can be approximated by linear solutions on bounded time intervals and so require no special attention.  However, highly concentrated profiles require independent treatment; this is the content of Propostion~\ref{P:embedding}.  There are two subtle points here: (i) Such profiles are merely $\dot H^1_x$ and so do not have finite energy. (ii) The characteristic time scale associated to such profiles is very short; thus, understanding such solutions even on a bounded interval essentially requires an understanding of their infinite time behavior.

The nonlinear profile decomposition posits that the nonlinear evolution of the initial data can be approximated by the sum of the nonlinear evolutions of its constituent profiles.  This is verified by demonstrating decoupling of the profiles inside the nonlinearity (see Lemma~\ref{L:approx}) and exploiting a suitable stability theory for our equation (see Proposition~\ref{P:Stab}).  The latter requires certain \emph{a priori} bounds, which are shown to hold in Lemma~\ref{L:STB}.  Once it is known that the nonlinear profile decomposition faithfully represents the true solution, it is relatively elementary to complete the proof of Theorem~\ref{T:weak}.

This completes our discussion of the new difficulties (relative to \cite{GNT:3d}) associated to the presence of energy-critical nonlinear terms.  The second main group of difficulties stems from the shape of the energy functional.  The lack of coercivity when $\gamma\in (0,\frac23)$ was discussed already as an obstacle to proving global well-posedness.

As also discussed above, convexity of the energy functional plays a key role in upgrading weak convergence to strong convergence, via an argument of Radon--Riesz type.
The analogue of \eqref{E:GPEident} for our equation is as follows:  For $z=M(u)$ as in \eqref{1st M},
\begin{equation}\label{E:CQEident}
E(u) = \tfrac12 \| \jb z \|_{L^2_x}^2 + \tfrac\gamma4 \| U |u|^2 \|_{L^2_x}^2 + \int \tfrac16 q(u)^3\,dx  .
\end{equation}
Unlike its analogue \eqref{E:GPEident}, this does not yield an inequality between the energy and the $H^1_x$-norm of $z$.  Indeed, the leading order correction is the sign indefinite term $\frac43 \int (u_1)^3\,dx$.  Correspondingly, we will need to be concerned with the structure of our solution $u^\infty(t)$ as $t\to\infty$ to ensure that it does not contain surplus energy beyond that needed for its (putative) scattering state.  Recall that $u^\infty(t)$ is merely constructed as a weak limit of solutions $u^{T_n}(t)$ defined by their values at $t=T_n$, which gives very little \emph{a priori} information on its structure.  The resolution of this dilemma is to prove a form of energy decoupling between the part of the solution matching the scattering state and any residual part (cf. Lemma~\ref{L:energy1}).   Ultimately, this energy decoupling shows that any residual part of the solution must converge to zero in norm, which then obviates any explicit implementation of the Radon--Riesz-style argument.

We turn now to our second result in this paper, which guarantees a degree of uniqueness of the nonlinear solution with prescribed scattering state.  Specifically, for scattering states with good linear decay, we can guarantee that there is only one nonlinear solution scattering to it with comparable decay.  The decay of such solutions will be measured in the following norm:
$$
\|u\|_{X_T}:=\sup_{t\geq T} \, t^{\frac12}\|u(t)\|_{H_x^{1,3}(\R^3)}.
$$

\begin{theorem}\label{thm:wave ops2}
Fix $\gamma\in(0,1)$. There exists $\eta>0$ so that if $u_+\in \Hr$ satisfies 
		\begin{equation}
		\label{eq:scattering smallness condition}
		\|V^{-1}e^{-itH}Vu_+\|_{X_1} \leq \eta,
		\end{equation}
then there exists a global solution $u\in C(\R; \E)$ to \eqref{eq:cq3} such that
		\begin{equation}
		\label{eq:scattering1}
		\lim_{t\to\infty}\|u(t)-V^{-1}e^{-itH}Vu_+\|_{\Hr}=0.
		\end{equation}
Moreover $u$ is unique in the class of solutions with $\|u\|_{X_T} \leq 4\eta$ for some $T\geq 1$.
\end{theorem}

\begin{remark}
The proof of this theorem gives a quantitative rate in \eqref{eq:scattering1}, namely,
\begin{equation}\label{E:scattering11}
		\|u(t)-V^{-1}e^{-itH}Vu_+\|_{\Hr} \lesssim t^{-1/4}.
\end{equation}
\end{remark}

\begin{remark} Writing $\ulin(t)=V^{-1}e^{-itH}Vu_+$, we note that $u_+\in \Hr$ and $\|\ulin\|_{X_1}<\infty$ guarantee that $\ulin$ is uniformly bounded in the energy space $\E$ for $t\geq 1$. 
\end{remark}

Finally, we observe that we can guarantee the smallness condition \eqref{eq:scattering smallness condition} by assuming control over weighted norms.

\begin{corollary}\label{cor} Let $\gamma\in(0,1)$ and $u_+\in \Hr $. If	
	$$\|\langle x\rangle^{\frac12+}\langle\nabla\rangle u_+\|_{L_x^2}
	+\|\langle x\rangle^{\frac43+}\langle\nabla\rangle^{\frac56}\Re u_+\|_{L_x^2}$$
is sufficiently small, then there exists a global solution $u\in C(\R;\E)$ to \eqref{eq:cq3} such that \eqref{eq:scattering1} holds.
\end{corollary}	

We prove Theorem~\ref{thm:wave ops2} and Corollary~\ref{cor} in Section~\ref{section:fs2}. The proof, which relies primarily on dispersive and Strichartz estimates, consists of a contraction mapping argument that simultaneously solves the requisite PDE for $z=M(u)$ and inverts the normal form transformation.  The argument differs little from that used to prove Theorem~1.1 in \cite{GNT:2d}.

The rest of the paper is organized as follows. In Section~\ref{section:notation} we set some notation and collect some useful lemmas. In Section~\ref{section:gwp} we prove global well-posedness and unconditional uniqueness in the energy space for \eqref{eq:cq3}, as well as a stability result. In Section~\ref{section:weak} we prove a well-posedness result in the weak topology, namely, Theorem~\ref{T:weak}. In Section~\ref{section:normal form} we discuss the normal form transformation \eqref{1st M}. In Section~\ref{section:fs1} we prove Theorem~\ref{thm:wave ops1}, and finally in Section~\ref{section:fs2} we prove Theorem~\ref{thm:wave ops2} and Corollary~\ref{cor}.

\subsection*{Acknowledgements} R.~K. was supported by NSF grant DMS-1265868. J.~M. was supported by DMS-1400706. M.~V. was supported by NSF grant DMS-1161396.  We are indebted to the Hausdorff Institute of Mathematics, which hosted us during our work on this project.

								%%%%%%%%%%%%%%%%%%%%%%%%
					\section{Notation and useful lemmas}\label{section:notation}
\subsection{Some notation} 					
We write $A\lesssim B$ or $A= O(B)$ to indicate that $A\leq CB$ for some constant $C>0$. Dependence of implicit constants on various parameters will be indicated with subscripts. For example $A\lesssim_\varphi B$ means that $A\leq CB$ for some $C=C(\varphi)$. The dependence of implicit constants on the parameter $\gamma$ defined in \eqref{def:gamma} will not be explicitly indicated. We write $A\sim B$ if $A\lesssim B$ and $B\lesssim A$. We write $A\ll B$ if $A\leq cB$ for some small $c>0$. 

We write a complex-valued function $u$ as $u=u_1+iu_2$. When $X$ is a monomial, we use the notation $\text{\O}(X)$ to denote a finite linear combination of products of the factors of $X$, where Mikhlin multipliers (e.g. Littlewood--Paley projections) and/or complex conjugation may be additionally applied in each factor.  We extend $\text{\O}$ to polynomials via $\text{\O}(X+Y)=\text{\O}(X)+\text{\O}(Y)$.

For a time interval $I$ we write $L_t^q L_x^r(I\times\R^3)$ for the Banach space of functions $u:I\times\R^3\to\C$ equipped with the norm
	$$\|u\|_{L_t^qL_x^r(I\times\R^3)}=\left(\int_I \|u(t)\|_{L_x^r(\R^3)}^q\,dt\right)^{1/q},$$
with the usual adjustments when $q$ or $r$ is infinity. If $q=r$ we write $L_t^qL_x^q=L_{t,x}^q$. We often abbreviate $\|u\|_{L_t^qL_x^r(I\times\R^3)}=\|u\|_{L_t^qL_x^r}$ and $\|u\|_{L_x^r(\R^3)}=\|u\|_{L_x^r}.$ We also write $C(I;X)$ to denote the space of continuous functions on $I$ taking values in $X$. 

We use the following convention for the Fourier transform on $\R^3$:
$$
\wh{f}(\xi)=(2\pi)^{-3/2}\int_{\R^3} e^{-ix\cdot\xi} f(x)\,dx \qtq{so that} f(x)=(2\pi)^{-3/2}\int_{\R^3} e^{ix\cdot\xi} \hat f(\xi)\,d\xi.
$$

The fractional differential operator $\vert\nabla\vert^s$ is defined by $\wh{\vert\nabla\vert^s f}(\xi)=\vert\xi\vert^s\wh{f}(\xi).$ We will also make use of the following Fourier multiplier operators (and powers thereof):
	\begin{align*}
	\langle\xi\rangle=\sqrt{2\gamma+\vert\xi\vert^2},\quad &\jb=\sqrt{2\gamma-\Delta},
	\\
	U(\xi)=\sqrt{\vert\xi\vert^2(2\gamma+\vert\xi\vert^2)^{-1}},\quad & U=\sqrt{(-\Delta)(2\gamma-\Delta)^{-1}},
	\\
	H(\xi)=\sqrt{\vert\xi\vert^2(2\gamma+\vert\xi\vert^2)},\quad& H=\sqrt{(-\Delta)(2\gamma-\Delta)}.
	\end{align*}

Fix $\gamma\in(0,1)$ as in \eqref{def:gamma}. We define homogeneous and inhomogeneous Sobolev norms $\dot H^{s,r}_x$ and $H_x^{s,r}$ as the completion of Schwartz functions under the norms
$$
\|f\|_{\dot H^{s,r}_x}:= \|(-\Delta)^{s/2} f\|_{L^r_x} \quad\text{and}\quad 
 \|f\|_{H^{s,r}_x}:= \|(2\gamma-\Delta)^{s/2} f\|_{L^r_x},
$$
respectively.  When $r=2$ we abbreviate $\dot H^{s,2}_x=\dot H^s_x$ and $H^{s,2}_x=H^s_x$. Note that this definition of the $H_x^{s}$-norm is equivalent (up to constants depending on $\gamma$) to the standard one, which uses the operator $(1-\Delta)^{s/2}$. 
	
\subsection{Basic harmonic analysis} We employ the standard Littlewood--Paley theory. Let $\phi$ be a radial bump function supported in $\{\vert\xi\vert\leq\tfrac{11}{10}\}$ and equal to one on the unit ball. For $N\in2^{\mathbb{Z}}$ we define the Littlewood--Paley projections 
\begin{align*}
\wh{P_{\leq N}u}(\xi)=\phi(\tfrac{\xi}{N})\wh{u}(\xi), \quad \wh{P_Nu}(\xi)=[\phi(\tfrac{\xi}{N})-\phi(\tfrac{\xi}{2N})]\wh{u}(\xi), \quad\text{and} \quad P_{>N}= Id - P_{\leq N}.
\end{align*}

These operators commute with all other Fourier multiplier operators. They are self-adjoint and bounded on every $L_x^p$ and $H_x^s$ space for $1\leq p\leq\infty$ and $s\geq 0$. We write $\PLo=P_{\leq 1}$ and $\PHi  =P_{> 1}$.

The Littlewood--Paley projections obey the following standard estimates.

\begin{lemma}[Bernstein estimates] For $1\leq r\leq q\leq\infty$ and $s\geq 0$ we have
\begin{align*}
\|\vert\nabla\vert^s P_{\leq N}u\|_{L_x^r(\R^3)}
&\lesssim N^s \|P_{\leq N}u\|_{L_x^r(\R^3)}\\
\|P_{>N}u\|_{L_x^r(\R^3)}&\lesssim N^{-s}
\|\vert\nabla\vert^s P_{>N}u\|_{L_x^r(\R^3)},\\
\|P_{\leq N}u\|_{L_x^q(\R^3)}
&\lesssim N^{\frac3r-\frac3q} \|P_{\leq N}u\|_{L_x^r(\R^3)}.
\end{align*} 	
\end{lemma}

We will need the following fractional chain rule from \cite{ChrWei}. 

\begin{lemma}[Fractional chain rule, \cite{ChrWei}] Suppose $G\in C^1(\C)$ and $s\in(0,1]$. Let $1<r,r_2<\infty$ and $1<r_1\leq\infty$ satisfy $\tfrac{1}{r_1}+\tfrac{1}{r_2}=\tfrac{1}{r}$. Then
	$$\|\vert\nabla\vert^s G(u)\|_{L_x^r}\lesssim \|G'(u)\|_{L_x^{r_1}}\|\vert\nabla\vert^s u\|_{L_x^{r_2}}.$$ 
\end{lemma}

We will also need a result of Coifman--Meyer \cite{CoiMey, CoiMey2} concerning bilinear Fourier multipliers. For a real-valued function $B(\xi_1,\xi_2)$ we define the operator $B[f,g]$ via
	\begin{equation}
	\label{eq:bilin}
	\wh{B[f,g]}(\xi):=(2\pi)^{3/2}\int_{\R^3} B(\eta,\xi-\eta)\wh{f}(\eta)\wh{g}(\xi-\eta)\,d\eta.
	\end{equation}

\begin{lemma}[Coifman--Meyer bilinear estimate, \cite{CoiMey, CoiMey2}] \label{L:CM} If the symbol $B(\xi_1,\xi_2)$ satisfies
	$$\vert\partial_{\xi_1}^\alpha\partial_{\xi_2}^\beta B(\xi_1,\xi_2)\vert\lesssim_{\alpha,\beta}(\vert\xi_1\vert+\vert\xi_2\vert)^{-(\vert\alpha\vert+\vert\beta\vert)}$$ 
for all multi-indices $\alpha,\beta$ up to sufficiently high order, then
	$$\|B[f,g]\|_{L_x^r}\lesssim\|f\|_{L_x^{r_1}}\|g\|_{L_x^{r_2}}$$
for all $1< r<\infty$ and $1<r_1,r_2<\infty$ satisfying $\tfrac{1}{r}=\tfrac{1}{r_1}+\tfrac{1}{r_2}.$ 
\end{lemma}

\subsection{Linear estimates} We record here the dispersive and Strichartz estimates for the propagators $e^{it\Delta}$ and $e^{-itH}$.

As is well known, the linear Schr\"odinger propagator in three space dimensions can be written as
	$$[e^{it\Delta}f](x)=(4\pi i t)^{-\frac32}\int_{\R^3} e^{\frac{i\vert x-y\vert^2}{4t}}f(y)\,dy$$
for $t\neq 0$.  This yields the dispersive estimates
\begin{equation}\label{NLS disp}
\|e^{it\Delta} f\|_{L_x^r(\R^3)}\lesssim \vert t\vert^{-(\frac32-\frac3r)}\|f\|_{L_x^{r'}(\R^3)}
\end{equation}
for $t\neq 0$, where $2\leq r\leq\infty$ and $\tfrac{1}{r}+\tfrac{1}{r'}=1$. This estimate can be used to prove the standard Strichartz estimates for $e^{it\Delta}$.  We state the result we need in three space dimensions.

\begin{proposition}[Strichartz estimates for $e^{it\Delta}$, \cite{GinVel, KeeTao, Str}]
For a space-time slab $I\times\R^3$ and $2\leq q,\tilde{q}\leq\infty$ with $\tfrac{2}{q}+\tfrac{3}{r}=\tfrac{2}{\tilde{q}}+\tfrac{3}{\tilde{r}}=\tfrac{3}{2}$, we have
\begin{align*}
    \Bigl\|e^{it\Delta}\varphi + \int_0^t e^{i(t-s)\Delta}F(s)\,ds\Bigr\|_{L_t^q L_x^r(I\times\R^3)}
    &\lesssim\|\varphi\|_{L_x^2} + \|F\|_{L_t^{\tilde{q}'}L_x^{\tilde{r}'}(I\times\R^3)}.
\end{align*}
\end{proposition}

Using stationary phase one can prove a similar dispersive estimate for $e^{-itH}$ (see \cite{GNT:dd, GNT:3d}). In fact, there is a small gain at low frequencies compared to the estimates for the linear Schr\"odinger propagator; while the dispersion relation for this propagator has less curvature in the radial direction than that for Schr\"odinger, this is more than compensated for by the increased curvature in the angular directions.

\begin{proposition}[Estimates for $e^{-itH}$, \cite{GNT:dd, GNT:3d}] \label{lemma:dispersive} For $2\leq r\leq\infty$ we have
    \begin{equation}
    \label{eq:dispersive}
    \|e^{-itH}f\|_{L_x^r(\R^3)}\lesssim \vert t\vert^{-(\frac32-\frac3r)}\|U^{\frac12-\frac1r}f\|_{L_x^{r'}(\R^3)}
    \end{equation}
for $t\neq 0$.  In particular, for a space-time slab $I\times\R^3$ and $2\leq q,\tilde{q}\leq\infty$ with $\tfrac{2}{q}+\tfrac{3}{r}=\tfrac{2}{\tilde{q}}+\tfrac{3}{\tilde{r}}=\tfrac{3}{2}$, we have
\begin{align*}
    \Bigl\| e^{-itH}\varphi + \int_0^t e^{-i(t-s)H}F(s)\,ds\Bigr\|_{L_t^q L_x^r(I\times\R^3)}
    &\lesssim \|\varphi\|_{L_x^2} + \|F\|_{L_t^{\tilde{q}'}L_x^{\tilde{r}'}(I\times\R^3)}.
\end{align*}
\end{proposition}

For an interval $I$ and $s\geq 0$ we define the Strichartz norm by
$$\|u\|_{\dot{S}^s(I)}=\sup\big\{\|\vert\nabla\vert^s u\|_{L_t^qL_x^r(I\times\R^3)}:\ 
	2\leq q\leq\infty,\quad \tfrac{2}{q}+\tfrac{3}{r}=\tfrac{3}{2}\big\}.$$	
The Strichartz space $\dot{S}^s(I)$ is then defined to be the closure of test functions under this norm.  We let $\dot{N}^s(I)$ denote the corresponding dual Strichartz space.

In several places it will be more convenient to work with equation \eqref{eq:cq3} rather than the diagonalized equation \eqref{eq:cq v}.  The linear propagator associated with \eqref{eq:cq3} takes the following form:
\begin{equation}\label{E:matrix prop}
V^{-1}e^{-itH} V \begin{bmatrix} f_1\\[1ex] f_2\end{bmatrix} =\begin{bmatrix} \cos(tH) &  U\sin(tH) \\[1ex]
            -U^{-1}\sin(tH) & \cos(tH) \end{bmatrix}\begin{bmatrix} f_1\\[1ex] f_2\end{bmatrix},
\end{equation}
for any function $f=f_1+if_2$. We will make use of the following Strichartz estimates for this propagator:

\begin{lemma}\label{L:VStrichartz} Fix $T>0$.  Given $2< q,\tilde{q}\leq\infty$ with $\tfrac{2}{q}+\tfrac{3}{r}=\tfrac{2}{\tilde{q}}+\tfrac{3}{\tilde{r}}=\tfrac{3}{2}$, we have 
\begin{align}
\Bigl\| V^{-1}e^{-itH}V\varphi + \int_0^t V^{-1}e^{-i(t-s)H}VF(s)\,ds\Bigr\|_{L_t^q L_x^r}
    &\lesssim_T \|\varphi\|_{L_x^2} + \|F\|_{L_t^{\tilde{q}'}L_x^{\tilde{r}'}},
\end{align}
where all space-time norms are over $[-T,T]\times\R^3$.
\end{lemma}

\begin{proof}
As we are excluding the endpoint, it suffices (via a $TT^*$ argument) to prove the result when $F\equiv 0$; moreover, it clearly suffices to consider each entry in the matrix \eqref{E:matrix prop} separately.  In view of the boundedness of $U$, three out of four of these matrix elements obey the same Strichartz estimates as $e^{-itH}$; see Proposition~\ref{lemma:dispersive}.  As $\PHi U^{-1}$ is also bounded, we need only prove Strichartz estimates for $\PLo U^{-1}\sin(tH)$.  However, this is easily done via H\"older and Bernstein's inequality:
\begin{align}
\| \PLo U^{-1}\sin(tH) \varphi\|_{L_t^q L_x^r([-T,T]\times\R^3)} &\lesssim T^{\frac1q} \| \PLo U^{-1}\sin(tH) \varphi\|_{L_t^\infty L_x^2([-T,T]\times\R^3)} \notag\\
&\lesssim T^{1+\frac1q} \|\varphi\|_{L^2_x(\R^3)}. \label{E:2:30}
\end{align}
This completes the proof of the lemma.
\end{proof}

At high frequencies, the operator $e^{-itH}$ closely resembles the Schr\"odinger propagator (on bounded time intervals); specifically, we have
\begin{align}\label{HvsDelta}
\sqrt{|\xi|^2 (2\gamma+|\xi|^2)} = |\xi|^2 + \gamma + m(\xi) \qtq{with} |m(\xi)|\lesssim \langle\xi\rangle^{-2}.
\end{align}
Indeed, it is not difficult to verify that $m(\xi)$ defines a Mikhlin multiplier.  This observation will play a key role in our treatment of highly concentrated profiles in Section~\ref{section:weak}.  For the moment, however,
we simply use it to obtain a crude local smoothing estimate.

\begin{lemma}[Local smoothing]\label{L:VSmoothing} Given $T>0$ and $R>0$,
\begin{align}
\bigl\| |\nabla|^{\frac12}  V^{-1}e^{-itH}V \varphi \bigr\|_{L^2_{t,x}(\{|t|\leq T\}\times\{|x|\leq R\})} &\lesssim_{R,T} \|\varphi\|_{L_x^2}.
\end{align}
\end{lemma}

\begin{proof}
We treat high and low frequencies separately.  In the low-frequency regime, we exploit \eqref{E:matrix prop} and argue as in \eqref{E:2:30} to deduce that
$$
\bigl\| |\nabla|^{\frac12} \PLo V^{-1}e^{-itH}V \varphi \bigr\|_{L^2_{t,x}(\{|t|\leq T\}\times\{|x|\leq R\})} \lesssim T^{\frac12}(1+T) \|\varphi\|_{L_x^2}. 
$$
In the high-frequency regime, we can use the usual local smoothing estimate for the Schr\"odinger equation together with
$$
\bigl\| |\nabla|^{\frac12} \PHi  V^{-1}[e^{-itH} - e^{-it(\gamma-\Delta)} ] V\varphi \bigr\|_{L^\infty_t L^2_{x}(\{|t|\leq T\}\times\R^3)} \lesssim T \|\varphi\|_{L_x^2},
$$
which follows from \eqref{HvsDelta}.
\end{proof}

In practice, we will use the following corollary.

\begin{corollary}\label{C:smoothing} Let $K$ be a compact subset of $I\times\R^3$ for some interval $I\subset\R$. Then the following estimates hold:
\begin{align*}
\|\nabla e^{it\Delta}f\|_{L_{t,x}^2(K)}&\lesssim_K \|e^{it\Delta}f\|_{L_{t,x}^{10}(I\times\R^3)}^{\frac13}\| f\|_{\dot H_x^1}^{\frac23},\\
\|\nabla V^{-1}e^{-itH}Vf\|_{L_{t,x}^2(K)}&\lesssim_K \|V^{-1}e^{-itH}Vf\|_{L_{t,x}^{10}(I\times\R^3)}^{\frac13}\| f\|_{\dot H_x^1}^{\frac23}.
\end{align*}
\end{corollary}

\begin{proof} Fix $N>0$.  By the Bernstein and H\"older inequalities,
$$
\|\nabla P_{\leq N}e^{it\Delta}f\|_{L_{t,x}^2(K)}\lesssim_K N\|e^{it\Delta} f\|_{L_{t,x}^{10}(I\times\R^3)}.
$$
By the local smoothing estimate for $e^{it\Delta}$ and Bernstein, we also have
$$
\|\nabla P_{>N}e^{it\Delta}f\|_{L_{t,x}^2(K)}\lesssim_K \|\vert\nabla\vert^{\frac12}P_{>N}f\|_{L_x^2}
\lesssim_K N^{-\frac12}\|\nabla f\|_{L_x^2}.
$$
Optimizing in the choice of $N$ yields the first estimate. 

To obtain the second estimate one argues in exactly the same way, making use of Lemma~\ref{L:VSmoothing}. \end{proof}

						%%%%%%%%%%%%%%%%%%%%%%%%
						\section{Global well-posedness in the energy space}
						\label{section:gwp}
						%%%%%%%%%%%%%%%%%%%%%%%%

In this section we discuss the well-posedness of \eqref{eq:cq3} in the energy space.   We begin by justifying the name `energy space' given to the set $\E$ defined in \eqref{eq:energyspace}.  Recall from the introduction that if $u\in\E$, then $|E(u)|<\infty$. The following two lemmas prove that if the energy of $u$ is finite, then $u\in \E$; when $\gamma\in(0,\frac23)$, this requires an additional smallness condition.

\begin{lemma}\label{lemma:coercive1}
If $\gamma\in(\tfrac23,1)$ and $E(u)<\infty$, then $u\in\E$ with $\|u\|_{\E}^2\lesssim E(u).$

If $\gamma=\tfrac23$ and $E(u)<\infty$, then $u\in\E$ with
	\begin{equation}
	\nonumber
	\|\nabla u\|_{L_x^2}^2\lesssim E(u)\qtq{and} \|q\|_{L_x^2}^2\lesssim E(u)+[E(u)]^3.
	\end{equation}
\end{lemma}

\begin{proof}
When $\gamma>\tfrac23$ we use the fact that $q\geq-1$ in \eqref{eq:ham2} to write
	$$E(u)\geq\tfrac12\int\vert\nabla u\vert^2\,dx+\tfrac\gamma4\int \big(1-\tfrac{2}{3\gamma}\big)q^2\,dx,$$
which immediately implies the result.

We now turn to the case when $\gamma=\tfrac23$.  In this case, the energy takes the form
	$$E(u)=\tfrac12\int\vert\nabla u\vert^2\,dx+\tfrac16\int q^2(q+1)\,dx.$$ 
As $q\geq-1$, we have $q^2(q+1)\geq 0$. Thus $u\in\dot{H}_x^1(\R^3)$ and $\|\nabla u\|_{L_x^2}^2\lesssim E(u).$  

To estimate the $L^2_x$-norm of $q$ we note that
	$$\int_{\{q\geq -\frac12\}}q^2\,dx\leq2\int q^2(q+1)\,dx\lesssim E(u).$$ 
On the other hand, if $q<-\frac12$ then $|u_1|>\frac14$; thus, by Chebyshev's inequality and Sobolev embedding,
	$$\int_{\{q<-\frac12\}} q^2\,dx\leq 4^{6}\|u_1\|_{L_x^6}^{6}\lesssim\|\nabla u \|_{L_x^2}^6\lesssim [E(u)]^3.$$ 
	
This completes the proof of Lemma~\ref{lemma:coercive1}.
\end{proof} 

We next consider the full range $\gamma\in(0,1)$. In this case, we can guarantee coercivity of the energy under an appropriate smallness assumption.

\begin{lemma}\label{lemma:coercive2} For any $\gamma\in(0,1)$ there exists $\delta_\gamma>0$ so that the following hold:
\begin{SL}
\item If $E(u)<\infty$ and $\|\nabla u_1\|_{L_x^2}^2\leq\delta_\gamma,$ then $u\in\E$ with $\|u\|_{\E}^2\lesssim E(u)$.
\item For any ball $B$,
\begin{align}\label{EoutsideB}
\|\nabla u_1\|_{L_x^2(\R^3)}^2\leq\delta_\gamma \implies \int_{B^c} \tfrac12|\nabla u|^2 + \tfrac\gamma4 q^2 + \tfrac16 q^3\,dx \geq 0.
\end{align}
\item If $u:I\times\R^3\to \C$ is a solution to \eqref{eq:cq3} with $E(u)\leq \frac14 \delta_\gamma$ and $\|\nabla \Re u(t_0)\|_{L_x^2}^2\leq\delta_\gamma$ for some $t_0\in I$, then
$$
\|\nabla \Re u\|_{L_t^\infty L_x^2(I\times\R^3)}^2\leq\delta_\gamma \qtq{and} \|u\|_{L^\infty (I;\E)}^2\lesssim E(u).
$$
\end{SL}
\end{lemma}

\begin{proof}
We begin by writing
	\begin{align*}E(u)&=\int\tfrac12\vert\nabla u\vert^2+\tfrac{\gamma}{8}q^2+\tfrac16q^2\big(q+\tfrac{3\gamma}{4}\big)\,dx
	\\ &\geq\int\tfrac12\vert\nabla u\vert^2+\tfrac{\gamma}{8}q^2\,dx+\int_{\{ q<-\frac{3\gamma}{4}\}}\tfrac16 q^2\big(q+\tfrac{3\gamma}{4}\big)\,dx.
	\end{align*}
For $q< -\frac{3\gamma}4$ we have $|u_1|>\frac{3\gamma}8$.  Thus, by Chebyshev's inequality and Sobolev embedding, we have
	$$\vert\{q<-\tfrac{3\gamma}{4}\}\vert\leq(\tfrac{8}{3\gamma})^6\|u_1\|_{L_x^6}^6\lesssim \gamma^{-6}\|\nabla u_1\|_{L_x^2}^6.$$ 
Recalling that $q\geq -1$, we find that for $\|\nabla u_1\|_{L_x^2}^2\ll\gamma^{3/2}$ we have
	$$\bigg\vert\int_{\{q<-\frac{3\gamma}{4}\}}\tfrac16q^2\big(q+\tfrac{3\gamma}{4}\big)\,dx\bigg\vert
	\lesssim \gamma^{-6}\|\nabla u_1\|_{L_x^2}^6\leq\tfrac14\|\nabla u_1\|_{L_x^2}^2.$$
Thus
$$
E(u)\geq\int\tfrac 14\vert\nabla u\vert^2+\tfrac{\gamma}{8}q^2\,dx,
$$
which yields conclusion (i) of the lemma.  Claim (iii) also follows from this and a continuity argument. 

To obtain (ii), we repeat the argument above, using the fact that Sobolev embedding holds in the exterior of any ball $B$.
\end{proof}

We next turn to the question of global well-posedness for \eqref{eq:cq3} with initial data $u_0\in\E$. From the lemmas above we see that $u(t)\in\E$ and $\|\nabla u(t)\|_{L_x^2}^2\lesssim E(u_0)$ for all times of existence, whenever $(1)$ $\gamma\in[\tfrac23,1)$ or $(2)$ $\gamma\in(0,\tfrac23)$ and $E(u_0)$ and $\|\nabla\Re u_0\|_{L_x^2}$ are sufficiently small.  This \emph{a priori} bound on $\|\nabla u(t)\|_{L_x^2}$ allows us to treat \eqref{eq:cq3} as a perturbation of the defocusing energy-critical NLS, which was proven to be globally wellposed with finite space-time bounds in \cite{CKSTT:gwp}.  See also \cite{KOPV, Matador} for similar perturbative arguments.

\begin{theorem}[Global well-posedness and unconditional uniqueness]\label{thm:gwp} 

For $\\ \gamma\in[\tfrac23,1)$ and  $u_0\in\E$, there exists a unique global solution $u\in C(\R;\E)$ to \eqref{eq:cq3}. 

For $\gamma\in(0,\frac23)$, if $u_0\in\E$ satisfies $\|\nabla\Re u_0\|_{L_x^2}^2\leq\delta_\gamma$ and $E(u_0)\leq\tfrac14\delta_\gamma$, then there exists a unique global solution $u\in C(\R;\E)$ to \eqref{eq:cq3}.  Here $\delta_\gamma$ is as in Lemma~\ref{lemma:coercive2}.

In both cases the solution remains uniformly bounded in $\E$ and for any $T>0$,
$$
\|u\|_{\dot S^1([-T,T])}\lesssim_T 1.
$$
\end{theorem} 

\begin{remark}\label{R:unique}
When $\gamma\in(0,\tfrac23)$, smallness of the initial data is only exploited to prove global existence; the proof we present below guarantees uniqueness of any solution in $C(I;\E)$ on any time interval $I\subseteq \R$.
\end{remark}

\begin{proof}
As mentioned above, Lemmas~\ref{lemma:coercive1} and \ref{lemma:coercive2} imply that under the hypotheses of Theorem~\ref{thm:gwp} we have 
$\|\nabla u(t)\|_{L_x^2}^2\lesssim E(u_0)$ for all times $t$ of existence.  This allows us to treat \eqref{eq:cq3} as a perturbation of the defocusing energy-critical NLS.  Indeed, we may rewrite equation \eqref{eq:cq3} as follows:
$$
(i\partial_t+\Delta)u= |u|^4u + \mathcal R(u),
$$
where $\mathcal R(u)= 2\gamma \Re u + \sum_{j=2}^4N_j(u)$.  Observing that the `error' $\mathcal R(u)$ is energy-subcritical, one may argue as in \cite[Section~4.2]{KOPV} to construct a global solution $u\in C(\R;\E)\cap L^{10}_t \dot H^{1,\frac{30}{13}}_x(\R\times\R^3)$ to \eqref{eq:cq3}.  A key ingredient in this argument is the main result in \cite{CKSTT:gwp}, which guarantees that the defocusing energy-critical NLS is globally wellposed with finite $L^{10}_t \dot H^{1,\frac{30}{13}}_x(\R\times\R^3)$ norm.  We omit the details of this argument. Instead, we present the proof of uniqueness of solutions in the energy space, because the choice of energy space in this paper does not allow for a direct implementation of the methods in \cite[Section~4.3]{KOPV}.
					
Fix a compact time interval $I=[0,\tau]$ with $\tau>0$ small. Let $u\in C(\R;\E)\cap L^{10}_t \dot H^{1,\frac{30}{13}}_x(\R\times\R^3)$ be the solution to \eqref{eq:cq3} constructed via the perturbative argument described above.  Suppose $\tilde{u}\in C(I;\E)$ is another solution such that $\tilde{u}(0)=u(0).$ We wish to show that $u=\tilde{u}$ almost everywhere on $I\times\R^3$. 

To this end, we define $w=\tilde{u}-u$ and let $0<\eta<1$ be a small parameter to be determined below. As $w(0)=0$ and $w\in C(I;\dot{H}_x^1)$, we can choose $\tau$ small enough so that
\begin{align}\label{small1}
\|w\|_{L_t^\infty\dot{H}_x^1(I\times\R^3)}\leq\eta.
\end{align}
As $\nabla u\in L_t^{10}L_x^{\frac{30}{13}}(I\times\R^3)$, we may also use Sobolev embedding and choose $\tau$ possibly even smaller to guarantee that
\begin{align}\label{small2}
\|u\|_{L_{t,x}^{10}(I\times\R^3)}\leq\eta.
\end{align}				
We also note that as $u$ and $\tilde{u}$ are bounded in $\E$, we have that $q(u)$, $q(\tilde{u})$ are bounded in $L_x^2$; $u$, $\tilde{u}$ are bounded in $L_x^6$; and $u_1,$ $\tilde{u}_1$ are bounded in $L_x^3\cap L_x^6$.

We will first show that $w$ is bounded in Strichartz spaces on $I\times\R^3$. To see this, we write
$$
(i\partial_t+\Delta)w=2\gamma\tilde{u}_1+N(\tilde{u})-[2\gamma u_1+N(u)],
$$ 
where $N(u)$ is as in \eqref{eq:cq3}. We make use of $q(u)$ and $q(\tilde{u})$ to rewrite
\begin{align*}
(i\partial_t+\Delta)w&=O(|\tilde{u}|^5+|u|^5)+O(|\tilde{u}|^4+|u|^4)+O(|\tilde{u}|^3+|u|^3)
\\ &\quad+\gamma q(\tilde{u})+(2\gamma+4)\tilde{u}_1^2+2i\gamma \tilde{u}_1\tilde{u}_2
\\ &\quad-[\gamma q(u)+(2\gamma+4)u_1^2+2i\gamma u_1u_2].
\end{align*}
As $w(0)=0$, we can use Strichartz to estimate
\begin{align*}
\|w\|_{L_t^2 L_x^6}&+\|w\|_{L_t^4 L_x^3}+\|w\|_{L_t^\infty L_x^2}
\\ &\lesssim \|\tilde{u}^5\|_{L_t^2 L_x^{6/5}}+\|u^5\|_{L_t^2 L_x^{6/5}}
+\|\tilde{u}^4\|_{L_t^{4/3}L_x^{3/2}}+\|u^4\|_{L_t^{4/3}L_x^{3/2}}
\\ &\quad+\|\tilde{u}^3\|_{L_t^1 L_x^2}+\| u^3\|_{L_t^1 L_x^2}
+\|q(\tilde{u})\|_{L_t^1 L_x^2}+\|q(u)\|_{L_t^1 L_x^2}
\\ &\quad+\|\tilde{u}\tilde{u}_1\|_{L_t^1L_x^2}+\|uu_1\|_{L_t^1L_x^2},
\end{align*}
where all space-time norms are over $I\times\R^3$. Using H\"older, we find
\begin{align*}
&\|u^5\|_{L_t^2L_x^{6/5}}\lesssim \tau^{1/2}\|u\|_{L_t^\infty L_x^6}^5, 
& \|u^4\|_{L_t^{4/3}L_x^{3/2}}\lesssim \tau^{3/4}\|u\|_{L_t^\infty L_x^6}^4,
\\ &\|u^3\|_{L_t^1 L_x^2}\lesssim \tau\|u\|_{L_t^\infty L_x^6}, 
& \|q(u)\|_{L_t^1L_x^2}\lesssim\tau\|q(u)\|_{L_t^\infty L_x^2},
\\ &\|uu_1\|_{L_t^1 L_x^2}\lesssim\tau\|u\|_{L_t^\infty L_x^6}\|u_1\|_{L_t^\infty L_x^3},
\end{align*}
and we can estimate similarly for $\tilde u$. Thus we conclude
\begin{align} \label{eq:w is finite}
\|w\|_{L_t^2 L_x^6}+\|w\|_{L_t^4 L_x^3}+\|w\|_{L_t^\infty L_x^2}<\infty.
\end{align}
	
We will show that in fact,
	\begin{equation}\label{eq:w is zero}\|w\|_{L_t^2 L_x^6}+\|w\|_{L_t^4 L_x^3}+\|w\|_{L_t^\infty L_x^2}=0,
	\end{equation}
which implies $w=0$ almost everywhere, as desired. To this end, we again rewrite the equation for $w$, using $z$ to indicate that either $w$ or $u$ may appear. We have
	$$(i\partial_t+\Delta)w=O(|w| |u|^4+|w|^5+|w| |z|^3+|w| |z|^2+|w| |z|+|w|).$$ 
We now use Strichartz, \eqref{small1} and \eqref{small2} to estimate
	\begin{align*}
	\|w\|_{L_t^2 L_x^6}&+\|w\|_{L_t^4 L_x^3}+\|w\|_{L_t^\infty L_x^2} 
	\\ &\lesssim \|wu^4\|_{L_t^{10/9}L_x^{30/17}}+\|w^5\|_{L_t^2 L_x^{6/5}}+\|wz^3\|_{L_t^2L_x^{6/5}}
	+\|wz^2\|_{L_t^{4/3}L_x^{3/2}}
	\\ &\quad+\|wz\|_{L_t^1L_x^2}+\|w\|_{L_t^1L_x^2}
	\\ &\lesssim \|u\|_{L_{t,x}^{10}}^4\|w\|_{L_t^2L_x^6}
	+\|w\|_{L_t^\infty L_x^6}^4\|w\|_{L_t^2L_x^6}
	+\tau^{1/4}\|z\|_{L_t^\infty L_x^6}^3\|w\|_{L_t^4L_x^3}
	\\ &\quad+\tau^{1/2}\|z\|_{L_t^\infty L_x^6}^2\|w\|_{L_t^4L_x^3}+ \tau^{3/4}\|z\|_{L_t^\infty L_x^6}\|w\|_{L_t^4L_x^3}
	+\tau\|w\|_{L_t^\infty L_x^2}
	\\ &\lesssim \eta^4\|w\|_{L_t^2L_x^6}+\tau^{1/4}\|w\|_{L_t^4L_x^3}+\tau\|w\|_{L_t^\infty L_x^2}.
	\end{align*}
Choosing $\eta,\tau$ sufficiently small and using \eqref{eq:w is finite}, we conclude that \eqref{eq:w is zero} holds and so $u=\tilde{u}$ almost everywhere on $I\times\R^3$.  As uniqueness is a local property, this yields uniqueness in the energy space for solutions to \eqref{eq:cq3}.
\end{proof}

Next we develop a stability theory for \eqref{eq:cq3}, which we will need in Section~\ref{section:weak}.

\begin{proposition}[Stability theory]\label{P:Stab}
Fix $T>0$ and let $\tilde u:[-T, T]\times\R^3\to \C$ be a solution to
the perturbed equation
$$
(i\partial_t+\Delta-2\gamma \Re  )\tilde u = N(\tilde u) +e
$$
for some function $e$.  Suppose that
\begin{align}\label{E:Lbound}
\|\tilde u\|_{L_t^\infty \dot H^1_x([-T, T]\times\R^3)} + \|\nabla
\tilde u\|_{L_t^{10} L_x^{\frac{30}{13}}([-T, T]\times\R^3)} \leq L
\end{align}
for some constant $L>0$.  Let $u_0\in \dot H_x^1(\R^3)$ and assume that
\begin{align}\label{small error}
\|\tilde u(0) - u_0\|_{\dot H^1_x} + \Bigl\|\int_0^t e^{i(t-s)\Delta}
\nabla e(s)\, ds \Bigr\|_{L_t^\infty L_x^2 \cap L_{t,x}^{\frac{10}{3}}([-T, T]\times\R^3)}\leq \eps
\end{align}
for some $\eps\leq \eps_0(L,T)$.  Then for $\eps_0(L,T)$ sufficiently small there exists a solution $u:[-T, T]\times\R^3\to \C$ to \eqref{eq:cq3} with data $u(0)=u_0$
and
\begin{align}
\|\nabla(\tilde u-u)\|_{L_t^\infty L_x^2 \cap L_{t,x}^{\frac{10}{3}}([-T, T]\times\R^3)}\leq C(L,T) \eps \label{E:S1close}\\
\|u\|_{\dot S^1([-T, T])}\leq C(L,T). \label{E:uS1}
\end{align}
\end{proposition}

\begin{proof}
The existence of the solution $u$ on a small neighborhood of $t=0$
follows from the arguments described in Theorem~\ref{thm:gwp}.  In
that setting, the solution could be extended globally due to energy
control.  That argument does not apply here as $u_0\in \dot H^1_x$ by
itself does not guarantee finiteness of the energy; furthermore, we
permit here large data even when $\gamma<\frac23$, in which case the
energy need not be coercive.  However, these earlier arguments do show
that if a solution should blow up in finite time, then the $\dot S^1$-norm must diverge.  
Consequently, we can prove that the solution
exists and obeys \eqref{E:S1close} and \eqref{E:uS1} on the whole interval $[-T,T]$ by
showing that it obeys \eqref{E:S1close} and \eqref{E:uS1} on any subinterval $0\ni
I\subseteq [-T,T]$ on which it does exist.  This is what we do.

For brevity, we define the following norm: given a time interval $[a,b]\subset \R$,
$$
\|u\|_{Y([a, b])}:= \|\nabla u\|_{L_t^\infty L_x^2 \cap L_{t,x}^{\frac{10}{3}}([a, b]\times\R^3)}.
$$

Given $0<\eta<1$ to be chosen later, we divide $I$ into intervals $J$ where
\begin{align}\label{302}
|J| \leq \eta \qtq{and} \|\nabla \tilde
u\|_{L_t^{10}L^{\frac{30}{13}}_x (J\times\R^3)} \leq \eta.
\end{align}
The number $K$ of such intervals depends only on $L$, $T$, and $\eta$.  Below we will show that for $\eta$ sufficiently small,
\begin{equation}\label{keyimplication}
\inf_{t_0\in J} \|\tilde u(t_0)-u(t_0)\|_{\dot H^1_x} \leq \eta
\implies \|\tilde u-u\|_{Y(J)}\leq A \inf_{t_0\in J} \|\tilde
u(t_0)-u(t_0)\|_{\dot H^1_x}
\end{equation}
for some absolute constant $A$ on such intervals $J$.  Iterating this completes the proof of \eqref{E:S1close} and yields constants
$$
\eps_0 = A^{-K(L,T,\eta)} \eta \qtq{and} C(L,T) = K(L,T,\eta) A^{K(L,T,\eta)}.
$$

We now verify \eqref{keyimplication}. Writing $u=\tilde u +v$, we use Strichartz and \eqref{small error} to estimate
\begin{align*}
\|v\|_{Y(J)} \lesssim \inf_{t_0\in J} \|v(t_0)\|_{\dot H^1_x}
+ \| \nabla [N(\tilde u+v) - N(\tilde u)] \|_{\dot N^0(J)} + |J|
\|v\|_{L^\infty_t \dot H^1_x} +\eps,
\end{align*}
where $N(\cdot)$ denotes the nonlinearity, as in \eqref{eq:cq3}.  Moreover,
\begin{align*}
\| \nabla [N(\tilde u+v) - N(\tilde u)] \|_{\dot N^0(J)} &\lesssim \|\nabla \tilde u\|_{L^{10}_t L^{\frac{30}{13}}_{x\vphantom{t}}} \|v\|_{L^{10}_{t,x}\vphantom{1^{\frac11}_1}}
	\sum_{k=2}^{5} |J|^{\frac{5-k}{4}} \Bigl( \|\tilde u\|_{L^{10}_{t,x}}^{k-2} + \|v\|_{L^{10}_{t,x}}^{k-2}\Bigr) \\
&\quad + \|\nabla v\|_{L^{10}_t L^{\!\frac{30}{13}}_{x\vphantom{t}}}
	\sum_{k=2}^{5} |J|^{\frac{5-k}{4}} \Bigl( \|\tilde u\|_{L^{10}_{t,x}}^{k-1} + \|v\|_{L^{10}_{t,x}}^{k-1}\Bigr),
\end{align*}
where all space-time norms are over $J\times\R^3$.  Using Sobolev embedding and \eqref{302}, we therefore obtain
\begin{align*}
\|v\|_{Y(J)}\lesssim \inf_{t_0\in J} \|v(t_0)\|_{\dot H^1_x} +\sum_{k=1}^5 \eta^{\frac{5-k}4} \|v\|_{Y(J)}^k +\eps.
\end{align*}
Choosing $\eta$ sufficiently small, a simple bootstrap argument yields \eqref{keyimplication}.

Using the fact that $u$ is a solution to \eqref{eq:cq3}, a further application of the Strichartz inequality gives \eqref{E:uS1}.
\end{proof}

We also record the following corollary.

\begin{corollary}[Small-data space-time bounds]\label{C:small}
Given $T>0$ there exists $\eta(T)>0$ such that 
$$
\text{if} \quad \|u_0\|_{\dot H^1_x}\leq \eta(T) \qtq{then} \|u\|_{\dot S^1([-T,T])}\lesssim_T \|u_0\|_{\dot H^1_x},
$$
where $u$ denotes the solution to \eqref{eq:cq3} with data $u(0)=u_0$.
\end{corollary}

\begin{proof}
We apply Proposition~\ref{P:Stab} with $\tilde u = e^{it\Delta} u_0$.  By the Strichartz inequality,
$$
\|\tilde u\|_{L_t^\infty \dot H^1_x([-T, T]\times\R^3)} + \|\nabla \tilde u\|_{L_t^{10}L_x^{\frac{30}{13}}([-T, T]\times\R^3)} \lesssim \|u_0\|_{\dot H^1_x},
$$
while a little computation yields 
\begin{align*}
\Bigl\|\int_0^t e^{i(t-s)\Delta} e(s)\, ds \Bigr\|_{\dot S^1([-T, T])} \lesssim \sum_{k=1}^5 T^{\frac{5-k}4} \|u_0\|_{\dot H^1_x}^k.
\end{align*}
Proposition~\ref{P:Stab} now gives the claim, provided $\eta(T)$ is taken sufficiently small.
\end{proof}

                     				%%%%%%%%%%%%%%	
						\section{Well-posedness in the weak topology}\label{section:weak}
						%%%%%%%%%%%%%%
						
In this section we prove the following well-posedness result. As described in the introduction, this theorem will play a key role in the proof of Theorem~\ref{thm:wave ops1} in Section~\ref{section:fs1}. 

\begin{theorem}[Weak topology well-posedness]\label{T:weak}
Let $\gamma\in (0,1)$ and let $\{u_n(0)\}_{n\geq 1}$ be a bounded sequence in $\mathcal E$.  Assume that $u_n(0)\rightharpoonup u_0$ weakly in $\dot H_x^1(\R^3)$.  If $\gamma\in (0, \tfrac23)$ we assume additionally that 
$$\|\nabla \Re  u_n(0)\|_{L_x^2}\leq \delta_\gamma
\quad\text{and}\quad
E(u_n(0))\leq \tfrac14 \delta_\gamma,$$
where $\delta_\gamma$ is as in Theorem~\ref{thm:gwp}.  Then there exists a unique solution $u\in C(\R;\E)$ to \eqref{eq:cq3} with $u(0)=u_0,$ and for all $t\in \R$ we have
\begin{align}\label{212}
u_n(t)\rightharpoonup u(t) \qtq{weakly in} \dot H^1_x(\R^3),
\end{align}
where $u_n\in C(\R;\E)$ denotes the solution to \eqref{eq:cq3} with initial data $u_n(0)$, whose existence is guaranteed by Theorem~\ref{thm:gwp}.
\end{theorem}											

We begin with the following lemma, which guarantees that the limit $u_0$ belongs to the energy space and obeys the necessary smallness conditions when $\gamma\in(0,\tfrac23)$, so that the existence and uniqueness of the solution $u\in C(\R;\E)$ follow from Theorem~\ref{thm:gwp}.

\begin{lemma}\label{L:bounds survive}
Fix $\gamma\in (0,1)$ and suppose $\{u_n\}_{n\geq 1}$ is a bounded sequence in $\E$ that satisfies $u_n(x-x_n) \rightharpoonup u_0(x)$ weakly in $\dot H_x^1(\R^3)$ for some sequence $\{x_n\}_{n\geq 1}\subseteq\R^3$.  Then $u_0\in \mathcal E$. Moreover, if $\gamma\geq \frac23$, then
\begin{equation}\label{E:Eweaklydown}
E(u_0)\leq \liminf_{n\to\infty} E(u_n).
\end{equation}
If  $\gamma\in(0,\frac23)$ and $\|\nabla\Re u_n\|_{L_x^2}^2\leq\delta_\gamma$, then $\|\nabla\Re u_0\|_{L_x^2}^2\leq\delta_\gamma$ and \eqref{E:Eweaklydown} holds. Here $\delta_\gamma$ is as in Theorem~\ref{thm:gwp}.
\end{lemma}

\begin{proof}
Without loss of generality, we may assume that $x_n\equiv 0$.

To prove that $u_0\in \mathcal E$, it suffices to show that $q(u_0)\in L^2_x$.  As $u_n \rightharpoonup u_0$ weakly in $\dot H_x^1(\R^3)$, invoking Rellich--Kondrashov and passing to a subsequence, we deduce that $u_n\to u_0$ in $L_x^p(K)$ for any $2\leq p<6$ and any compact set $K\subset \R^3$.  Therefore, for any ball $B\subset \R^3$,
\begin{align*}
\int_B |q(u_0(x))|^2\, dx = \lim_{n\to \infty} \int_B |q(u_n(x))|^2\, dx \leq \liminf_{n\to \infty} \int_{\R^3} |q(u_n(x))|^2\, dx <\infty.
\end{align*}
As the bound does not depend on $B$, this proves $q(u_0)\in L^2_x$.

Proceeding similarly and using (weak) lower semicontinuity of the $\dot H^1_x$- and $L^6_x$-norms, we obtain
$$
\int_B \tfrac12|\nabla u_0|^2+\tfrac\gamma4q(u_0)^2+\tfrac16q(u_0)^3\,dx \leq \liminf_{n\to \infty} \int_B \tfrac12|\nabla u_n|^2+\tfrac\gamma4q(u_n)^2+\tfrac16q(u_n)^3\,dx
$$
for any ball $B$.  It is crucial here that the sextic term in the energy appears with a positive coefficient.

When $\gamma\in[\frac23, 1)$, the energy density is positive and so the right-hand side above is majorized by $\liminf E(u_n)$.  When $\gamma\in(0, \frac23)$, we use instead \eqref{EoutsideB} to reach the same conclusion.  As $u_0\in \E$, the dominated convergence theorem yields \eqref{E:Eweaklydown}.
\end{proof}

We next prove a linear profile decomposition adapted to \eqref{eq:lin} for $\dot{H}_x^1$-bounded sequences. Beginning with the profile decomposition for the linear Schr\"odinger equation, we group the profiles according to the behavior of their associated parameters. We also show that the error term vanishes in the limit under propagation by $V^{-1}e^{-itH}V$ (in addition to propagation by $e^{it\Delta}$). 
								
\begin{proposition}[Linear profile decomposition]\label{P:LPD}
Suppose $\{f_n\}_{n\geq 1}$ is a bounded sequence in $\dot H^1_x(\R^3)$ and let $T>0$.  Passing to a subsequence, there exists $J^*\in \{0, 1, 2, \ldots\}\cup \{\infty\}$ and for each finite $1\leq j\leq J^*$ there exist a non-zero profile $\phi^j\in \dot H^1_x(\R^3)$, scales $\{\lambda_n^j\}_{n\geq 1}\subset (0,\infty)$, and positions $\{(t_n^j, x_n^j)\}_{n\geq 1}\subset \R\times\R^3$ conforming to one of the following two scenarios:
\begin{CI}
\item $\lambda_n^j\equiv 1$ and $t_n^j\equiv 0$, 
\item $\lambda_n^j\to 0$ as $n\to \infty$ and either $t_n^j\equiv 0$ or $t_n^j(\lambda_n^j)^{-2} \to \pm \infty$ as $n\to \infty$,  
\end{CI}
so that for any finite $0\leq J\leq J^*$ we have the decomposition
\begin{align}%\label{118}
\nonumber
f_n(x)=\sum_{j=1}^J  e^{-it_n^j\Delta}\bigl[(\lambda_n^j)^{-\frac12}\phi^j\bigl(\tfrac{x-x_n^j}{\lambda_n^j}\bigr)\bigr] +w_n^J(x)
\end{align}
satisfying the following properties:
\begin{gather}
(\lambda_n^j)^{\frac12}\bigl( e^{it_n^j\Delta}f_n\bigr)(\lambda_n^j x +x_n^j) \rightharpoonup \phi^j \quad\text{weakly in } \dot H^1_x, \label{119}\\
\lim_{J\to J^*}\!\limsup_{n\to \infty} \Bigl[\bigl\|V^{-1}e^{-itH}Vw_n^J\bigr\|_{L_{t,x}^{10}([-T,T]\times\R^3)}\!\!+\! \|e^{it\Delta} w_n^J\|_{L_{t,x}^{10}([-T,T]\times\R^3)}\Bigr]=0,\label{120}\\
\sup_J \limsup_{n\to \infty} \Bigl[ \| f_n\|_{\dot H^1_x}^2 - \sum_{j=1}^J \|\phi^j\|_{\dot H^1_x}^2 -\|w_n^J\|_{\dot H^1_x}^2\Bigr] \!=0,\label{121}\\
(\lambda_n^j)^{\frac12}\bigl( e^{it_n^j\Delta}w_n^J\bigr)(\lambda_n^j x +x_n^j) \rightharpoonup 0 \quad\text{weakly in } \dot H^1_x \qtq{for all} 1\leq j\leq J,\label{122}\\
\lim_{n\to \infty} \tfrac{\lambda_n^j}{\lambda_n^l}+\tfrac{\lambda_n^l}{\lambda_n^j} + \tfrac{|x_n^j-x_n^l|^2}{\lambda_n^j\lambda_n^l} + \tfrac{|t_n^j-t_n^l|}{\lambda_n^j\lambda_n^l}=\infty \qtq{for all} j\neq l.\label{123}
\end{gather}
\end{proposition}	

\begin{proof}
Using the linear profile decomposition for the Schr\"odinger propagator for bounded sequences in $\dot H^1_x$ (see, for example, \cite{Keraani-H1} or \cite[Theorem~4.1]{Oberwolfach}), we obtain a decomposition
\begin{align}\label{118'}
f_n(x)=\sum_{j=1}^J  e^{-it_n^j\Delta}\bigl[(\lambda_n^j)^{-\frac12}\phi^j\bigl(\tfrac{x-x_n^j}{\lambda_n^j}\bigr)\bigr] +r_n^J(x)
\end{align}
 satisfying \eqref{119}, \eqref{121}, \eqref{122}, and \eqref{123} (with $w_n^J$ replaced by $r_n^J$), as well as
\begin{align}\label{120'}
\lim_{J\to J^*}\limsup_{n\to \infty} \bigl\|e^{it\Delta}r_n^J\bigr\|_{L_{t,x}^{10}([-T,T]\times\R^3)}=0.
\end{align}

We will first show that we may assume the parameters conform to the two scenarios described above; in particular, we will show that we may absorb any other bubbles of concentration into the error $r_n^J$, while maintaining condition \eqref{120'}.  To complete the proof of the proposition, we will show that condition \eqref{120'} (for the new error term) suffices to prove \eqref{120}.  Note that it is essential in what follows that we work on a compact time interval.

In what follows we use the notation
$$
\phi_n^j(x):=e^{-it_n^j\Delta}\bigl[(\lambda_n^j)^{-\frac12}\phi^j\bigl(\tfrac{x-x_n^j}{\lambda_n^j}\bigr)\bigr].
$$

We begin with the following lemma.
\begin{lemma}\label{L:redundant profiles}
If $|t_n^j| + \lambda_n^j\to \infty$ as $n\to \infty$, then
\begin{align*}
\lim_{n\to \infty} \bigl\| e^{it\Delta} \phi_n^j\bigr\|_{L_{t,x}^{10}([-T,T]\times\R^3)}=0.
\end{align*}
\end{lemma}

\begin{proof}
A direct computation gives
\begin{align*}
\bigl\| e^{it\Delta} \phi_n^j\bigr\|_{L_{t,x}^{10}([-T,T]\times\R^3)} = \bigl\| e^{it\Delta} \phi^j\bigr\|_{L_{t,x}^{10}\bigl(\bigl[\frac{-t_n^j-T}{(\lambda_n^j)^2},\frac{-t_n^j+T}{(\lambda_n^j)^2}\bigr]\times\R^3\bigr)}.
\end{align*}

If $\lambda_n^j\to \infty$, then the lengths of the time intervals appearing on the right-hand side of the equality above shrink to zero; consequently, by the dominated convergence theorem combined with the Strichartz inequality, we deduce the claim.

Passing to a subsequence, we may henceforth assume that $\lambda_n^j\to \lambda^j\in [0, \infty)$.  In this case, we have $|t_n^j|\to \infty$, and so the time intervals escape to infinity.  Thus the claim follows once again from the dominated convergence theorem combined with the Strichartz inequality.
\end{proof}

Discarding the bubbles of concentration whose parameters satisfy the hypotheses of Lemma~\ref{L:redundant profiles}, we can now see that we may reduce attention to the two scenarios described in Proposition~\ref{P:LPD}.  Indeed, passing to a subsequence, we may assume that $\lambda_n^j\to \lambda^j\in [0, \infty)$ and $t_n^j\to t^j\in (-\infty, \infty)$.  If $\lambda^j\neq 0$, then we may assume that $\lambda_n^j\equiv 1$ and $t_n^j\equiv 0$ by redefining the corresponding profile to be $(\lambda^j)^{-\frac12} e^{-it^j\Delta}[\phi^j(\frac{\cdot}{\lambda^j})]$. The error incurred by this modification can be absorbed into $r_n^J$; indeed, we have
\begin{align*}
\bigl\|&\phi_n^j - (\lambda^j)^{-\frac12} e^{-it^j\Delta}\bigl[\phi^j\bigl(\tfrac{x-x_n^j}{\lambda^j}\bigr)\bigr]\bigr\|_{\dot H^1_x}\\
&\leq \bigl\|(\lambda^j_n)^{-\frac12} \phi^j\bigl(\tfrac{x}{\lambda^j_n}\bigr)-(\lambda^j)^{-\frac12} \phi^j\bigl(\tfrac{x}{\lambda^j}\bigr)\bigr\|_{\dot H^1_x} + 
\bigl\|(e^{-it_n^j\Delta} - e^{-it^j\Delta})\bigl[(\lambda^j)^{-\frac12} \phi^j\bigl(\tfrac{x}{\lambda^j}\bigr)\bigr]\bigr\|_{\dot H^1_x},
\end{align*}
which tends to zero as $n\to \infty$ by the strong convergence of the linear Schr\"odinger propagator.  If instead $\lambda^j=0$, then passing to a further subsequence we may assume that either $t_n^j\equiv 0$ or $t_n^j(\lambda_n^j)^{-2} \to \pm \infty$ as $n\to \infty$.  Indeed, if there is a subsequence along which
$t_n^j(\lambda_n^j)^{-2} \to \tau\in (-\infty, \infty)$, then we redefine the profile to be $e^{-i\tau\Delta}\phi^j$ and $t_n^j\equiv 0$.  It is easy to see that the resulting error can be absorbed into $r_n^J$.

It  remains to prove that the new error $w_n^J$ (which consists of $r_n^J$ plus the bubbles of concentration whose parameters satisfy the hypotheses of Lemma~\ref{L:redundant profiles}) obeys \eqref{120}.  This is a consequence of the following: if
\begin{align*}
\lim_{J\to J^*}\limsup_{n\to \infty} \bigl\|e^{it\Delta}w_n^J\bigr\|_{L_{t,x}^{10}([-T,T]\times\R^3)} = 0,
\end{align*}
then
\begin{align*}
\lim_{J\to J^*}\limsup_{n\to \infty} \bigl\|V^{-1}e^{-itH}Vw_n^J\bigr\|_{L_{t,x}^{10}([-T,T]\times\R^3)}=0.
\end{align*}
To prove this final implication, we argue as follows: In view of the representation \eqref{E:matrix prop} and the boundedness of $U$ and $\PHi  U^{-1}$, it suffices to verify that $e^{\mp itH}e^{\mp it\Delta}$ and $\PLo U^{-1}\sin(tH) e^{-it\Delta}$ are Mikhlin multipliers with bounds that are uniform for $t\in[-T,T]$.  In the former case, this follows from \eqref{HvsDelta}; with regard to the latter, see \eqref{E:2:30}.

This completes the proof of Proposition~\ref{P:LPD}.
\end{proof}

In the proof of Theorem~\ref{T:weak}, we will construct solutions to \eqref{eq:cq3} associated to each $\phi_n^j$. For profiles conforming to the first scenario in Proposition~\ref{P:LPD}, we can achieve this by an application of Lemma~\ref{L:bounds survive} and Theorem~\ref{thm:gwp}. For profiles conforming to the second scenario, this is a more difficult problem, which we address in the following proposition.

\begin{proposition}[Highly concentrated nonlinear profiles]\label{P:embedding}
Let $\phi\in \dot H_x^1(\R^3)$ and $T>0$. Assume $\{\lambda_n\}_{n\geq 1}\subset (0, \infty)$ and $\{(t_n, x_n)\}_{n\geq 1}\subset \R\times\R^3$ satisfy $\lambda_n\to 0$ and either $t_n\equiv 0$ or $t_n \lambda_n^{-2} \to \pm \infty$.  Then for $n$ sufficiently large, there exists a solution $u_n$ to \eqref{eq:cq3} with initial data
$$
u_n(0,x)= \phi_n(x):=e^{-it_n\Delta}\bigl[\lambda_n^{-\frac12}\phi\bigl(\tfrac{x-x_n}{\lambda_n}\bigr)\bigr]
$$
satisfying 
\begin{align}\label{1012}
\|u_n\|_{\dot S^1([-T,T])}\leq C(\|\phi\|_{\dot H^1_x}).
\end{align}
Moreover, for all $\eps>0$ there exist $\phi_\eps, \psi_\eps\in C^\infty_c([-T,T]\times\R^3)$ such that
\begin{align}
&\limsup_{n\to \infty}\bigl\|u_n(t,x) - e^{-i\gamma t}\lambda_n^{-\frac12} \phi_\eps\bigl(\tfrac{t-t_n}{\lambda_n^2}, \tfrac{x-x_n}{\lambda_n} \bigr)\bigr\|_{L_{t,x}^{10}([-T, T]\times\R^3)}\leq\eps,\label{147}\\
&\limsup_{n\to \infty}\bigl\|\nabla u_n(t,x) - e^{-i\gamma t}\lambda_n^{-\frac32} \psi_\eps\bigl(\tfrac{t-t_n}{\lambda_n^2}, \tfrac{x-x_n}{\lambda_n} \bigr)\bigr\|_{L_{t,x}^{\frac{10}3}([-T, T]\times\R^3)}\leq \eps.\label{148}
\end{align}
\end{proposition}					
						
\begin{proof}
As \eqref{eq:cq3} is space-translation invariant, without loss of generality we may assume that $x_n\equiv0$.

We proceed via a perturbative argument.  Specifically, using a solution to the defocusing energy-critical NLS, we will construct an approximate solution $\tilde u_n$ to \eqref{eq:cq3} with initial data asymptotically matching $\phi_n$.  This approximate solution will have good space-time bounds inherited from the solution to the defocusing energy-critical NLS.  Using the stability result Proposition~\ref{P:Stab}, we will then deduce that for $n$ sufficiently large, there exist true solutions $u_n$ to \eqref{eq:cq3} with $u_n(0)=\phi_n$  that inherits the space-time bounds of $\tilde u_n$, thus proving~\eqref{1012}. We turn to the details. 

If $t_n\equiv 0$, let $v$ be the solution to the defocusing energy-critical NLS
\begin{align}\label{EC}
(i\partial_t + \Delta)v = |v|^4 v
\end{align}
with initial data $v(0) = \phi$.  If $t_n\lambda_n^{-2}\to \pm\infty$, let $v$ be the solution to \eqref{EC} which scatters in $\dot{H}^1_x$ to $e^{it\Delta}\phi$ as $t\to\pm \infty$.  By the main result in \cite{CKSTT:gwp}, we have
\begin{align*}
\|v\|_{\dot S^1(\R)}\leq C\big(\|\phi\|_{\dot H^1_x}\big).
\end{align*}

We are now in a position to introduce the approximate solutions $\tilde u_n$ to \eqref{eq:cq3}.  For $n\geq 1$, we define
\begin{equation*}
\tilde{u}_n(t, x):=e^{-i\gamma t} \lambda_n^{-\frac{1}{2}}v\big(\tfrac{t-t_n}{\lambda_n^2},\tfrac{x}{\lambda_n}\big).
\end{equation*}
The phase factor $e^{-i\gamma t}$ is necessary; it replaces the linear factor in equation \eqref{eq:cq3} by a non-resonant term; see \eqref{tilde error} below.

Note that 
\begin{align}\label{1013}
\|\tilde u_n\|_{\dot S^1(\R)} =\|v\|_{\dot S^1(\R)}\leq C\big(\|\phi\|_{\dot H^1_x}\big).
\end{align}
Moreover, $\tilde u_n(0)$ asymptotically matches the initial data $u_n(0)=\phi_n$; indeed, by construction, we have
$$
\|\tilde{u}_n(0)-\phi_n\|_{\dot{H}^1_x}= \bigl\|v\bigl(-\tfrac{t_n}{\lambda_n^2}\bigr)-e^{-i\tfrac{t_n}{\lambda_n^2}\Delta}\phi\bigr\|_{\dot{H}^1_x}\to 0 \qtq{as} n\to \infty.
$$

To invoke the stability result Proposition~\ref{P:Stab} and deduce claim \eqref{1012}, it remains to show that $\tilde u_n$ is an approximate solution to \eqref{eq:cq3} on the interval $[-T,T]$ as $n\to \infty$.  A computation yields
\begin{align}\label{tilde error}
e_n:=(i\partial_t+\Delta-2\gamma\Re ) \tilde u_n - N(\tilde u_n) = -\gamma \overline{\tilde u_n} -\sum_{j=2}^4 N_j(\tilde u_n).
\end{align}
To establish \eqref{1012}, we have to verify that the error $e_n$ satisfies the smallness condition in \eqref{small error} for $n$ sufficiently large.

Let $\delta>0$ to be chosen later.  There exist $T_1,T_2>0$ sufficiently large so that
\begin{gather}
\|v\|_{L_{t,x}^{10}(\{|t|>T_1\}\times\R^3)}<\delta, \label{1047}\\
\|v(t)-e^{it\Delta}v_{\pm}\|_{\dot{H}_x^1}<\delta\quad\text{for}\quad \pm t>T_2,\label{1048}%\\
%\|v(t)-e^{it\Delta}v_+\|_{\dot H^1_x} <\delta \text{ for } t>T_2,\quad \|v(t)-e^{it\Delta}v_-\|_{\dot H^1_x} <\delta \text{ for } t<-T_2 \label{1048},
\end{gather}
where $v_{\pm}$ denote the asymptotic states for the solution $v$. Note that the existence of $v_\pm$ is a consequence of the global space-time bounds for $v$, as discussed in \cite{CKSTT:gwp}.

We first estimate the contribution of the higher order terms appearing in $e_n$ on the space-time slab $[-T,T]\times\R^3$. Defining
	$$I_n=\{\vert t-t_n\vert\leq\lambda_n^2 T_1\}\cap[-T,T]\quad
	\text{and}\quad
	I_n^c=\{\vert t-t_n\vert>\lambda_n^2 T_1\}\cap[-T,T],$$
we use Strichartz, H\"older, \eqref{1013} and \eqref{1047} to obtain
\begin{align*}
&\Bigl\|\int_0^t e^{i(t-s)\Delta} \nabla \sum_{k=2}^4 N_k(\tilde u_n)(s)\, ds \Bigr\|_{L_t^\infty L_x^2 \cap L_t^{10}L_x^{\frac{30}{13}}}\\
&\lesssim \|\nabla N_2(\tilde u_n)\|_{L_t^{\frac{20}{19}}L_x^{\frac{30}{16}}} +\|\nabla N_3(\tilde u_n)\|_{L_t^{\frac{5}{4}}L_x^{\frac{30}{19}}} + \|\nabla N_4(\tilde u_n)\|_{L_t^{\frac{20}{13}}L_x^{\frac{30}{22}}}\\
& \lesssim \|\nabla \tilde u_n\|_{L_t^{10}L_x^{\frac{30}{13}}} \Bigl\{\|\tilde u_n\|_{L_t^{\frac{20}{17}}L_x^{10}(I_n\times\R^3)} 
	+ \|\tilde u_n\|_{L_t^{\frac{20}{17}}L_x^{10}(I_n^c\times\R^3)}\Bigr\}\\
&\quad+ \|\nabla \tilde u_n\|_{L_t^{10}L_x^{\frac{30}{13}}} \|\tilde u_n\|_{L_{t,x}^{10}}\Bigl\{\|\tilde u_n\|_{L_t^{\frac{5}{3}}L_x^{10}(I_n\times\R^3)} + \|\tilde u_n\|_{L_t^{\frac{5}{3}}L_x^{10}(I_n^c\times\R^3)}\Bigr\}\\
& \quad +\|\nabla \tilde u_n\|_{L_t^{10}L_x^{\frac{30}{13}}} \|\tilde u_n\|_{L_{t,x}^{10}}^2\Bigl\{\|\tilde u_n\|_{L_t^{\frac{20}{7}}L_x^{10}(I_n\times\R^3)} + \|\tilde u_n\|_{L_t^{\frac{20}{7}}L_x^{10}(I_n^c\times\R^3)}\Bigr\}\\
&\lesssim_{\|\phi\|_{\dot H^1_x}} \sum_{k=2}^4 \big\{(\lambda_n^2T_1)^{\frac{5-k}4} + T^{\frac{5-k}4}\delta\big\}.
\end{align*}
Taking $\delta$ sufficiently small depending on $T$ and $n$ sufficiently large, we see that this contribution is acceptable.

Next we consider the contribution of the linear term appearing in $e_n$, again on the space-time slab $[-T,T]\times\R^3$.  First, we observe that by Strichartz and \eqref{1013}, we have
\begin{align}\label{tilde e_n 1}
\Bigl\|\int_0^t e^{i(t-s)\Delta} \overline{\tilde u_n(s)}\, ds \Bigr\|_{L_t^2 \dot H_x^{1,6}([-T,T]\times\R^3)}
\lesssim \| \tilde u_n\|_{L_t^1\dot H^1_x([-T,T]\times\R^3)} \lesssim_{\|\phi\|_{\dot H^1_x}} T.
\end{align}

To continue, using \eqref{1048} we cover $\R$ by three disjoint intervals $I_n^0$ and $I_n^{\pm}$ such that 
\begin{align}\label{110}
|I_n^0|\leq 2\lambda_n^{2}T_2 \qtq{and} \bigl\|\tilde u_n - e^{-i\gamma t} e^{i(t-t_n)\Delta}\bigl[\lambda_n^{-\frac12} v_{\pm}\bigl(\tfrac{\cdot}{\lambda_n}\bigr)\bigr] \bigr\|_{L_t^\infty \dot H^1_x(I_n^{\pm}\times\R^3)}<\delta.
\end{align}
By Strichartz, H\"older, \eqref{1013}, and \eqref{110}, we have
\begin{align}\label{111}
\Bigl\|\int_0^t e^{i(t-s)\Delta} \chi_{I_n^0}(s)\overline{\tilde u_n(s)}\, ds \Bigr\|_{L_t^\infty \dot H^1_x([-T,T]\times\R^3)}
\lesssim \|\tilde u_n\|_{L_t^1\dot H^1_x(I_n^0\times\R^3)} \lesssim_{\|\phi\|_{\dot H^1_x}} \lambda_n^{2}T_2.
\end{align}
Using the triangle inequality, Strichartz, and \eqref{110},
\begin{align}\label{112}
\Bigl\|\int_0^t e^{i(t-s)\Delta} & \chi_{I_n^{\pm}}(s)\overline{\tilde u_n(s)}\, ds \Bigr\|_{L_t^\infty \dot H^1_x([-T,T]\times\R^3)}\\
&\lesssim T\delta + \Bigl\|\int_0^t e^{i(t+t_n-2s)\Delta} \chi_{I_n^\pm}(s)e^{i\gamma s}\lambda_n^{-\frac12} \overline{v_{\pm}}\bigl(\tfrac{\cdot}{\lambda_n}\bigr)\, ds \Bigr\|_{L_t^\infty \dot H^1_x([-T,T]\times\R^3)}.\notag
\end{align}
Now for any $-T\leq a<b\leq T$, an application of Plancherel gives
\begin{align}\label{113}
\Bigl\|\int_a^b e^{is(\gamma-2\Delta)}\lambda_n^{-\frac12} \overline{v_{\pm}}\bigl(\tfrac{\cdot}{\lambda_n}\bigr)\, ds \Bigr\|_{\dot H^1_x}\notag
&=\Bigl\|\int_a^b e^{is(\gamma+2|\xi|^2)}|\xi|\lambda_n^{\frac52}  \widehat{\overline{v_{\pm}}}(\xi\lambda_n)\, ds \Bigr\|_{L^2_\xi}\notag\\
&\lesssim \bigl\| (\gamma +2|\xi|^2)^{-1}|\xi|\lambda_n^{\frac52}  \widehat{\overline{v_{\pm}}}(\xi\lambda_n) \bigr\|_{L^2_\xi}\notag\\
&\lesssim \bigl\| \tfrac{\lambda_n^2}{2|\xi|^2+\gamma \lambda_n^2} |\xi|\widehat{\overline{v_{\pm}}}(\xi)\bigr\|_{L^2_\xi},
\end{align}
which converges to zero as $n\to \infty$ by the dominated convergence theorem.  Collecting \eqref{111}, \eqref{112}, and \eqref{113}, we obtain that
\begin{align*}
\Bigl\|\int_0^t e^{i(t-s)\Delta} \overline{\tilde u_n(s)}\, ds \Bigr\|_{L_t^\infty \dot H_x^1([-T,T]\times\R^3)}
&\lesssim_{\|\phi\|_{\dot H^1_x}} \lambda_n^{2}T_2 + T\delta +o(1) \qtq {as} n\to \infty.
\end{align*}
Interpolating with \eqref{tilde e_n 1} and taking $\delta$ sufficiently small depending on $T$ and taking $n$ sufficiently large, we see that the contribution of the linear term in $e_n$ is also acceptable.  This completes the proof of \eqref{1012}.

Finally, we turn to \eqref{147} and \eqref{148}.  For $\eps>0$, we approximate $v$ by $\phi_{\eps}, \psi_{\eps} \in C_c^{\infty}(\R\times\R^3)$ such that
\begin{align*}
\|v-\phi_{\eps}\|_{L^{10}_{t,x}(\R\times\R^3)} < \tfrac{\eps}{2} \qtq{and} \|\nabla v-\psi_{\eps}\|_{L^{\frac{10}{3}}_{t,x}(\R\times\R^3)}< \tfrac{\eps}{2}
\end{align*}
and take $n$ sufficiently large so that
\begin{align*}
\|u_n-\tilde u_n\|_{L_{t,x}^{10}\cap L_t^{\frac{10}3}\dot H^{1,\frac{10}3}_x([-T,T]\times\R^3)} < \tfrac{\eps}{2}.
\end{align*}
The two claims now follow easily from the triangle inequality.
\end{proof}						

Finally we turn to the proof of Theorem~\ref{T:weak}.

\begin{proof}[Proof of Theorem~\ref{T:weak}]
As mentioned above, by Lemma~\ref{L:bounds survive} and Theorem~\ref{thm:gwp} we have that $u$ and all of the  $u_n$ are global in time solutions to \eqref{eq:cq3}.

Fix $T>0$.  We will show that for any subsequence in $n$ there exists a further subsequence so that along that subsequence, $u_n(t)\rightharpoonup u(t)$ weakly in $\dot H^1_x$ for all $t\in [-T,T]$.  As the limit is independent of the original subsequence, this will prove the theorem.

Given a subsequence in $n$, we apply Proposition~\ref{P:LPD} to $u_n(0)-u_0$ and pass to a further subsequence to obtain the decomposition
\begin{align*}
u_n(0) - u_0= \sum_{j=1}^J \phi_n^j + w_n^J \qtq{with} \phi_n^j(x):=e^{-it_n^j\Delta}\bigl[(\lambda_n^j)^{-\frac12}\phi^j\bigl(\tfrac{x-x_n^j}{\lambda_n^j}\bigr)\bigr],
\end{align*}
which satisfies the conclusions of that proposition.  By hypothesis, $u_n(0)-u_0\rightharpoonup 0$ weakly in $\dot H^1_x$; using also \eqref{122} and \eqref{123}, this implies that for all $j\geq 1$ we must have
\begin{align}\label{915}
w_n^J\rightharpoonup 0 \text{ weakly in } \dot H^1_x \qtq{and} (\lambda_n^j)^{-1} + |t_n^j| + |x_n^j|\to \infty \qtq{as} n\to \infty.
\end{align}
Indeed, one can first prove the divergence of the parameters by a contradiction argument. Briefly, if some $(\lambda_n^{j})^{-1}+\vert t_n^{j}\vert+\vert x_n^{j}\vert$ were to remain bounded as $n\to\infty$ then one could use \eqref{122} and \eqref{123} to deduce that $\phi^{j}=0$, a contradiction. Once the divergence of the parameters is established, the weak convergence of $w_n^J$ to zero then follows. 

Throughout the proof we write
$$
\phi_n^0(x):= e^{-it_n^0\Delta}\bigl[(\lambda_n^0)^{-\frac12}u_0\bigl(\tfrac{x-x_n^0}{\lambda_n^0}\bigr)\bigr] \qtq{with parameters} \lambda_n^0\equiv 1, \, t_n^0\equiv 0, \, x_n^0\equiv 0.
$$
In view of \eqref{915}, the decomposition
$$
u_n(0) = \sum_{j=0}^J \phi_n^j + w_n^J
$$
satisfies the conclusions of Proposition~\ref{P:LPD}.

We next construct nonlinear profiles associated to each $\phi_n^j$.  If $j$ conforms to the first scenario described in Proposition~\ref{P:LPD}, then \eqref{119} and Lemma~\ref{L:bounds survive} guarantee that $\phi_n^j\in \mathcal E$ and moreover, $\|\nabla \Re  \phi_n^j\|_{L_x^2}\leq \delta_\gamma$ and $E(\phi_n^j)\leq \frac14\delta_\gamma$ if $\gamma\in(0, \frac23)$.  Thus by Theorem~\ref{thm:gwp} there exists a unique solution $u_n^j$ to \eqref{eq:cq3} with data $u_n^j(0) = \phi^j_n$; in particular, $\|u_n^j\|_{\dot S^1([-T,T])}<\infty$.  Note that $u_n^0$ is simply the solution $u$ from the statement of Theorem~\ref{T:weak}.

If $j$ conforms to the second scenario described in Proposition~\ref{P:LPD}, we let $u_n^j$ denote the solution to \eqref{eq:cq3} with data $u_n^j(0)=\phi_n^j$ constructed in Proposition~\ref{P:embedding}.

In either scenario, for all $\eps>0$ there exists $\phi_\eps^j, \psi_\eps^j\in C^\infty_c([-T,T]\times\R^3)$ such that
\begin{align}
&\limsup_{n\to \infty}\bigl\|u_n^j(t,x) - e^{-i\gamma t}(\lambda_n^j)^{-\frac12} \phi_\eps^j\bigl(\tfrac{t-t_n^j}{(\lambda_n^j)^2}, \tfrac{x-x_n^j}{\lambda_n^j} \bigr)\bigr\|_{L_{t,x}^{10}([-T, T]\times\R^3)}\leq \eps\label{147'}\\
&\limsup_{n\to \infty}\bigl\|\nabla u_n^j(t,x) - e^{-i\gamma t}(\lambda_n^j)^{-\frac32} \psi_\eps^j\bigl(\tfrac{t-t_n^j}{(\lambda_n^j)^2}, \tfrac{x-x_n^j}{\lambda_n^j} \bigr)\bigr\|_{L_{t,x}^{\frac{10}3}([-T, T]\times\R^3)}\leq \eps.\label{148'}
\end{align}
Note that the phase $e^{-i\gamma t}$ has no significance for $j$ conforming to the first scenario described in Proposition~\ref{P:LPD}; we simply incorporate it so as to treat both cases uniformly.  For these $j$, $\phi_\eps^j$ and $\psi_n^j$ are chosen to approximate $e^{i\gamma t} u_n^j$. 

As a consequence of \eqref{147'}, \eqref{148'}, and the asymptotic orthogonality of parameters given by \eqref{123}, for all $j\neq l$ we have
\begin{align}\label{decouple}
\|u_n^j u_n^l\|_{L_{t,x}^5} +\| u_n^j \nabla u_n^l\|_{L_t^5L_x^{\frac{15}{8}}}+ \| \nabla u_n^j \nabla u_n^l\|_{L_t^5L_x^{\frac{15}{13}}} \to 0,
\end{align}
where all space-time norms are over $[-T, T]\times\R^3$.

We next claim that for $j\geq 1$ we have 
\begin{align}
\label{nonlin profile to zero}
u_n^j(t)\rightharpoonup 0\quad\text{weakly in } \dot{H}_x^1(\R^3)
\quad\text{as }n\to\infty\quad\text{for every }t\in[-T,T].
\end{align} 
Indeed, if $j$ conforms to the first scenario, then \eqref{915} implies that $\vert x_n^j\vert\to\infty$ and hence \eqref{nonlin profile to zero} follows from the space-translation invariance of \eqref{eq:cq3} together with uniqueness. If $j$ conforms to the second scenario, then we have $\lambda_n^j\to 0$; however, as \eqref{eq:cq3} is not scale invariant, the argument just described does not apply directly. For this case, we recall that according to the construction in Proposition~\ref{P:embedding}, $u_n^j$ are asymptotically close in $L_t^\infty\dot{H}_x^1$ (up to a phase factor) to rescaled solutions to the defocusing energy-critical NLS as $n\to\infty$. Using the scaling symmetry and uniqueness for \eqref{EC}, we see that these rescaled solutions converge weakly to $0$ in $\dot{H}_x^1$ at each time; by construction, $u_n^j$ inherit this property.

To continue, we define
$$
u_n^J(t) = \sum_{j=0}^J u_n^j(t) + V^{-1}e^{-itH} V w_n^J.
$$
Note that $u_n^J(0)=u_n(0)$.  In what follows we will prove that for $n$ and $J$ sufficiently large, $u_n^J$ is an approximate solution to \eqref{eq:cq3} with uniform space-time bounds on $[-T,T]\times\R^3$.  An application of Proposition~\ref{P:Stab} will then yield that for any $\eps>0$ there exist $n$ and $J$ sufficiently large so that 
$$
\|u_n-u_n^J\|_{L_t^\infty \dot H^1_x([-T,T]\times\R^3)}\leq \eps.
$$
On the other hand, using \eqref{915} and \eqref{nonlin profile to zero} and recalling that $u=u_n^0$, we see that for $J$ fixed, $u_n^J(t)- u(t)\rightharpoonup 0$ weakly in $\dot H^1_x$ for all $t\in [-T,T]$.  Thus by the triangle inequality, for any $\varphi\in C^\infty_c(\R^3)$, we have
\begin{align*}
\bigl|\langle u_n(t) - u(t) , \varphi\rangle \bigr|
&\leq \bigl|\langle u_n(t) - u_n^J(t) , \varphi\rangle \bigr| + \bigl|\langle u_n^J(t) - u(t) , \varphi\rangle \bigr|\\
&\leq \|u_n(t)-u_n^J(t)\|_{\dot H^1_x} \|\varphi\|_{\dot H^{-1}_x} +  \bigl|\langle u_n^J(t) - u(t) , \varphi\rangle \bigr|\\
&\lesssim_\varphi \eps + o(1) \qtq{as} n\to \infty,
\end{align*}
which proves the claim in Theorem~\ref{T:weak}.

Thus it remains to show that for $n$ and $J$ sufficiently large, $u_n^J$ are approximate solutions to \eqref{eq:cq3} with uniform space-time bounds on $[-T,T]\times\R^3$.  

Our first step in this direction is the following lemma.

\begin{lemma}[Finite space-time bounds]\label{L:STB}  Given $T>0$, we have
\begin{align}\label{unJ finite}
\sup_J \limsup_{n\to \infty} \Bigl[\| u_n^J\|_{L_{t,x}^{10}([-T,T]\times\R^3)} +\| \nabla u_n^J\|_{L_t^{10}L_{x}^{\frac{30}{10}}([-T,T]\times\R^3)}\Bigr]\lesssim 1.
\end{align}
Moreover, for any $\eta>0$ there exists $J'=J'(\eta)$ sufficiently large so that
\begin{align}\label{1143}
\limsup_{n\to \infty} \Bigl[\Bigl\| \sum_{j=J'}^Ju_n^j\Bigr\|_{L_{t,x}^{10}([-T,T]\times\R^3)} +\Bigl\| \sum_{j=J'}^J \nabla u_n^j\Bigr\|_{L_t^{10}L_{x}^{\frac{30}{10}}([-T,T]\times\R^3)}\Bigr]\leq \eta
\end{align}
uniformly in $J\geq J'$.
\end{lemma}

\begin{proof}
By the asymptotic decoupling of the $\dot H^1_x$-norm in \eqref{121}, there exists $J_0=J_0(T)$ such that for all $j\geq J_0$ we have 
$\|\phi^j\|_{\dot H^1_x} \leq \eta(T)$, where $\eta(T)$ is as in Corollary~\ref{C:small}.  In particular, 
\begin{align}\label{eventually small}
\|u_n^j\|_{\dot S^1([-T,T]}\lesssim_T\|\phi^j\|_{\dot H^1_x} \qtq{for all} j\geq J_0.
\end{align}

On the space-time slab $[-T,T]\times\R^3$ we use Lemma~\ref{L:VStrichartz} to estimate
\begin{align*}
\| u_n^J\|_{L_{t,x}^{10}}^2
&\lesssim  \|V^{-1}e^{-itH}Vw_n^J\|_{L_{t,x}^{10}}^2 + \Bigl\| \Bigl(\sum_{j=0}^J u_n^j\Bigr)^2\Bigr\|_{L_{t,x}^5} \\
&\lesssim_T \|w_n^J\|_{\dot H^1_x}^2 + \sum_{j=0}^J \|u_n^j\|_{L_{t,x}^{10}}^2 + \sum_{j\neq l} \|u_n^j u_n^l\|_{L_{t,x}^5} .
\end{align*}
This suffices to show that the first term on the left-hand side of \eqref{unJ finite} is finite.   Indeed, we use \eqref{121} and \eqref{eventually small} to bound the first two summands and \eqref{decouple} to bound the last (double) sum.  An analogous argument yields that the second term on the left-hand side of \eqref{unJ finite} is also bounded.

To prove \eqref{1143} one argues as above, taking $J'\geq J_0$ large enough that 
$$\sum_{j\geq J'} \|\phi^j\|^2_{\dot H^1_x}\lesssim \eta.$$ 
Note that this is possible because of \eqref{121}.\end{proof}

We next prove that the $u_n^J$ are indeed approximate solutions to \eqref{eq:cq3}. 

\begin{lemma}[Asymptotic solution to \eqref{eq:cq3}] \label{L:approx} We have
\begin{align*}
\lim_{J\to J^*}\limsup_{n\to \infty}\bigl\|\nabla \bigl[(i\partial_t +\Delta-2\gamma \Re ) u_n^J - N(u_n^J)\bigr]\bigr\|_{\dot{N}^0([-T, T])}=0.
\end{align*}
\end{lemma}

\begin{proof}
Throughout the proof of the lemma, all space-time norms will be over $[-T,T]\times\R^3$.  Writing $\tilde w_n^J := V^{-1}e^{-itH}Vw_n^J$, we have
\begin{align*}
e_n^J&:= (i\partial_t +\Delta-2\gamma \Re ) u_n^J - N(u_n^J) \\
&= \sum_{j=0}^J N(u_n^j) - N\Bigl(\sum_{j=0}^J u_n^j\Bigr) + N\bigl(u_n^J-\tilde w_n^J\bigr) - N(u_n^J).
\end{align*}

Computations similar to those employed in the proof of Proposition~\ref{P:Stab} yield
\begin{align*}
\Bigl\|\nabla& \Bigl[\sum_{j=0}^J N(u_n^j) - N\Bigl(\sum_{j=0}^J u_n^j\Bigr) \Bigr]\Bigr\|_{\dot N^0}
\lesssim \sum_{k=2}^5 \sum_{j\neq l} \sum_{m=0}^J T^{\frac{5-k}4}\|u_n^l\nabla u_n^j\|_{L_t^5L_x^{\frac{15}{8}}} \|u_n^m\|_{L_{t,x}^{10}}^{k-2},
\end{align*}
which converges to zero as $n\to \infty$ in view of \eqref{decouple} and \eqref{eventually small}.

Thus, it remains to show that
\begin{align}\label{222}
\lim_{J\to J^*}\limsup_{n\to \infty}\bigl\|\nabla \bigl[N\bigl(u_n^J-\tilde w_n^J\bigr) - N(u_n^J)\bigr]\bigr\|_{\dot{N}^0([-T, T])}=0.
\end{align}
We argue as follows: First, we estimate
\begin{align*}
\bigl\|\nabla &\bigl[N\bigl(u_n^J-\tilde w_n^J\bigr) - N(u_n^J)\bigr]\bigr\|_{\dot N^0([-T, T])}\\
&\lesssim \|\nabla u_n^J\|_{L_t^{10}L_x^{\frac{30}{13}}}\|\tilde w_n^J\|_{L_{t,x}^{10}} \sum_{k=2}^5 T^{\frac{5-k}4} \bigl(\|u_n^J\|_{L_{t,x}^{10}}^{k-2} +\|\tilde w_n^J\|_{L_{t,x}^{10}}^{k-2}\bigr)\\
&\quad+ \|\nabla \tilde w_n^J\|_{L_t^{10}L_x^{\frac{30}{13}}}\sum_{k=2}^5 T^{\frac{5-k}4}\|\tilde w_n^J\|_{L_{t,x}^{10}}^{k-1}+\|u_n^J\nabla \tilde w_n^J\|_{L_t^{5}L_x^{\frac{15}{8}}}\sum_{k=2}^5 T^{\frac{5-k}4} \|u_n^J\|_{L_{t,x}^{10}}^{k-2}.
\end{align*}
That the first two summands above go to zero as $n\to \infty$ and then $J\to \infty$ follows from \eqref{120} and Lemma~\ref{L:STB}.  Thus, \eqref{222} will follow from Lemma~\ref{L:STB} once we establish
\begin{align}\label{223}
\lim_{J\to J^*}\limsup_{n\to \infty}\|u_n^J\nabla \tilde w_n^J\|_{L_t^{5}L_x^{\frac{15}{8}}([-T, T]\times\R^3)}=0.
\end{align}

We will prove that $\text{LHS\eqref{223}}\lesssim \eta$ for arbitrary $\eta>0$.  By the definition of $u_n^J$, the triangle inequality, and H\"older, we estimate
\begin{align*}
\|u_n^J\nabla \tilde w_n^J\|_{L_t^{5}L_x^{\frac{15}{8}}}
&\lesssim \|\tilde w_n^J\|_{L_{t,x}^{10}} \|\nabla \tilde w_n^J\|_{L_t^{10}L_x^{\frac{30}{13}}} + \Bigl \| \sum_{j=J'}^J u_n^j\Bigl\|_{L_{t,x}^{10}} \|\nabla \tilde w_n^J\|_{L_t^{10}L_x^{\frac{30}{13}}}\\
&\quad+ \Bigl \| \sum_{j=0}^{J'-1} u_n^j \nabla \tilde w_n^J\Bigr\|_{L_t^{5}L_x^{\frac{15}{8}}},
\end{align*}
where $J'=J'(\eta)$ is as in the statement of Lemma~\ref{L:STB}.  Using \eqref{120} and \eqref{1143}, we see that the contribution of the first two summands on the right-hand side of the formula above is acceptable. 

It remains to prove that
\begin{align}\label{250}
\lim_{J\to J^*}\limsup_{n\to \infty}\|u_n^j\nabla \tilde w_n^J\|_{L_t^{5}L_x^{\frac{15}{8}}([-T, T]\times\R^3)}=0 \qtq{for each} 0\leq j<J'.
\end{align}

Assume first that $0\leq j<J'$ conforms to the first scenario in Proposition~\ref{P:LPD}.  Fix $\eps>0$.  Invoking \eqref{147'} and using the triangle inequality, H\"older, interpolation, and Corollary~\ref{C:smoothing}, we estimate
\begin{align*}
&\|u_n^j\nabla \tilde w_n^J\|_{L_t^{5}L_x^{\frac{15}{8}}}\\
&\leq \|u_n^j(t,x) - e^{-i\gamma t} \phi_\eps^j (t,x-x_n^j)\|_{L_{t,x}^{10}} \|\nabla \tilde w_n^J\|_{L_t^{10}L_x^{\frac{30}{13}}}
	+ \|\phi_\eps^j \nabla \tilde w_n^J (x+x_n^j)\|_{L_t^{5}L_x^{\frac{15}{8}}}\\
&\lesssim\eps+\|\phi_\eps^j\|_{L_t^\infty L_x^{12}}\|\nabla \tilde{w}_n^J(x+x_n^j)\|_{L_{t,x}^2(\supp\phi_\eps^j)}^{\frac14}
	\|\nabla\tilde{w}_n^J\|_{L_t^{10} L_x^{\frac{30}{13}}}^{\frac34}\\
&\lesssim_{\phi_\eps^j}\eps + \|\tilde{w}_n^J\|_{L_{t,x}^{10}}^{\frac1{12}}\| w_n^J\|_{\dot H_x^1}^{\frac16}
	\|\nabla\tilde{w}_n^J\|_{L_t^{10}L_x^{\frac{30}{13}}}^{\frac34}.
\end{align*}
By \eqref{120}, we see that \eqref{250} follows in this case. 

Now assume that $1\leq j<J'$ conforms to the second scenario in Proposition~\ref{P:LPD}.  We split $\tilde w_n^J$ into low and high frequencies and estimate them separately, starting with the low-frequency piece.  Fix $\eps>0$.  Arguing as before, using \eqref{147'}, H\"older, and Bernstein, we estimate
\begin{align*}
\|u_n^j P_{\leq (\lambda_n^j)^{-1}}\!\nabla \tilde w_n^J\|_{L_t^{5}L_x^{\frac{15}{8}}}
&\lesssim \eps + \bigl\| (\lambda_n^j)^{-\frac12} \phi_\eps^j\bigl(\tfrac{t-t_n^j}{(\lambda_n^j)^2}, \tfrac{x-x_n^j}{\lambda_n^j} \bigr) \bigr\|_{L_t^{10}L_x^{\frac{30}{13}}}\bigl\|P_{\leq (\lambda_n^j)^{-1}}\!\nabla \tilde w_n^J\bigr\|_{L_{t,x}^{10}}\\
&\lesssim \eps+ \|\phi_\eps^j\|_{L_t^{10}L_x^{\frac{30}{13}}}\|\tilde w_n^J\|_{L_{t,x}^{10}}.
\end{align*}
In view of \eqref{120}, this contribution is acceptable.

We now consider the high-frequency piece.  Using \eqref{HvsDelta} we can deduce
$$
\bigl\| P_{\geq N}  -  P_{\geq N} e^{it(\gamma-\Delta)} V^{-1}e^{-itH}V \bigr\|_{\dot H^1_x \to \dot H^1_x} \lesssim_T N^{-2}
$$
uniformly for $N\geq 1$ and $t\in[-T,T]$.  Thus
\begin{align*}
\|u_n^j &\nabla P_{\geq (\lambda_n^j)^{-1}}\tilde w_n^J\|_{L_t^{5}L_x^{\frac{15}{8}}([-T,T]\times\R^3)}\\
&\lesssim_T \|u_n^j P_{\geq (\lambda_n^j)^{-1}}\nabla e^{it\Delta} w_n^J\|_{L_t^{5}L_x^{\frac{15}{8}}([-T,T]\times\R^3)}
	+ (\lambda_n^j)^2 \| u_n^j \|_{L^{10}_{t,x}([-T,T]\times\R^3)} \|w_n^J\|_{\dot H^1_x} \\
&\lesssim_T \eps+ \| \phi_\eps^j \nabla e^{it\Delta}f_n\|_{L_t^{5}L_x^{\frac{15}{8}}(I_n\times\R^3)} +o(1) \qtq{as} n\to \infty, 
\end{align*}
where
$$
f_n(x) = P_{\geq 1} (\lambda_n^j)^{\frac12}w_n^J(\lambda_n^j x+x_n^j) \qtq{and} I_n = \{|t-t_n^j|\leq (\lambda_n^j)^2 T\}.
$$
To continue, we estimate in much the same manner as for $j$ conforming to the first scenario:
\begin{align*}
\| \phi_\eps^j &\nabla e^{it\Delta}f_n\|_{L_t^{5}L_x^{\frac{15}{8}}(I_n\times\R^3)}\\
&\lesssim \|\phi_\eps^j\|_{L_t^\infty L_x^{12}}\|\nabla e^{it\Delta}f_n\|_{L_{t,x}^2(\supp\phi_\eps^j \cap I_n\times\R^3)}^{\frac14}
	\|\nabla e^{it\Delta}f_n\|_{L_t^{10}L_x^{\frac{30}{13}}(I_n\times\R^3)}^{\frac34}\\
& \lesssim_{\phi_\eps^j} \|e^{it\Delta}f_n\|_{L_{t,x}^{10}(I_n\times\R^3)}^{\frac{1}{12}}
	\| f_n\|_{\dot H_x^1}^{\frac16}\|\nabla e^{it\Delta}f_n\|_{L_t^{10}L_x^{\frac{30}{13}}(I_n\times\R^3)}^{\frac34}\\
&\lesssim_{\phi_\eps^j} \|e^{it\Delta}w_n^J\|_{L_{t,x}^{10}([-T,T]\times\R^3)}^{\frac1{12}}
	\| w_n^J\|_{\dot H_x^1}^{\frac16}\|\nabla e^{it\Delta}w_n^J\|_{L_t^{10}L_x^{\frac{30}{13}}([-T,T]\times\R^3)}^{\frac34}.	
\end{align*}
where we have again used Corollary~\ref{C:smoothing}.  Recalling \eqref{120}, we see that the contribution of the high-frequency piece is acceptable.  This completes the proof of \eqref{250} and hence the proof of Lemma~\ref{L:approx}. 
\end{proof}

The final step in checking the hypotheses of Proposition~\ref{P:Stab}, which will finish the proof of Theorem~\ref{T:weak}, is to verify that
$$
\limsup_{J\to J^*}\limsup_{n\to \infty} \|u_n^J\|_{L_t^\infty \dot H^1_x([-T,T]\times\R^3)}\lesssim_T 1.
$$
In view of Lemma~\ref{L:VStrichartz}, we have
\begin{align*}
\|u_n^J\|_{L_t^\infty \dot H^1_x([-T,T]\times\R^3)} \lesssim_T \|u_n^J(0)\|_{\dot H^1_x} + \|\nabla e_n^J\|_{\dot N^0([-T,T])} + \|\nabla N(u_n^J)\|_{\dot N^0([-T,T])}.
\end{align*}
The requisite bounds on the right-hand side now follow from Lemmas~\ref{L:STB}~and~\ref{L:approx}. \end{proof}

						%%%%%%%%%%%%%%	
						\section{Normal form transformation}\label{section:normal form}
						%%%%%%%%%%%%%%

In this section we discuss the normal form transformation that we use throughout the rest of the paper. The use of normal form transformations originates in work of Shatah \cite{Sha} and has since become a widely used technique in the setting of nonlinear dispersive equations. The transformation we use is similar to the one used by Gustafson, Nakanishi, and Tsai in the setting of the Gross--Pitaevskii equation \cite{GNT:dd, GNT:2d, GNT:3d}. 

Suppose $u$ is a solution to \eqref{eq:cq3}. As mentioned in the introduction, the quadratic terms in the nonlinearity are the most problematic when it comes to questions of long-time behavior; in particular, the worst terms are those containing $u_2=\Im u$, since in the diagonal variables we have $u_2=U^{-1}v_2$. We would like to find a normal form transformation that eliminates if not all, at least the worst quadratic terms. 

To this end, we let $B_1[\cdot,\cdot]$ and $B_2[\cdot,\cdot]$ be arbitrary bilinear Fourier multipier operators defined as in \eqref{eq:bilin}, with symmetric real-valued symbols $B_1(\xi_1,\xi_2)$ and $B_2(\xi_1,\xi_2)$.  Then 
	$$\tilde{u}:=u+B_1[u_1,u_1]+B_2[u_2,u_2]$$
satisfies the following equation
\begin{align}
(i\partial_t+\Delta)\tilde{u}-2\gamma\tilde{u}_1&=(3\gamma+4)u_1^2-(2\gamma-\Delta)B_1[u_1,u_1] \label{u1u1}
\\ &\quad+\gamma u_2^2-(2\gamma-\Delta)B_2[u_2,u_2]									\label{u2u2}
\\ &\quad+2i\big(\gamma u_1u_2+B_1[u_1,-\Delta u_2]-B_2[u_2,(2\gamma-\Delta)u_1]\,\big)
	\label{u1u2}
\\ &\quad+\text{cubic and higher order terms.}													\nonumber%\label{cubicandhigher}
\end{align}
	
While the symmetry of $B_1$ and $B_2$ makes it impossible to eliminate \emph{all} of the quadratic terms, we see that if we choose
	$$B_2(\xi_1,\xi_2)=\gamma(2\gamma+\vert\xi_1+\xi_2\vert^2)^{-1},
	\quad\text{i.e.}\quad B_2[f,g]=\gamma\jb^{-2}(fg),$$ 
then \eqref{u2u2}=0.  This allows us to eliminate the worst quadratic term, namely, the one containing two copies of $u_2$.  Moreover, choosing $B_1=B_2$ we get
	$$\tilde{u}=u+\gamma\jb^{-2}\vert u\vert^2,$$ 
with 
	\begin{align*}
	\eqref{u1u1}=(2\gamma+4)u_1^2 \quad\text{and}\quad \eqref{u1u2}=-4i\gamma\jb^{-2}\nabla\cdot[u_1\nabla u_2].
	\end{align*}
The derivative appearing in front of $u_2$ is a welcome addition in light of the problem at low frequencies.   

Similarly one can compute the higher order terms. In general,  one finds that for $k\in\{3,4,5\}$ the terms of order $k$ are given by
	$$N_k(u)+2i\big\{B_1[u_1,\Im(N_{k-1}(u))]-B_2[u_2,\Re(N_{k-1}(u))]\big\},$$
where the $N_k$ are as in \eqref{eq:cq3}.  Notice that there are no sixth order terms, since $B_1=B_2$ and $u_1\Im(N_5(u))=u_2\Re(N_5(u))$. 

Finally, we employ the transformation $Vu=u_1+iUu_2$ to diagonalize the equation. Our normal form transformation is therefore given by
	\begin{equation}\label{eq:nft}
	z:=M(u):=Vu+\gamma \jb^{-2}\vert u\vert^2,
	\end{equation}								
and $z$ satisfies the equation
	\begin{equation}\label{eq:cq z}
	(i\partial_t-H)z=N_z(u)
	\end{equation}
with
	\begin{align*}
	\Re[N_z(u)]&=U\Re[N(u)-\gamma|u|^2]\\
	&=U\big[(2\gamma+4)u_1^2+(\gamma+8)u_1^3+(\gamma+4)u_1u_2^2\\
	&\quad\quad\quad+(5u_1^4+6u_1^2u_2^2+u_2^4)+|u|^4u_1)\big],
	\\
	\Im[N_z(u)]&=-\tfrac{\nabla}{\jb^2}\cdot\big[4\gamma u_1\nabla u_2+\nabla(\gamma|u|^2u_2 +q^2u_2)\big]\\
	&=-\tfrac{\nabla}{\jb^2}\cdot\big[4\gamma u_1\nabla u_2\big]+U^2\big[(\gamma+4)u_1^2u_2+\gamma u_2^3+4u_1u_2|u|^2+ |u|^4u_2)\big].
	\end{align*}

We should briefly pause to point out the improvements present in equation \eqref{eq:cq z} with respect to \eqref{eq:cq3}.  Firstly, equation \eqref{eq:cq z} does not contain a quadratic term involving two copies of $u_2$.  Secondly, the remaining quadratic terms involving $u_2$ exhibit a derivative of this problematic term.  Lastly, all the remaining terms appear with a derivative at low frequencies, which is helpful throughout.

We next discuss the invertibility of the transformation \eqref{eq:nft}. Note that by using the definition of $\jb^{-2}$ we can rewrite the transformation as
	\begin{equation}\label{eq:M rewrite}
	M(u)=U^2u_1+\gamma\jb^{-2}q+iUu_2,
	\end{equation}
where $q=q(u)=2u_1+\vert u\vert^2.$ 

This normal form transformation is a homeomorphism from a neighborhood of zero in $\E$ onto a neighborhood of zero in $H^1_x$.  To prove this, we make use of the following neighborhoods:
\begin{align*}
\M_{E_0,\eps_0}&:=\{u\in\E:E(u)\leq E_0^2,\ \|u\|_{L_x^6}\leq\eps_0\},\\ 	
\N_{E_0',\eps_0'}&:=\{f\in H_x^1:\|f\|_{H_x^1}\leq CE_0',\ \|f\|_{L_x^6}\leq C\eps_0'\},
\end{align*}
where $C$ denotes an absolute constant depending on $\gamma$. 
	
\begin{proposition}\label{prop:normal1} Fix $E_0>0$ and $\eps_0>0$.
\begin{SL}
\item If $\gamma\in(\tfrac23,1)$, then $M:\M_{E_0,\eps_0}\to \N_{E_0,\eps_0+\eps_0^2}$ continuously.
\item If $\gamma=\tfrac23$, then $M:\M_{E_0,\eps_0}\to \N_{E_0+E_0^3,\eps_0+\eps_0^2}$ continuously.
\item If $\gamma\in(0,\tfrac23)$, then $M:\M_{E_0,\eps_0}\cap\{\|\nabla u_1\|_2^2\leq\delta_\gamma\}\to \N_{E_0,\eps_0+\eps_0^2}$ continuously, where $\delta_\gamma$ is as in Lemma~\ref{lemma:coercive2}.
\item Given $E_1>0$, there exists $\eps_1=\eps_1(E_1)$ and a continuous mapping 
$$
R:\N_{E_1,\eps_1}\to\E
$$ 
such that $M\circ R=Id$ on $\N_{E_1,\eps_1}$ and $\|R(f)\|_{\E}\lesssim E_1$ for $f\in \N_{E_1,\eps_1}$. 
\item Suppose $\gamma\geq \frac23$ .  Given $E_2>0$, there exists $\eps_2=\eps_2(E_2)$ so that  $M$ is a homeomorphism of $\M_{E_2,\eps_2}$ onto a subset of $H_x^1(\R^3)$ and has inverse $R$.  In particular, $M$ is injective on $\M_{E_2,\eps_2}$.  If $\gamma<\frac23$, then the analogous assertions hold on $\M_{E_2,\eps_2}\cap\{\|\nabla u_1\|_2^2\leq\delta_\gamma\}$.
\end{SL}
\end{proposition}

\begin{remark}
We warn the reader that just because $M(u)$ is small in $L^6_x$, one cannot guarantee that $u=(R\circ M)(u)$.  However, this would follow if $u$ were sufficiently small in $L_x^6$.  This subtlety contributes nontrivially to the complexity of the proof of Theorem~\ref{thm:wave ops1}.
\end{remark}

\begin{proof} The proofs of the first three claims parallel one another closely.  We will only present the details when $\gamma\in(\tfrac23,1)$.  

Let $u\in\M_{E_0,\eps_0}$. Recall from Lemma~\ref{lemma:coercive1} that
	$$\|u\|_{\E}^2\lesssim E(u)\lesssim E_0^2.$$ 

We first show that $M(u)\in\N_{E_0,\eps_0+\eps_0^2}$. Using the representation \eqref{eq:M rewrite}, we estimate
	\begin{align*}
\|M(u)\|_{H^1_x}\lesssim\|U^2u_1\|_{H^1_x}+\|\jb^{-2}q\|_{H_x^1}+\|Uu_2\|_{H_x^1}\lesssim\| u\|_{\dot{H}_x^1}+\|q\|_{L_x^2} \lesssim E_0.
	\end{align*}
Using the representation \eqref{eq:nft} and Sobolev embedding, we estimate
	\begin{align*}
\|M(u)\|_{L_x^6}&\lesssim\|u_1\|_{L_x^6}+\|Uu_2\|_{L_x^6}+\|\jb^{-2}\vert u\vert^2\|_{L_x^6}\\
&\lesssim \|u\|_{L_x^6}+\|\vert\nabla\vert^{1/2}\jb^{-2}\vert u\vert^2\|_{L_x^3}\\
&\lesssim \| u\|_{L_x^6} + \|u\|_{L_x^6}^2\lesssim \eps_0+\eps_0^2.
	\end{align*}	
Collecting these estimates, we conclude $M(u)\in \N_{E_0,\eps_0+\eps_0^2}.$ 

To prove the continuity of $M$, we note that for $u,v\in\E$ we may write
	$$M(u)-M(v)=U^2(u_1-v_1)+\gamma\jb^{-2}[q(u)-q(v)]+iU(u_2-v_2).$$
Estimating as above  we find
	$$\|M(u)-M(v)\|_{H_x^1}\lesssim d_\E(u,v).$$ 
	
We turn now to the fourth claim in the statement of the proposition. Let $f\in \N_{E_1,\eps_1}$. We aim to show that for $\eps_1=\eps_1(E_1)>0$ sufficiently small, we can find a unique $u\in\E$ such that $M(u)=f$, that is, 
	\begin{equation}\nonumber\left\{ \begin{array}{l} 
	u_2=U^{-1}f_2,
	\\ u_1=f_1-\gamma\jb^{-2}[U^{-1}f_2]^2-\gamma\jb^{-2}u_1^2.
	\end{array}\right.\end{equation}
To this end, we define 
	$$R_f(u_1):=f_1-\gamma\jb^{-2}[U^{-1}f_2]^2-\gamma\jb^{-2}u_1^2.$$
We will show that for $\eps_1=\eps_1(E_1)$ sufficiently small, $R_f$ is a contraction on
	$$B:=\{u_1\in\dot{H}_x^1:\|u_1\|_{\dot{H}_x^1}\leq CE_1,\ \|u_1\|_{L_x^6}\leq C(\eps_1+\eps_1^{1/2}E_1^{3/2})\}$$
with respect to the metric 	$d(u_1,v_1)=\|u_1-v_1\|_{\dot{H}_x^1}$, where $C$ denotes an absolute constant depending on $\gamma$.

We first show that $R_f:B\to B$. We have
	\begin{align}
	\|R_f(u_1)\|_{L_x^6}
	&\lesssim\|f_1\|_{L_x^6}
	+\big\|\langle\nabla\rangle^{-2}\big[ U^{-1}f_2\big]^2\big\|_{L_x^6}
	+\|\langle\nabla\rangle^{-2}u_1^2\|_{L_x^6}.
	\label{eq:contraction1}
	\end{align}
The first term in \eqref{eq:contraction1} is controlled by $\eps_1$ by assumption. For the second term in \eqref{eq:contraction1}, we use Sobolev embedding, Bernstein, and interpolation to estimate
	\begin{align*}
	\big\|\jb^{-2}\big[U^{-1}f_2\big]^2\big\|_{L_x^6}&\lesssim \big\|\tfrac{\nabla}{\jb^{2}}\big[\PLo U^{-1}f_2\big]^2\big\|_{L_x^2} + \big\|\tfrac{\nabla}{\jb^{2}}\bigl[(\PHi   U^{-1}f_2)\text{\O}(U^{-1}f_2)\bigr]\big\|_{L_x^2}
	\\ &\lesssim\|\nabla \PLo U^{-1}f_2\|_{L_x^{3}}\| U^{-1} f_2\|_{L_x^6} + \| \PHi  U^{-1}f_2\|_{L_x^{3}}\|U^{-1}f_2\|_{L_x^6}
	\\ &\lesssim \|f\|_{L_x^6}^{\frac12}\|f\|_{H_x^1}^{\frac32}.
	\end{align*}
For the third term in \eqref{eq:contraction1}, we have
	\begin{align*}
	\|\jb^{-2} u_1^2\|_{L_x^6}\lesssim\|\vert\nabla\vert^{\frac12}\jb^{-2} u_1^2\|_{L_x^{3}}\lesssim\|u_1\|_{L_x^6}^2.
	\end{align*}	
Thus, for $u_1\in B$ and $\eps_1=\eps_1(E_1)$ sufficiently small we obtain
	$$\|R_f(u_1)\|_{L_x^6}\leq C(\eps_1+\eps_1^{1/2}E_1^{3/2}).$$

To continue, we estimate
	\begin{equation}\label{eq:contraction2}
	\|R_f(u_1)\|_{\dot{H}_x^1}\lesssim\|f_1\|_{\dot{H}_x^1}+\big\|\jb^{-2}\big[U^{-1}f_2\big]^{2}\big\|_{\dot{H}_x^1} +\|\jb^{-2}u_1^2\|_{\dot{H}_x^1}.
	\end{equation}
The first term in \eqref{eq:contraction2} is controlled by $E_1$ by assumption. For the second term in \eqref{eq:contraction2}, we argue as above to find
	\begin{align*}
	\big\|\jb^{-2}\big[ U^{-1}f_2\big]^2\big\|_{\dot{H}_x^1} &\lesssim\big\|\tfrac{\nabla}{\jb^{2}}\big[\PLo U^{-1}f_2\big]^2\big\|_{L_x^2}
		+\big\|\tfrac{\nabla}{\jb^{2}}\bigl[(\PHi   U^{-1}f_2)\text{\O}(U^{-1}f_2)\bigr]\big\|_{L_x^2}
	\\&\lesssim \|f\|_{L_x^6}^{\frac12}\|f\|_{H_x^1}^{\frac32}.
	\end{align*}
For the third term in \eqref{eq:contraction2} we estimate
	$$\|\jb^{-2} u_1^2\|_{\dot{H_x^1}} \lesssim \|\vert\nabla\vert^{3/2}\jb^{-2} u_1^2\|_{L_x^{3/2}} \lesssim \|u_1\|_{L_x^6}\|\nabla u_1\|_{L_x^2}.$$	
Thus for $u_1\in B$ and $\eps_1=\eps_1(E_1)$ sufficiently small we have
	$$\|R_f(u_1)\|_{\dot{H}_x^1}\leq C E_1.$$
	
Collecting these estimates, we conclude that $R_f:B\to B$. 

Next we show that $R_f$ is a contraction with respect to the $\dot{H}_x^1$-norm. We first use Sobolev embedding, Bernstein, and interpolation to estimate
\begin{align}\notag
\|&\tfrac{1}{\jb^{2}}\big[(u_1+v_1)(u_1-v_1)\big] \|_{\dot{H}_x^1}\\ \notag
&\lesssim \|\tfrac{\vert\nabla\vert^{3/2}}{\jb^{2}}\big[(u_1+v_1)\PHi  (u_1-v_1)\big]\|_{L_x^{3/2}}
 +\|\tfrac{1}{\jb^{2}}\big[\PLo (u_1+v_1)\, \PLo (u_1-v_1)\big]\|_{\dot{H}_x^1}\\ \notag
&\quad +\|\tfrac{1}{\jb^{2}}\big[\PHi  (u_1+v_1)\, \PLo (u_1-v_1)\big]\|_{\dot{H}_x^1}\\ \notag
&\lesssim\|u_1+v_1\|_{L_x^6} \| \PHi  (u_1-v_1)\|_{L_x^2}	+ \|\nabla \PLo (u_1+v_1)\|_{L_x^{3}}\|u_1-v_1\|_{L_x^6} \\ \notag
&\quad +\|u_1+v_1\|_{L_x^6}\|\nabla \PLo (u_1-v_1)\|_{L_x^{3}}+\|\PHi  (u_1+v_1)\|_{L_x^{3}}\|u_1-v_1\|_{L_x^6}\\
&\lesssim \big(\|u_1\|_{L_x^6}^{\frac12}\|u_1\|_{\dot{H}_x^1}^{\frac12}+\|v_1\|_{L_x^6}^{\frac12}\|v_1\|_{\dot{H}_x^1}^{\frac12}\big)\|u_1-v_1\|_{\dot{H}_x^1}.\label{4}
\end{align}
In particular, for $\eps_1=\eps_1(E_1)$ sufficiently small we deduce that 
$$
\| R_f(u_1)-R_f(v_1) \|_{\dot{H}_x^1}\leq \tfrac12 \|u_1-v_1\|_{\dot{H}_x^1}.
$$

Therefore, by the contraction mapping theorem there exists a unique $u_1\in B$ such that $R_f(u_1)=u_1$. We define $R(f):=u_1+iU^{-1}f_2$.  By construction, we have $M(R(f))=f$.

It remains to see that $u:=R(f)\in\E$ with $\|u\|_{\E}\lesssim E_1$.  As $u_1\in B$, we have
	$$\|u\|_{\dot H^1_x}\lesssim \|u_1\|_{\dot{H}_x^1} + \|U^{-1}f_2\|_{\dot{H}_x^1} \lesssim E_1 +\|f_2\|_{H_x^1}\lesssim E_1.$$
Moreover, by H\"older,
\begin{align*}
\|q(u)\|_{L_x^2}=\|2f_1+U^2|u|^2\|_{L_x^2}&\lesssim\|f\|_{L_x^2}+\|U(|u|^2)\|_{L_x^2}\\
&\lesssim E_1 +	\|\nabla\text{\O}[(\PLo u)^2]\|_{L_x^2} + \|\text{\O}(u\PHi  u)\|_{L_x^2}\\
&\lesssim E_1 +	\|\nabla \PLo u\|_{L_x^2} \|\PLo u\|_{L_x^\infty}+ \|u\|_{L_x^6}\|\PHi  u\|_{L_x^3}\\
&\lesssim E_1 + \|u\|_{\dot H^1_x} \bigl[\|\PLo u\|_{L_x^\infty}+\|\PHi  u\|_{L_x^3}\bigr].
\end{align*}
Using Bernstein, H\"older, and interpolation, we estimate
\begin{align*}
\|\PLo u\|_{L_x^\infty}
\lesssim \|u_1\|_{L^6_x} + \|\PLo U^{-1}f_2\|_{L_x^{10}}
&\lesssim \|u_1\|_{L^6_x} + \|f_2\|_{L_x^{\frac{30}{13}}}\\
&\lesssim \|u_1\|_{L^6_x} + \|f_2\|_{L_x^6}^{\frac15}\|f_2\|_{L_x^2}^{\frac45}
\end{align*}
and
\begin{align*}
&\|\PHi  u\|_{L_x^3}
\lesssim \|\PHi  u_1\|_{L^3_x} + \|\PHi  f_2\|_{L^3_x}
\lesssim  \|u_1\|_{L^6_x}^{\frac12}\|u_1\|_{\dot H^1_x}^{\frac12} + \|f_2\|_{L^6_x}^{\frac12}\|f_2\|_{\dot H^1_x}^{\frac12}.
\end{align*}
Taking $\eps_1=\eps_1(E_1)$ sufficiently small, this proves $\|q(u)\|_{L^2_x}\lesssim E_1$. 

To complete the proof of the proposition, it remains to address part (v).  From \eqref{eq:nft} and \eqref{4}, we see that $M$ is injective on $\M_{E_2,\eps_2}$
provided $\eps_2$ is sufficiently small depending on $E_2$.  By shrinking $\eps_2$, if necessary, we can further ensure that $M(\M_{E_2,\eps_2})$ is contained in a region where $R$ is defined (this relies on all the other parts of the proposition).  It then follows that $M$ is a homeomorphism on $M(\M_{E_2,\eps_2})$ with inverse $R$.\end{proof}

The last result of this section relates the energy and the inverse of the normal form transformation; this will be useful in the proof of Theorem~\ref{thm:wave ops1}.  

\begin{lemma}\label{L:energy2}
Let $\{z_n\}_{n\geq 1}\subset H^1_x$ be uniformly bounded and assume that $z_n\to 0$ in $L_x^6$. Then
\begin{align*}
E(R(z_n))=\tfrac12\|z_n\|_{H^1_x}^2 +  o(1) \qtq{as} n\to \infty. 
\end{align*}
\end{lemma}

\begin{proof}
By Proposition~\ref{prop:normal1} (and its proof), we have that $R(z_n)$ exists for $n$ large and
\begin{align}\label{314}
\limsup_n\|R(z_n)\|_{\E} \lesssim 1 \qtq{and} \lim_{n\to \infty}\|\Re  R(z_n)\|_{L_x^6}=0.
\end{align}

We first claim that
\begin{align}\label{313}
R(z_n)= V^{-1} z_n + o(1) \quad\text{in $\dot H^1_x$ as $n\to \infty$.}
\end{align}
Indeed, from the construction of $R$ via the fixed point argument in Proposition~\ref{prop:normal1}, this amounts to proving that
\begin{align*}
\|\langle \nabla\rangle^{-2}[U^{-1}\Im z_n]^2\|_{\dot H^1_x} + \|\langle \nabla\rangle^{-2}[\Re  R(z_n)]^2\|_{\dot H^1_x} = o(1) \quad\text{as $n\to \infty$.}
\end{align*}
To see this, we use the decomposition
\begin{align}\label{315}
[U^{-1}\Im z_n]^2 = [\PLo U^{-1}\Im z_n]^2 + \text{\O}[(U^{-1}\Im z_n) \PHi  U^{-1}\Im z_n]
\end{align}
together with Bernstein, H\"older, \eqref{314}, and the hypotheses of the lemma to estimate
\begin{align*}
\|\langle \nabla\rangle^{-2}[U^{-1}\Im z_n]^2\|_{\dot H^1_x}
&\lesssim \||\nabla|[\PLo U^{-1}\Im z_n]^2\|_{L_x^2} + \|(U^{-1}\Im z_n) \PHi  U^{-1}\Im z_n\|_{L_x^2}\\
&\lesssim \|\Im z_n\|_{L_x^3}\|U^{-1}\Im z_n\|_{L_x^6} \\
&\lesssim \|z_n\|_{L_x^6}^{\frac12} \|z_n\|_{L_x^2}^{\frac12} \|z_n\|_{H^1_x} = o(1)\quad\text{as $n\to \infty$,}\\
\|\langle \nabla\rangle^{-2}[\Re  R(z_n)]^2\|_{\dot H^1_x}
&\lesssim \||\nabla|^{\frac32}\langle \nabla\rangle^{-2}[\Re  R(z_n)]^2\|_{L_x^{\frac32}}\\
&\lesssim \|\nabla \Re  R(z_n)\|_{L_x^2} \|\Re  R(z_n)\|_{L_x^6}= o(1)\quad\text{as $n\to \infty$.}
\end{align*}
This completes the proof of \eqref{313}.

We now turn our attention to the terms in the formula for $E(R(z_n))$ containing $q(R(z_n))$.  Using the representation \eqref{eq:M rewrite}, we observe that
$$
M(u) = \tfrac12 q(u) -\tfrac12 U^2(|u|^2) +iU\Im u \ \ \text{and so}\ \ q(R(z_n))= 2\Re  z_n + U^2(|R(z_n)|^2).
$$

We next claim that
\begin{align}\label{316}
q(R(z_n))= 2\Re  z_n + o(1) \quad\text{in $L^2_x$ as $n\to \infty$.}
\end{align}
To prove this, we note that $\Im R(z_n) = U^{-1}\Im z_n$ and use the decomposition \eqref{315}, as well as the analogous decomposition for $\Re  R(z_n)$.
Arguing as for \eqref{313}, we estimate
\begin{align*}
\|U^2&[U^{-1}\Im z_n]^2\|_{L^2_x}\\
&\lesssim \||\nabla|[\PLo U^{-1}\Im z_n]^2\|_{L_x^2} + \|(U^{-1}\Im z_n) \PHi  U^{-1}\Im z_n\|_{L_x^2}= o(1),\\
\|U^2&[\Re  R(z_n)]^2\|_{L^2_x}\\
&\lesssim \||\nabla|[\PLo \Re  R(z_n)]^2\|_{L_x^{\frac32}} + \|(\Re  R(z_n))\PHi  \Re  R(z_n)\|_{L_x^2}\\
&\lesssim \|\nabla \Re  R(z_n)\|_{L_x^2} \|\Re  R(z_n)\|_{L_x^6} + \|\Re  R(z_n)\|_{L_x^6} \|\PHi  \Re  R(z_n)\|_{L_x^3} \\
&\lesssim \|\nabla R(z_n)\|_{L_x^2} \|\Re  R(z_n)\|_{L_x^6}= o(1)\quad\text{as $n\to \infty$.}
\end{align*}
This completes the proof of \eqref{316}.

Finally, we note that
\begin{align}\label{317}
q(R(z_n))=o(1) \quad\text{in $L^3_x$ as $n\to \infty$.}
\end{align}
Indeed, arguing as above we find
\begin{align*}
\|\Re  z_n\|_{L_x^3} &\lesssim \|z_n\|_{L_x^2}^{\frac12} \|z_n\|_{L_x^6}^{\frac12} =  o(1)\quad\text{as $n\to \infty$,}\\
\|U^2[\Re  R(z_n)]^2\|_{L_x^3}&\lesssim \|\Re R(z_n)\|_{L_x^6}^2= o(1)\quad\text{as $n\to \infty$,}\\
\|U^2[\Im R(z_n)]^2\|_{L^3_x}
&\lesssim \||\nabla|[\PLo U^{-1}\Im z_n]^2\|_{L_x^3} + \|(U^{-1}\Im z_n) \PHi  U^{-1}\Im z_n\|_{L_x^3}\\
&\lesssim \|\Im z_n\|_{L_x^6} \||\nabla|^{-1} \Im z_n\|_{L_x^6} + \|U^{-1}\Im z_n\|_{L_x^6} \|\Im z_n\|_{L_x^6}\\
&\lesssim \|z_n\|_{L_x^6} \|z_n\|_{H^1_x}= o(1)\quad\text{as $n\to \infty$.}
\end{align*}

Putting together \eqref{313}, \eqref{316}, and \eqref{317} completes the proof of the lemma.
\end{proof}

					%%%%%%%
					\section{Proof of Theorem~\ref{thm:wave ops1}}\label{section:fs1}		
In this section we prove Theorem~\ref{thm:wave ops1}. To this end, we fix $u_+\in\Hr$. We define $z_+=Vu_+\in H_x^1$ and we let $E_0:=\|z_+\|_{H_x^1}.$
	
We first claim that 
\begin{equation}\label{eq:L6 to zero}
\lim_{t\to\infty} \|e^{-itH}z_+\|_{L_x^6}= 0.
\end{equation}
Indeed, given $\eta>0$ we may find $\varphi\in \mathcal{S}(\R^3)$ such that $\|z_+-\varphi\|_{\dot{H}_x^1}<\eta$. Using the dispersive estimate \eqref{eq:dispersive} and Sobolev embedding, we find	
$$
\|e^{-itH}z_+\|_{L_x^6}\lesssim\| e^{-itH}\varphi\|_{L_x^6}+\|z_+-\varphi\|_{\dot{H}_x^1}\lesssim\vert t\vert^{-1}\|\varphi\|_{L_x^{6/5}}+\eta,
$$
which yields \eqref{eq:L6 to zero}.

Next, we choose $\eps_0$ sufficiently small depending on $E_0$ as in Proposition~\ref{prop:normal1}. By \eqref{eq:L6 to zero}, there exists $T_0>0$ such that $e^{-itH}z_+\in\mathcal{N}_{E_0,\eps_0}$ for $t\geq T_0$; thus for any $T\geq T_0$, we may define $R(e^{-iTH}z_+)\in\E$ so that $M(R(e^{-iTH}z_+))=e^{-iTH}z_+$, with $\|R(e^{-iTH}z_+)\|_{\E}\lesssim E_0.$ 

By Theorem~\ref{thm:gwp}, there exists a global solution $u^T\in C(\R;\E)$ to \eqref{eq:cq3} with $u^T(T)=R(e^{-iTH}z_+)$. Note that when $\gamma\in(0,\tfrac23)$, we require $E_0$ to be sufficiently small to guarantee that
\begin{equation}\label{eq:uniform bounds2}
\|\nabla\Re (u^T(0))\|_{L_x^2}^2\leq\delta_\gamma\quad\text{and}\quad E(u^T(0))\leq \tfrac14\delta_\gamma
\end{equation}
uniformly in $T$, where $\delta_\gamma$ is as in Theorem~\ref{thm:gwp}. We define 
$$
%\psi^T:=1+u^T,\quad
q^T:=q(u^T)=2u_1^T+\vert u^T\vert^2 \qtq{and} z^T:=M(u^T).
$$ 
Note that %$\psi^T$ solves \eqref{eq:cq2} and 
$(u^T,z^T)$ solves \eqref{eq:nft}--\eqref{eq:cq z} with $z^T(T)=e^{-iTH}z_+$. Furthermore, we have 
\begin{equation}\label{eq:uniform bounds}
\|z^T(t)\|_{H_x^1}+\|u^T(t)\|_{\dot{H}_x^1}+\|q^T(t)\|_{L_x^2\cap L_x^3}+\|u_1^T(t)\|_{L_x^3\cap L_x^6}\lesssim_{E_0} 1,
\end{equation}
uniformly in $t$ and $T$. 

As a consequence of \eqref{eq:uniform bounds}, there exists a sequence $T_n\to \infty$ and a function $u_0\in \dot H^1_x$ such that $u^{T_n}(0)\rightharpoonup u_0$ weakly in $\dot H^1_x$.  As \eqref{eq:uniform bounds} and \eqref{eq:uniform bounds2} imply that $\{u^{T_n}(0)\}$ satisfy the hypotheses of Theorem~\ref{T:weak}, we may apply this theorem to deduce that
\begin{align}\label{1200}
u^{T_n}(t) \rightharpoonup u^\infty(t) \quad \text{weakly in $\dot H^1_x$ for all $t\in \R$,}
\end{align}
where $u^\infty\in C(\R;\E)$ denotes the solution to \eqref{eq:cq3} with initial data $u^\infty(0)=u_0\in\E$.

We define $z^\infty:=M(u^\infty)$ and note that $(u^\infty,z^\infty)$ solves \eqref{eq:nft}--\eqref{eq:cq z}.  We will prove that $u^\infty$ is a solution to \eqref{eq:cq3} that satisfies the conclusions of Theorem~\ref{thm:wave ops1}. A first step in this direction is the following weak convergence result.

\begin{proposition}\label{P:weak convergence}  We have
$$
e^{itH}z^\infty(t)\rightharpoonup z_+ \qtq{weakly in $H_x^1$ as $t\to \infty$.}
$$
\end{proposition}

Assuming Proposition~\ref{P:weak convergence} for now, we proceed with the proof of Theorem~\ref{thm:wave ops1}.  We begin by upgrading the weak convergence from Proposition~\ref{P:weak convergence} to strong convergence, namely,
\begin{align}\label{1}
\lim_{t\to \infty}\|z^\infty(t) - e^{-itH}z_+\|_{H^1_x}=0.
\end{align}

Using Lemma~\ref{L:bounds survive} combined with \eqref{1200} and Lemma~\ref{L:energy2} combined with \eqref{eq:L6 to zero}, we can first write
\begin{align}\label{957}
E(u^\infty) \leq \liminf_{n\to\infty} E(u^{T_n}) = \liminf_{n\to\infty} E(R(e^{-iT_nH}z_+)) = \tfrac12 \|z_+\|_{H^1_x}^2.
\end{align}

At this moment, it is tempting to attempt a Radon--Riesz style argument.  Recall that the Radon--Riesz theorem says that if $x_n\rightharpoonup x$
weakly in some Banach space $X$ and
$
\limsup F(x_n)  \leq F(x)  
$
for some uniformly convex function $F:X\to \R$, then $x_n\to x$ in norm.  (This is most often quoted in the case of a uniformly convex Banach space with $F$ being the norm.)

The ideas just sketched were adapted beautifully to the Gross--Pitaevskii setting treated in \cite{GNT:3d}.  As discussed in the introduction, those authors exploit
$$
E_{GP}(u) = \tfrac12 \| M(u) \|_{H^1_x}^2 + \tfrac14 \|U|u|^2\|_{L_x^2}^2 \geq \tfrac12 \| M(u) \|_{H^1_x}^2,
$$
which holds under no additional hypotheses.  As also discussed there (see \eqref{E:CQEident}, in particular) the energy functional for the cubic-quintic problem admits no such global inequality.  Correspondingly, we need to keep track of the structure of $z^\infty(t_n)$ as $t_n\to\infty$ and then demonstrate the requisite coercivity is available in this particular limiting regime.   To achieve the this goal we will use the following lemma.  Note that the result on $\tilde E$ plays a key role in controlling the kinetic energy of the real part when $\gamma < \frac23$.

\begin{lemma}\label{L:energy1}
Let $\{u_n\}_{n\geq 1}\subset \E$ be uniformly bounded.  Assume that we may write $u_n=\xi_n+r_n$, where $\xi_n$ satisfies
\begin{align*}
\sup_n\|\xi_n\|_{\Hr}\lesssim 1 \qtq{and} \lim_{n\to \infty}\|\xi_n\|_{L^3_x\cap L_x^6}=0.
\end{align*}
Then
\begin{align}\label{955}
E(u_n) =E(r_n) + \tfrac12\|V\xi_n\|_{H^1_x}^2 + \Re  \langle M(r_n), V\xi_n\rangle_{H^1_x} + o(1) \qtq{as} n\to \infty. 
\end{align}

Furthermore, if $\tilde E$ denotes the reduced energy defined via
$$
\tilde E(f) : = \int \tfrac14 |\nabla f|^2 + \tfrac\gamma8 |q(f)|^2\, dx = \tfrac12 E(f) - \tfrac1{12} \int q(f)^3\, dx,
$$
then
\begin{align}\label{956}
\tilde E(u_n) =\tilde E(r_n) + \tfrac14\|V\xi_n\|_{H^1_x}^2 + \tfrac12\Re  \langle M(r_n), V\xi_n\rangle_{H^1_x} + o(1) \qtq{as} n\to \infty. 
\end{align}
\end{lemma}

\begin{proof}
We will only prove \eqref{955}.  Claim \eqref{956} can be read off from the proof we give below.

To begin we observe that
$$
q(u_n) = q(r_n) + 2\Re  \xi_n + |\xi_n|^2+ 2\Re  (\bar \xi_n r_n).
$$
By hypothesis, $r_n=u_n-\xi_n$ is uniformly bounded in $L^6_x$.  Using this and our assumptions on $\xi_n$, we see that
\begin{align*}
q(u_n) &= q(r_n) + 2\Re  \xi_n +o(1) \quad\text{in $L^2_x$ as $n\to \infty$},\\
q(u_n) &= q(r_n) +o(1) \quad\text{in $L^3_x$ as $n\to \infty$}.
\end{align*}
Moreover, as $u_n$ is bounded in $\E$ and $\Re  \xi_n$ is bounded in $L_x^2$, we deduce that $q(r_n)$ is uniformly bounded in both $L_x^2$ and $L_x^3$.

Therefore, we obtain
\begin{align*}
E(u_n) &=E(r_n) + \int\tfrac12 |\nabla \xi_n|^2+ \Re (\nabla \bar\xi_n \nabla r_n) + \gamma q(r_n)\Re  \xi_n + \gamma(\Re  \xi_n)^2 \, dx + o(1)\\
&=E(r_n) + \tfrac12\|V\xi_n\|_{H^1_x}^2 + \Re  \langle (2\gamma-\Delta)M(r_n), V\xi_n\rangle_{L^2_x}+ o(1)\qtq{as} n\to \infty.
\end{align*}

This completes the proof of the lemma.
\end{proof}

We return now to the proof of \eqref{1}.  Let us begin by showing that 
\begin{align}\label{958}
E(u^\infty) \geq \tfrac12 \|z_+\|_{H^1_x}^2,
\end{align}
which combined with \eqref{957} fully identifies $E(u^\infty)$.  While natural, this is not (in and of itself) essential to the argument; it does, however, force us to control the contributions of parts of the energy with the unhelpful sign.  It will be this control that will ultimately allow us to complete the proof of \eqref{1}.

Let $t_n\to \infty$ be an arbitrary sequence.  We apply Lemma~\ref{L:energy1} with
\begin{align}\label{6}
u_n:= u^\infty(t_n) \qtq{and} \xi_n:= (Id \oplus P_{\geq N_n}) V^{-1} e^{- it_n H}z_+,
\end{align}
where $N_n\in 2^{\mathbb{Z}}$ converges to zero sufficiently slowly to guarantee that
\begin{align}\label{xi small}
\|\xi_n\|_{L_x^3\cap L_x^6} \to 0 \quad\text{as $n\to \infty$.}
\end{align}
Note that this is possible because of \eqref{eq:L6 to zero}.  In view of \eqref{955}, we obtain
\begin{align}\label{959}
E(u^\infty) &= E(r_n) + \tfrac12 \| (Id \oplus P_{\geq N_n}) e^{- it_n H}z_+\|_{H^1_x}^2 \notag\\
&\quad+ \Re  \langle M(r_n), (Id \oplus P_{\geq N_n}) e^{- it_n H}z_+\rangle_{H^1_x} +o(1) \quad\text{as $n\to \infty$.}
\end{align}

By Proposition~\ref{P:weak convergence}, $e^{it_nH}M(u^\infty(t_n))= e^{it_nH}z^\infty(t_n) \rightharpoonup z_+$ weakly in $H^1_x$.  On the other hand, by \eqref{xi small}, we have
\begin{align}
M(u^\infty(t_n))\label{962}
&= e^{-it_nH}z_+\! + M(r_n) - P_{\leq N_n} \Im e^{-it_nH}z_+\! + \gamma  \langle\nabla \rangle^{-2}[ |\xi_n|^2 + 2\Re  (\bar\xi_nr_n)]\notag\\
&=e^{-it_nH}z_+\! + M(r_n) + o(1) \quad\text{in $H^1_x$ as $n\to \infty$.}
\end{align}
Thus, we may deduce that
$$
e^{it_nH}M(r_n) \rightharpoonup 0 \quad\text{weakly in $H^1_x$ as $n\to \infty$.}
$$
Combining this with the dominated convergence theorem (which allows us to replace $P_{\geq N_n}$ by $Id$), \eqref{959} becomes
\begin{align}\label{960}
E(u^\infty) = E(r_n) + \tfrac12 \| z_+\|_{H^1_x}^2 +o(1) \quad\text{as $n\to \infty$.}
\end{align}

Arguing similarly and using \eqref{956} in place of \eqref{955}, we obtain
\begin{align}\label{961}
\tilde E(u^\infty) = \tilde E(r_n) + \tfrac14 \| z_+\|_{H^1_x}^2 +o(1) \quad\text{as $n\to \infty$.}
\end{align}

Note that \eqref{958} follows immediately from \eqref{960}, provided that $E(r_n)\geq 0$.  By Lemma~\ref{lemma:coercive1}, this is immediate if $\gamma\in [\frac23,1)$.  In view of Lemma~\ref{lemma:coercive2}, if $\gamma\in (0,\frac23)$ we simply have to verify that $\|\nabla \Re  r_n\|_{L_x^2}^2\leq \delta_\gamma$.  This however follows from \eqref{957} and \eqref{961}, provided $E_0$ is chosen sufficiently small depending on $\gamma$.

Combining \eqref{957} with \eqref{958} and \eqref{960}, we deduce that
\begin{align*}
E(u^\infty) = \tfrac12 \| z_+\|_{H^1_x}^2 \qtq{and} E(r_n)\to 0 \text{ as $n\to \infty$.}
\end{align*}
By the argument in the preceding paragraph, this implies
\begin{align}\label{7}
\|r_n\|_{\E}\to 0 \quad\text{as $n\to \infty$.}
\end{align}
Therefore, using the representation \eqref{eq:M rewrite} for $M$, we see that
\begin{align*}
\|M(r_n)\|_{H^1_x}
&\lesssim \|U\Im r_n\|_{H^1_x} + \|U^2 \Re  r_n\|_{H^1_x} + \| \langle\nabla \rangle^{-2}q(r_n)\|_{H^1_x} \\
&\lesssim \|r_n\|_{\dot H^1_x} + \|q(r_n)\|_{L_x^2}\to 0 \quad\text{as $n\to \infty$.}
\end{align*}
Combining this with \eqref{962}, we get
\begin{align*}
\|z^\infty(t_n) - e^{-it_nH} z_+\|_{H^1_x} \to 0 \quad \text{as $n\to \infty$.}
\end{align*}
As the sequence $t_n\to \infty$ was arbitrary, this completes the proof of \eqref{1}.

We next prove that \eqref{1} implies the conclusions of Theorem~\ref{thm:wave ops1}.  We first show that \eqref{1} implies \eqref{E:T:H1}.
Let $t_n\to\infty$ be an arbitrary sequence and define $u_n$ and $\xi_n$ as in \eqref{6}.  Using \eqref{xi small} and \eqref{7}, we deduce that $u^\infty(t_n)\to 0$ in $L_x^6$.  Furthermore, by \eqref{1} and \eqref{eq:L6 to zero}, we have that $z^\infty(t_n)\to 0$ in $L_x^6$.  Using Proposition~\ref{prop:normal1}(v), we find that $u^\infty(t_n)=R(z^\infty(t_n))$ for $n$ sufficiently large.  Arguing as in Lemma~\ref{L:energy2} and using \eqref{313}, we may write $u^\infty(t_n)=V^{-1}z^\infty(t_n)+o(1)$ in $\dot{H}_x^1$, which together with \eqref{1} yields \eqref{E:T:H1}.

We now turn to \eqref{E:T:E}.  We begin with the following strengthening of \eqref{eq:L6 to zero}:
\begin{align}\label{8}
\lim_{t\to\infty}\|U^{-1}e^{-itH}z_+\|_{L_x^6}=0.
\end{align}
Given $0<N<1$, we have
\begin{align*}
&\|U^{-1}P_{\leq N} e^{-itH}z_+\|_{L_x^6}\lesssim \|\nabla U^{-1}P_{\leq N}e^{-itH}z_+\|_{L_x^2}\lesssim\|P_{\leq N}z_+\|_{L_x^2},\\
&\|U^{-1}P_{>N}e^{-itH}z_+\|_{L_x^6}\lesssim N^{-1}\|e^{-itH}z_+\|_{L_x^6}.
\end{align*}
In view of \eqref{eq:L6 to zero}, choosing $N$ sufficiently small and then sending $t\to\infty$ yields \eqref{8}. 

Using \eqref{8}, we now show that the modification $\gamma\jb^{-2}\vert\ulin\vert^2$ appearing in \eqref{E:T:E} is negligible in the $\dot{H}_x^1$-norm. Indeed, we have the stronger statement
\begin{align}\label{10}
\|\jb^{-1}\vert\ulin(t)\vert^2\|_{\dot{H}_x^1}&\lesssim \| |\nabla|^{1/2}\ulin(t)\|_{L_x^3}\|\ulin(t)\|_{L_x^6} \notag\\
&\lesssim \|\nabla \ulin(t)\|_{L_x^2}\|U^{-1}e^{-itH}z_+\|_{L_x^6} \notag\\
&\lesssim \|z_+\|_{H_x^1}\|U^{-1}e^{-itH}z_+\|_{L_x^6}\to 0\qtq{as}t\to\infty.
\end{align}

It remains to show 
\begin{align}\label{9}
\lim_{t\to\infty}\bigl\|q\bigl(u^\infty(t)\bigr)-q\bigl(\ulin(t)-\gamma\jb^{-2}\vert\ulin(t)\vert^2\bigr)\bigr\|_{L_x^2}=0.
\end{align}
As demonstrated above, $u^\infty (t) = R(z^\infty (t))$ for $t$ sufficiently large and $z^\infty(t)\to 0$ in $L_x^6$ as $t\to \infty$.  Thus, arguing as for \eqref{316} and using \eqref{1}, we deduce that
$$
q(u^\infty(t))=2\Re \ulin(t) +o(1) \quad\text{in $L_x^2$ as $t\to \infty$}.
$$
 On the other hand, a straightforward computation yields
\begin{align*}
q\bigl(\ulin(t)-\gamma\jb^{-2}\vert\ulin(t)\vert^2\bigr) &= 2\Re \ulin(t) + U^2|\ulin(t)|^2 + \bigl[\gamma \jb^{-2}\vert\ulin(t)\vert^2\bigr]^2 \\
&\quad -2\gamma\bigl[\jb^{-2}|\ulin(t)|^2 \bigr] \Re \ulin(t).
\end{align*}
Thus, to prove \eqref{9} it suffices to show that the last three terms on the right-hand side above are $o(1)$ in $L_x^2$ as $t\to \infty$.  Indeed, we may estimate
\begin{align*} 
\|U^2|\ulin(t)|^2\|_{L_x^2} &\lesssim \|\jb^{-1}\vert\ulin(t)\vert^2\|_{\dot{H}_x^1},\\
\|\bigl[\jb^{-2}\vert\ulin(t)\vert^2\bigr]^2\|_{L_x^2}&\lesssim \||\nabla|^{\frac14}\jb^{-2}|\ulin(t)|^2\|_{L_x^3}^2
\lesssim \|U^{-1}e^{-itH}z_+\|_{L_x^6}^4,\\
\|\bigl[\jb^{-2}|\ulin(t)|^2 \bigr] \Re \ulin(t)\|_{L^2_x}&\lesssim \|U^{-1}e^{-itH}z_+\|_{L_x^6}^3,
\end{align*}
and so by \eqref{8} and \eqref{10}, we have
$$
q\bigl(\ulin(t)-\gamma\jb^{-2}\vert\ulin(t)\vert^2\bigr) = 2\Re \ulin(t)+o(1) \quad\text{in $L_x^2$ as $t\to \infty$}.
$$
This completes the proof of \eqref{9} and hence that of Theorem~\ref{thm:wave ops1}.

It remains to prove Proposition~\ref{P:weak convergence}.

\begin{proof}[Proof of Proposition~\ref{P:weak convergence}]
We first claim that
\begin{align}\label{1201}
z^{T_n}(t) \rightharpoonup z^\infty(t) \quad \text{weakly in $\dot H^1_x$ for all $t\in \R$.}
\end{align}
This relies in an essential way on Theorem~\ref{T:weak} via \eqref{1200}. Henceforth, we let $t\in\R$ be fixed. Using \eqref{1200} and Rellich--Kondrashov and passing to a subsequence, we have $u^{T_n}(t) \to u^\infty(t)$ strongly in $L^2_x(K)$ for any compact $K\subset \R^3$.  Now fix $\varphi\in C^\infty_c(\R^3)$.  Then $\langle \nabla \rangle^{-2} \varphi \in \mathcal S(\R^3)$; in particular, for any $\eps>0$ there exists $\tilde \varphi_\eps\in C^\infty_c(\R^3)$ such that
$$
\|\langle \nabla \rangle^{-2} \varphi -\tilde\varphi_\eps\|_{L_x^{3/2}}\leq \eps.
$$
Using this, H\"older, \eqref{eq:uniform bounds}, and \eqref{1200}, we obtain
\begin{align*}
\langle z^{T_n}(t), \varphi\rangle 
&= \langle u^{T_n}(t), V\varphi\rangle + \gamma\langle |u^{T_n}(t)|^2, \tilde\varphi_\eps\rangle +\gamma \langle |u^{T_n}(t)|^2, \langle \nabla \rangle^{-2} \varphi -\tilde\varphi_\eps\rangle\\
&=\langle u^{T_n}(t), V\varphi\rangle + \gamma\langle |u^{T_n}(t)|^2, \tilde\varphi_\eps\rangle + O(\eps)\\
&\to \langle u^{\infty}(t), V\varphi\rangle + \gamma\langle |u^{\infty}(t)|^2, \tilde\varphi_\eps\rangle + O(\eps) = \langle z^{\infty}(t), \varphi\rangle +O(\eps).
\end{align*}
As $\eps>0$ was arbitrary, this proves \eqref{1201}.

To continue, we write
\begin{align*}
e^{itH} z^\infty(t) - z_+ &= [e^{itH} z^\infty(t) - e^{iT_0H} z^\infty(T_0)] + [e^{iT_0H} z^\infty(T_0) - e^{iT_0H} z^{T_n}(T_0)]\\
&\quad + [e^{iT_0H} z^{T_n}(T_0) - e^{iT_nH} z^{T_n}(T_n)].
\end{align*}

As the above is bounded in $H^1_x$, it suffices to prove weak convergence when testing against a dense set of functions in $H^{-1}_x$.  In this role, we take $\varphi\in \mathcal S(\R^3)$ with $\hat \varphi\in C^\infty_c(\R^3\setminus\{0\})$.  To continue, we choose $N_0\in 2^{\mathbb Z}$ so that $\supp\hat \varphi\subset\{|\xi|\geq 100N_0\}$ and fix $\eps>0$.  By \eqref{1201}, there exists $n$ sufficiently large (depending on $T_0$) so that
\begin{align}\label{117}
|\langle e^{iT_0H} z^\infty(T_0) - e^{iT_0H} z^{T_n}(T_0), \varphi\rangle| \leq \eps.
\end{align}
To handle the remaining two differences, we will prove the following inequality:
\begin{align}\label{118}
|\langle e^{it_2H} z(t_2) - e^{it_1H} z(t_1), \varphi\rangle|\lesssim_\varphi |t_1|^{-\frac14}
\end{align} 
uniformly for $t_2>t_1$, where $z$ denotes any of the functions $z^{T_n}$.  In view of \eqref{1201}, we see that \eqref{118} also holds (with the same implicit constant) for $z=z^\infty$.  Thus, taking $T_0$ large enough depending on $\eps$ and then $n$ large enough so that $T_n>T_0$ and \eqref{117} holds, we get
$$
\sup_{t>T_0}|\langle e^{itH} z^\infty(t) - z_+, \varphi\rangle|\lesssim_\varphi \eps.
$$
As $\eps>0$ was arbitrary, this proves Proposition~\ref{P:weak convergence}.

It remains to verify \eqref{118}. By Duhamel's formula, we have
\begin{align}\label{eq:equicts1}
|\langle e^{it_2H} z(t_2) - e^{it_1H} z(t_1), \varphi\rangle| \leq \int_{t_1}^{t_2} |\langle N_z(u(s)), e^{-isH}\varphi\rangle|\, ds.
\end{align}
To continue, we decompose the nonlinearity as follows:
$$
N_z(u)=N_z^1(u) - \gamma N_z^2(u),
$$
where
\begin{align*}
N_z^1(u)&=U\bigl[\tfrac{\gamma+2}2q^2+q^2u_1-\gamma u_1^3-\tfrac\gamma2|u|^4\bigr] \nonumber\\
&\quad -2\gamma i\jb^{-2}\nabla\cdot[q\nabla  u_2-u_1^2\nabla u_2]+iU^2[\gamma u_1^2u_2 +q^2u_2], %\label{eq:Nz easy}
\\
N_z^2(u)&= U(u_1u_2^2)-\tfrac{i}3U^2(u_2^3). %\label{eq:Nz hard}
\end{align*}

We first estimate the contribution of $N_z^1(u)$ to \eqref{eq:equicts1}.  By H\"older and the dispersive estimate \eqref{eq:dispersive}, we can estimate
\begin{align*}
\int_{t_1}^{t_2}\big\vert\langle N_z^1(u(s)),e^{-isH}\varphi\rangle\big\vert\,ds&\lesssim\int_{t_1}^{t_2}\|N_z^1(u(s))\|_{L_x^{12/11}}\|e^{-isH}\varphi\|_{L_x^{12}}\,ds\\ &\lesssim_\varphi\int_{t_1}^{t_2}\vert s\vert^{-5/4}\|N_z^1(u(s))\|_{L_x^{12/11}}\,ds.
\end{align*}
Most of the terms appearing in $\Re(N_z^1)$ can be handled using H\"older and \eqref{eq:uniform bounds}:  
\begin{align*}
&\bigl\|U\bigl[\tfrac{\gamma+2}2q^2+q^2u_1-\gamma u_1^3\bigr]\bigr\|_{L_x^{12/11}}
\lesssim\|q\|_{L_x^{24/11}}^2+\|q\|_{L_x^{8/3}}^2\|u_1\|_{L_x^6}+\|u_1\|_{L_x^{36/11}}^3\lesssim 1.
\end{align*}
To estimate the remaining term in $\Re(N_z^1)$ we also use  the fractional chain rule and Sobolev embedding, as follows:
\begin{align*}
\|U(|u|^4)\|_{L_x^{12/11}}
&\lesssim \||\nabla|^{\frac34} (|u|^4)\|_{L_x^{12/11}}\lesssim \||\nabla|^{\frac34} u\|_{L_x^{12/5}}\|u\|_{L_x^6}^3\lesssim\|\nabla u\|_{L_x^2}^4\lesssim 1.
	\end{align*}
To estimate the terms in $\Im(N_z^1)$, we use H\"older and \eqref{eq:uniform bounds}: 
\begin{align*}
\|\jb^{-2}\nabla\cdot[q\nabla  u_2-u_1^2\nabla u_2]\|_{L_x^{12/11}}
&\lesssim\|\nabla u_2\|_{L_x^2}\|q\|_{L_x^{12/5}}+\|\nabla u_2\|_{L_x^2}\|u_1\|_{L_x^{24/5}}^2\lesssim 1,\\
\|U^2[\gamma u_1^2u_2 +q^2u_2] \|_{L_x^{12/11}}&\lesssim \|\nabla (u_1^2u_2)\|_{L_x^{12/11}} + \|q^2u_2\|_{L_x^{12/11}}\\
&\lesssim \|u_1\|_{L_x^{4}}\|u\|_{L_x^6}\|\nabla u\|_{L_x^2} + \|q\|_{L_x^{8/3}}^2\|u_2\|_{L_x^6}\lesssim 1.
\end{align*}	
Putting everything together, we find
\begin{align*}
\int_{t_1}^{t_2}\big\vert\langle N_z^1(u(s)),e^{-isH}\varphi\rangle\big\vert\,ds &\lesssim_\varphi|t_1|^{-\frac14}.
\end{align*}
	
We turn now to estimating the contribution of $N_z^2(u)$ to \eqref{eq:equicts1}.  To complete the proof of \eqref{118} and so that of the proposition, we must show that
\begin{align}\label{Nz2 goal}
\int_{t_1}^{t_2}\big\vert\langle N_z^2(u(s)),e^{-isH}\varphi\rangle\big\vert\,ds &\lesssim_\varphi|t_1|^{-\frac14}.
\end{align}
Recalling that $\supp\hat \varphi\subset\{|\xi|\geq 100N_0\}$, we see that
$$
\langle P_{\leq 20N_0}N_z^2(u(s)),e^{-isH}\varphi\rangle \equiv 0.
$$
Writing $u_2= P_{\leq N_0} u_2 + P_{>N_0}u_2$, we may decompose the remaining part of $N_z^2(u)$ as follows:
\begin{align*}
P_{>20N_0}N_z^2(u)
&= P_{>20N_0}U\text{\O}\big(u_1u_2 P_{> N_0}u_2\big) + P_{>20N_0}U^2\text{\O}\big(u_2 [P_{>N_0}u_2]^2\big) \\
&\quad + P_{>20N_0}U\bigl\{[P_{> 8N_0}u_1][P_{\leq N_0}u_2]^2-iU\big([P_{> 8N_0}u_2][P_{\leq N_0}u_2]^2\big)\bigr\}.
\end{align*}
Writing $u_1= \bar v + iUu_2$ (with $v=Vu$) and $a:=[P_{\leq N_0}u_2]^2$, we arrive at the following decomposition:
\begin{align}
P_{>20N_0}N_z^2(u)
&=P_{>20N_0}U\text{\O}\big(u_1u_2 P_{> N_0}u_2\big) + P_{>20N_0}U^2\text{\O}\big(u_2 [P_{>N_0}u_2]^2\big)\label{eq:Nz hard1}\\
&\quad + iP_{>20N_0}U\bigl\{aU(P_{> 8N_0}u_2)-U\big(aP_{> 8N_0}u_2\big)\bigr\}\label{eq:Nz hard2}\\
&\quad +P_{>20N_0}U\bigl\{a P_{> 8N_0}\bar v] \label{eq:Nz hard3}.
\end{align}

As we will see, the terms in \eqref{eq:Nz hard1} and \eqref{eq:Nz hard2} are small. However, there is no reason to believe that \eqref{eq:Nz hard3} is small pointwise in time; instead, we will show that this term is non-resonant. 

We first consider \eqref{eq:Nz hard1}.  Using H\"older, Bernstein, and \eqref{eq:uniform bounds}, we estimate 
\begin{align*}
\|\eqref{eq:Nz hard1}\|_{L_x^{12/11}}
&\lesssim \|u_1\|_{L_x^4}\|u_2\|_{L_x^6}\|P_{>N_0}u_2\|_{L_x^2} + \|u_2\|_{L_x^6} \|P_{>N_0}u_2\|_{L_x^{8/3}}^2\\
&\lesssim_{N_0} \|u_1\|_{L_x^4}\|\nabla u_2\|_{L_x^2}^2 + \|\nabla u_2\|_{L_x^2}^3\lesssim_{N_0} 1.
\end{align*}
Thus, the contribution of this term to \eqref{Nz2 goal} is acceptable. 

We now turn to \eqref{eq:Nz hard2}, which includes the commutator $[a,U]$.  We regard this term as a bilinear operator $T(a, P_{> 8N_0} u_2)$ with symbol given by
\begin{align*}
\bigl[1- \phi\bigl(\tfrac{\xi}{20N_0}\bigr) \bigr]&U(\xi) \bigl[U(\xi_2)- U(\xi)\bigr]\phi\bigl(\tfrac{\xi_1}{4N_0} \bigr)\\
& = \Bigl\{-2\gamma \frac{U(\xi)(\xi_1+2\xi_2)}{\langle \xi\rangle \langle \xi_2\rangle(|\xi_2|\langle \xi\rangle+ |\xi|\langle \xi_2\rangle)}\bigl[1- \phi\bigl(\tfrac{\xi}{20N_0}\bigr) \bigr]\phi\bigl(\tfrac{\xi_1}{4N_0} \bigr)\Bigr\} \cdot \xi_1,
\end{align*}
where $\phi$ denotes the standard Littlewood--Paley multiplier.  Observing that the multiplier inside the braces is amenable to Lemma~\ref{L:CM}, we may estimate
\begin{align*}
\|\eqref{eq:Nz hard2}\|_{L_x^{12/11}}
&\lesssim \|\nabla a\|_{L_x^{3/2}}\|P_{> 8N_0} u_2\|_{L_x^4}
\lesssim_{N_0} \|\nabla P_{\leq N_0} u_2\|_{L_x^2} \|P_{\leq N_0}u_2\|_{L_x^6} \|\nabla u_2\|_{L_x^2}.
\end{align*}
In view of \eqref{eq:uniform bounds}, the contribution of this term to \eqref{Nz2 goal} is acceptable.

Finally, we consider \eqref{eq:Nz hard3}.  Using equation \eqref{eq:cq v}, we find
\begin{align*}
i\partial_t\langle aP_{>8N_0}\bar v, e^{-itH} \tfrac{U}{2H}\varphi\rangle
&= \langle U(aP_{>8N_0}\bar v), e^{-itH} \varphi\rangle + \langle aP_{>8N_0}\overline{N_v(u)}, e^{-itH} \tfrac{U}{2H}\varphi\rangle\\
&\quad  + \langle [a,H] P_{>8N_0}\bar v, e^{-itH} \tfrac{U}{2H}\varphi\rangle + i\langle \dot{a}P_{>8N_0}\bar v, e^{-itH} \tfrac{U}{2H}\varphi\rangle.
\end{align*}
By the fundamental theorem of calculus, we may thus estimate the contribution of \eqref{eq:Nz hard3} to \eqref{Nz2 goal} as follows:
\begin{align}\label{404}
&\int_{t_1}^{t_2} \big|\langle U(aP_{>8N_0}\bar v), e^{-isH} \varphi\rangle\big|\, ds\\
&\lesssim \sup_{t\geq t_1} \big|\langle aP_{>8N_0}\bar v, e^{-itH} \langle \nabla\rangle^{-2}\varphi\rangle\big| + \int_{t_1}^{t_2} \big|\langle aP_{>8N_0}\overline{N_v(u)}, e^{-isH} \langle \nabla\rangle^{-2}\varphi\rangle\big|\,ds\notag\\
&\quad+\int_{t_1}^{t_2} \big|\big\langle  \langle \nabla\rangle^{-2} \bigl([a,H] P_{>8N_0}\bar v\bigr), e^{-isH}\varphi\big\rangle\big|\,ds+\int_{t_1}^{t_2} \big|\langle \dot{a}P_{>8N_0}\bar v, e^{-isH} \langle \nabla\rangle^{-2}\varphi\rangle\big|\,ds.\notag
\end{align}

To estimate the terms on RHS\eqref{404}, we note that in view of \eqref{eq:uniform bounds},
\begin{align}\label{405}
\|\nabla a(t)\|_{L_x^{3/2}\cap L_x^\infty} + \| a(t)\|_{L_x^3\cap L_x^\infty} + \|\nabla v(t)\|_{L_x^2} + \|v(t)\|_{L_x^3\cap L_x^6}\lesssim_{N_0} 1
\end{align}
uniformly for $t\in\R$.
Using this, H\"older, and \eqref{eq:dispersive}, we estimate the first term on RHS\eqref{404} as follows:
\begin{align*}
\sup_{t\geq t_1} \big|\langle aP_{>8N_0}\bar v, e^{-itH} \langle \nabla\rangle^{-2}\varphi\rangle\big|
&\lesssim \|a\|_{L_t^\infty L_x^3} \|v\|_{L_t^\infty L_x^3} \sup_{t\geq t_1}\|e^{-itH}\langle \nabla\rangle^{-2}\varphi\|_{L_x^3} \\
&\lesssim_{\varphi,N_0} |t_1|^{-\frac12}. 
\end{align*}
Thus, the contribution of this term to \eqref{Nz2 goal} is acceptable.

Next we consider the second term on RHS\eqref{404}. By H\"older and \eqref{eq:dispersive},
\begin{align*}
&\int_{t_1}^{t_2} \big|\langle aP_{>8N_0}\overline{N_v(u)}, e^{-isH} \langle \nabla\rangle^{-2}\varphi\rangle\big|\,ds
\lesssim_\varphi |t_1|^{-\frac12}\|aP_{>8N_0}N_v(u)\|_{L_t^\infty L_x^1}.
\end{align*}
Note that
$$
N_v(u) = \sum_{k=2}^5 U\text{\O}(u^k)+\text{\O}(u^k).
$$
To estimate the contribution of the quadratic terms in $N_v(u)$, we use Bernstein, \eqref{eq:uniform bounds}, and \eqref{405}:
\begin{align*}
\bigl\|aP_{>8N_0}\bigl[U\text{\O}(u^2)+\text{\O}(u^2)\bigr]\bigr\|_{L_t^\infty L_x^1}
&\lesssim_{N_0} \| a\|_{L_t^\infty L_x^3} \|\nabla \text{\O}(u^2)\|_{L_t^\infty L_x^{3/2}}\\
&\lesssim_{N_0} \| a\|_{L_t^\infty L_x^3} \|\nabla u\|_{L_t^\infty L_x^2} \|u\|_{L_t^\infty L_x^6}\lesssim_{N_0}1.
\end{align*}
Similarly, we can estimate the cubic terms in $N_v(u)$ via
\begin{align*}
\bigl\|aP_{>8N_0}\bigl[U\text{\O}(u^3)+\text{\O}(u^3)\bigr]\bigr\|_{L_t^\infty L_x^1}
&\lesssim_{N_0} \| a\|_{L_t^\infty L_x^6} \|\nabla \text{\O}(u^3)\|_{L_t^\infty L_x^{6/5}}\\
&\lesssim_{N_0} \| a\|_{L_t^\infty L_x^6} \|\nabla u\|_{L_t^\infty L_x^2} \|u\|_{L_t^\infty L_x^6}^2\lesssim_{N_0}1.
\end{align*}
We estimate the quartic and quintic terms in $N_v(u)$ using H\"older, \eqref{eq:uniform bounds}, and \eqref{405}:
\begin{align*}
&\bigl\|aP_{>8N_0}\bigl[U\text{\O}(u^4)+\text{\O}(u^4)\bigr]\bigr\|_{L_t^\infty L_x^1}\lesssim \|a\|_{L_t^\infty L_x^3} \|u\|_{L_t^\infty L_x^6}^4\lesssim 1,\\
&\bigl\|aP_{>8N_0}\bigl[U\text{\O}(u^5)+\text{\O}(u^5)\bigr]\bigr\|_{L_t^\infty L_x^1}\lesssim \|a\|_{L_t^\infty L_x^6} \|u\|_{L_t^\infty L_x^6}^5\lesssim 1.
\end{align*}
Putting everything together, we see that the contribution of the second term on RHS\eqref{404} to \eqref{Nz2 goal} is acceptable.

We now turn to the third term on RHS\eqref{404}.  By H\"older and \eqref{eq:dispersive}, 
\begin{align*}
\int_{t_1}^{t_2} \big|\big\langle  \langle \nabla\rangle^{-2} &\bigl([a,H] P_{>8N_0}\bar v\bigr), e^{-isH}\varphi\big\rangle\big|\,ds
\lesssim_\varphi |t_1|^{-\frac14}  \big\|\langle \nabla\rangle^{-2} \bigl([a,H] P_{>8N_0}\bar v\bigr)\big\|_{L_t^\infty L_x^{\frac{12}{11}}}.
\end{align*}
We regard the term on the right-hand side above as a bilinear operator $T(a,v)$ with symbol given by
\begin{align*}
\tfrac{H(\xi_2) - H(\xi)}{\langle \xi\rangle^2}\phi\bigl(\tfrac{\xi_1}{4N_0} \bigr)\bigl[1- \phi\bigl(\tfrac{\xi_2}{8N_0}\bigr) \bigr] = m(\xi_1, \xi_2)\cdot \xi_1,
\end{align*}
where 
$$ m(\xi_1, \xi_2)=-\frac{(2\gamma+\vert\xi\vert^2+\vert\xi_2\vert^2)(\xi_1+2\xi_2)}{\langle \xi\rangle^2(\vert\xi_2\vert\langle\xi_2\rangle+\vert\xi\vert\langle\xi\rangle)}\phi\bigl(\tfrac{\xi_1}{4N_0} \bigr)\bigl[1- \phi\bigl(\tfrac{\xi_2}{8N_0}\bigr) \bigr]
$$ 
is a bounded bilinear multiplier in view of Lemma~\ref{L:CM}.  Using also \eqref{405}, we get
\begin{align*}
\big\|\langle \nabla\rangle^{-2} \bigl([a,H] P_{>8N_0}\bar v\bigr)\big\|_{L_t^\infty L_x^{12/11}}
\lesssim \|\nabla a\|_{L_t^\infty L_x^{3/2}} \|v\|_{L_t^\infty L_x^4}\lesssim_{N_0}1.
\end{align*}
Thus, the contribution of the third term on RHS\eqref{404} to \eqref{Nz2 goal} is acceptable.

We now turn to the fourth and last term on RHS\eqref{404}.  By H\"older, \eqref{eq:dispersive}, and Bernstein, 
\begin{align*}
\int_{t_1}^{t_2}\! \big|\langle \dot{a}P_{>8N_0}\bar v, e^{-isH} \langle \nabla\rangle^{-2}\varphi\rangle\big|\,ds
&\lesssim_\varphi |t_1|^{-\frac12} \|\dot{a}\|_{L_t^\infty L_x^2} \|P_{>8N_0} v\|_{L_t^\infty L_x^2}\\
&\lesssim_{\varphi, N_0}\! |t_1|^{-\frac12} \|P_{\leq N_0} \dot{u}_2\|_{L_t^\infty L_x^3} \|u_2\|_{L_t^\infty L_x^6} \|\nabla v\|_{L_t^\infty L_x^2}.
\end{align*}
In view of \eqref{eq:uniform bounds} and \eqref{405}, we need only bound $P_{\leq N_0} \dot{u}_2$ in $L_t^\infty L_x^3$.  To this end, we use \eqref{eq:cq3}, Bernstein, and \eqref{eq:uniform bounds}:
\begin{align*}
\|P_{\leq N_0} \dot{u}_2\|_{L_t^\infty L_x^3}
&\lesssim \|P_{\leq N_0} (2\gamma-\Delta)u_1\|_{L_t^\infty L_x^3} + \sum_{k=2}^5 \|P_{\leq N_0}\text{\O} (u^k)\|_{L_t^\infty L_x^3}\\
&\lesssim_{N_0} \|u_1\|_{L_t^\infty L_x^3} + \sum_{k=2}^5 \|u\|_{L_t^\infty L_x^6}^k\lesssim_{N_0} 1.
\end{align*}
Thus, the contribution of the fourth term on RHS\eqref{404} to \eqref{Nz2 goal} is acceptable.  This completes the justification of \eqref{Nz2 goal} and so the proof of Proposition~\ref{P:weak convergence}.
\end{proof}

					%%%%%%%%%%%%%%%%%%%%%%%%%
					\section{Proof of Theorem~\ref{thm:wave ops2}}
					\label{section:fs2}
					%%%%%%%%%%%%%%%%%%%%%%%%%

In this section we prove Theorem~\ref{thm:wave ops2} and Corollary~\ref{cor}. We recall the norm
$$
\|u\|_{X_T}:=\sup_{t\geq T} \, t^{\frac12}\|u(t)\|_{H_x^{1,3}(\R^3)}.
$$

The proof of Theorem~\ref{thm:wave ops2} will be effected by running a contraction mapping argument simultaneously for $u$ and $z=M(u)$.  The necessity of exploiting the normal form transformation can be seen when one endeavors to estimate the quadratic terms appearing in the nonlinearity.  

\begin{proof}[Proof of Theorem~\ref{thm:wave ops2}.]
We define maps
\begin{equation*}
\left\{\begin{aligned}\relax
[\Phi_1(u,z)](t)&=V^{-1}z(t)-\gamma\jb^{-2}\vert u(t)\vert^2,\\[1ex]
\left[\Phi_2(u)\right](t)&=e^{-itH}Vu_++i\displaystyle\int_t^\infty e^{-i(t-s)H}N_z(u(s))\,ds,
\end{aligned}\right.
\end{equation*} 
where $N_z$ is as in \eqref{eq:cq z}. 

We will show that the map $(u,z)\mapsto\Phi(u,z):=(\Phi_1(u,z),\Phi_2(u))$ is a contraction on a suitable complete metric space, and so deduce that $\Phi$ has a unique fixed point $(u,z)$ in this space, which then necessarily solves \eqref{eq:nft}--\eqref{eq:cq z}.

For $0<\eta<1$ and $T>1$ to be determined below, we define
	\begin{align*}
	B_1=\{u: \|u\|_{L_t^\infty (\Hr)}\leq {4}\|u_+\|_{\Hr},\ \|u\|_{X_T}\leq4\eta\},
	\end{align*}
and
	\begin{align*}
	B_2=\{z: \|z\|_{L_t^\infty H_x^1}\leq 2\|u_+\|_{\Hr},\ \|V^{-1} z\|_{X_T}\leq 2\eta\},
	\end{align*}
where here and in what follows all space-time norms are taken over $(T,\infty)\times\R^3$ unless stated otherwise.  We define $B=B_1\times B_2$ and equip $B$ with the metric
$$
d((u,z),(\tilde{u},\tilde{z}))=\|u-\tilde{u}\|_{X_T}+8\|V^{-1}(z-\tilde{z})\|_{X_T}.
$$		

We first show that $\Phi:B\to B$.   By Sobolev embedding, for $(u,z)\in B$ and $t>T\geq 1$,
\begin{align*}
\|\jb^{-2}\vert u(t)\vert^2\|_{ H_x^1} + \|\jb^{-2}\vert u(t)\vert^2\|_{H_x^{1,3}}\lesssim \|\vert u(t)\vert^2\|_{L_x^{3/2}}\lesssim \|u(t)\|_{L_x^3}^2\lesssim \eta^2 t^{-1}.
	\end{align*}
Thus choosing $T=T(\|u_+\|_{\Hr})$ large enough, we have
\begin{align*}
\|[\Phi_1(u,z)](t)\|_{\Hr}&\leq\|V^{-1}z(t)\|_{\Hr}+\gamma \|\jb^{-2}\vert u(t)\vert^2\|_{ H_x^1}\\&\leq  2\|z\|_{L_t^\infty H_x^1}.
\end{align*}	
Similarly,
\begin{align*}
\|[\Phi_1(u,z)](t)\|_{H_x^{1,3}}\leq\|V^{-1}z(t)\|_{H_x^{1,3}}+\gamma \|\jb^{-2}\vert u(t)\vert^2\|_{H_x^{1,3}}\leq 4\eta t^{-1/2},
\end{align*}
provided $\eta$ is chosen small enough.  Thus $\Phi_1:B\to B_1$.

We next show that $\Phi_2:B_1\to B_2$.  We first estimate $N_z(u)$, which satisfies
\begin{equation}\label{eq:z nonlinearity}
U^{-1}N_z(u)=\text{\O}(u^2+u^3+u^4+u^5)+U^{-1}\tfrac{\nabla}{\jb^2}\cdot\text{\O}(u\nabla u)+ U\text{\O}(u^3+u^4+u^5).
\end{equation} 
We estimate the quadratic terms at fixed time $t>T\geq 1$, as follows: 
\begin{align*}
\big\|\big[\text{\O}(u^2)+U^{-1}\tfrac{\nabla}{\jb^2}\cdot\text{\O}(u\nabla u)\big](t)\big\|_{H_x^{1,3/2}}\lesssim t^{-1} \|u\|_{X_T}^2\lesssim t^{-1} \eta^2.
\end{align*}
Similarly, for $k\in\{2,3,4\}$ we have
\begin{align}\label{132}
\big\|\big[\text{\O}(u^{k+1})+U\text{\O}(u^{k+1})\big](t)\big\|_{H_x^{1,3/2}}
&\lesssim \|u(t)\|_{L_x^{3k}}^k\|u(t)\|_{H_x^{1,3}}\\
&\lesssim \|\vert\nabla\vert^{1-\frac1k}u(t)\|_{L_x^3}^k\|u(t)\|_{H_x^{1,3}}\notag\\
&\lesssim t^{-\frac{k+1}{2}}\|u\|_{X_T}^{k+1}\lesssim t^{-\frac{k+1}{2}}\eta^{k+1}.\notag
\end{align}
Combining the above, we deduce that 
\begin{align}\label{131}
\|U^{-1}N_z(u(t))\|_{H_x^{1,3/2}} \lesssim \sum_{k=1}^4 t^{-\frac{k+1}{2}}\eta^{k+1} \quad\text{uniformly for $t>T\geq 1$.}
\end{align}

To continue, we use Strichartz and \eqref{131} to estimate
\begin{align*}
\|\Phi_2(u)\|_{L_t^\infty H_x^1}
&\leq \|Vu_+\|_{L_t^\infty H_x^1}+C\|N_z(u)\|_{L_t^{4/3}H_x^{1,3/2}}\\
&\leq\|u_+\|_{\Hr}+C\sum_{k=1}^4 T^{-\frac{2k-1}{4}}\eta^{k+1} \leq2\|u_+\|_{\Hr}, 
\end{align*}
provided $T=T(\|u_+\|_{\Hr})$ is chosen sufficiently large.

We turn to estimating $V^{-1}\Phi_2(u)$ in the $X$-norm for $u\in B_1$.  By hypothesis, the dispersive estimate \eqref{eq:dispersive}, and \eqref{131}, for $t>T\geq 1$ we have
\begin{align*}
\|[V^{-1}\Phi_2(u)](t)\|_{H_x^{1,3}}&\!\leq \|V^{-1}e^{-itH}Vu_+\|_{H_x^{1,3}}\!+\Bigl\|\int_t^\infty\!\!\!\! V^{-1}[ie^{-i(t-s)H}\!N_z(u(s))]\,ds\Bigr\|_{H_x^{1,3}}\\
&\leq \eta t^{-\frac12} + \int_t^\infty \|e^{-i(t-s)H}U^{-1}N_z(u(s))]\|_{H_x^{1,3}}\, ds\\
&\leq \eta t^{-\frac12} + C\int_t^\infty \vert t-s\vert^{-\frac12}\sum_{k=1}^4 s^{-\frac{k+1}{2}}\eta^{k+1}\,ds\\
&\leq \eta t^{-\frac12} + C\sum_{k=1}^4t^{-\frac{k}{2}}\eta^{k+1} \leq 2\eta t^{-\frac{1}{2}},
\end{align*}
provided $\eta$ is chosen sufficiently small.  This completes the proof that $\Phi:B\to B$.

We next claim that $\Phi$ is a contraction with respect to the metric defined above.  First, for $(u,z), (\tilde u, \tilde z)\in B$, we estimate
\begin{align*}
\|\Phi_1(u,z)-\Phi_1(\tilde{u},\tilde{z})\|_{X_T}
&\leq\|V^{-1}(z-\tilde{z})\|_{X_T}+\gamma\|\jb^{-2}(\vert u\vert^2-\vert \tilde{u}\vert^2)\|_{X_T}\\
&\leq \tfrac18 d((u,z),(\tilde{u},\tilde{z})) +C \sup_{t\geq T}t^{\frac12}\|(u+\tilde{u})(t)(u-\tilde{u})(t)\|_{L_x^{3/2}}\\
&\leq \tfrac18 d((u,z),(\tilde{u},\tilde{z})) + C\eta T^{-\frac12}\|u-\tilde{u}\|_{X_T} \\
&\leq \tfrac14 d((u,z),(\tilde{u},\tilde{z})),
\end{align*}
provided $\eta$ is sufficiently small.  

By \eqref{eq:dispersive}, for $t>T\geq 1$ we estimate
\begin{align*}
\|V^{-1}[\Phi_2(u)&-\Phi_2(\tilde{u})](t)\|_{H_x^{1,3}}\\
&\leq\Bigl\|\int_t^\infty V^{-1}\big(ie^{-i(t-s)H}[N_z(u(s))-N_z(\tilde{u}(s))]\big)\,ds\Bigr\|_{H_x^{1,3}}\\
&\leq \int_t^\infty \vert t-s\vert^{-\frac12}\|U^{-1}[N_z(u(s))-N_z(\tilde{u}(s))]\|_{H_x^{1,3/2}}\,ds.
\end{align*}
Writing $w$ to indicate that either $u$ or $\tilde{u}$ may appear, we have
\begin{align*}
U^{-1}[N_z(u)-N_z(\tilde{u})]
&=\sum_{k=1}^4\text{\O}[w^k(u-\tilde{u})]+\tfrac{U^{-1}\nabla}{\jb^2}\cdot[(u-\tilde{u})\nabla w+w\nabla(u-\tilde{u})] \\
&\quad +U\sum_{k=2}^4 \text{\O}[w^k(u-\tilde{u})].
\end{align*}
We estimate the contribution of the quadratic terms via
\begin{align*}
\int_t^\infty \vert t-s\vert^{-\frac12}&\big\|\text{\O}[w(u-\tilde{u})](s)+\tfrac{U^{-1}\nabla}{\jb^2}\cdot[(u-\tilde{u})\nabla w+w\nabla(u-\tilde{u})](s)\big\|_{H_x^{1,3/2}}\,ds\\
&\lesssim \|w\|_{X_T}\|u-\tilde{u}\|_{X_T}\int_t^\infty \vert t-s\vert^{-\frac12} s^{-1}\, ds\lesssim t^{-\frac12}\eta\|u-\tilde{u}\|_{X_T}.
\end{align*}
Arguing as in \eqref{132}, we obtain
\begin{align*}
\int_t^\infty &\vert t-s\vert^{-\frac12}\big\| \sum_{k=2}^4\text{\O}[w^k(u-\tilde{u})](s)+U\text{\O}[w^k(u-\tilde{u})](s)\big\|_{H_x^{1,3/2}}\,ds\\
&\lesssim \|w\|_{X_T}^k\|u-\tilde{u}\|_{X_T}\int_t^\infty \vert t-s\vert^{-\frac12} s^{-\frac{k+1}{2}}\,ds \lesssim t^{-\frac{k}{2}}\eta^k \|u-\tilde{u}\|_{X_T}.
\end{align*}
Thus for $\eta$ sufficiently small we get
$$
8\|V^{-1}[\Phi_2(u)-\Phi_2(\tilde{u})]\|_{X_T}\leq \tfrac14 d((u,z),(\tilde{u},\tilde{z})).
$$

This completes the proof that $\Phi$ is a contraction on $B$.  Hence there exists a unique $(u,z)\in B$ such that $\Phi(u,z)=(u,z)$. In particular $z=M(u)$ and $(u,z)$ solves \eqref{eq:nft}--\eqref{eq:cq z} on $(T,\infty)\times\R^3$. We note that by construction we have $u_1\in H_x^1$ and $u\in L_x^3\cap L_x^6$. In particular, $q(u)=\vert u\vert^2+2u_1\in L_x^2$ and hence $u(t)\in\mathcal{E}$ for $t>T$.

For $\gamma\in[\frac23,1)$, Theorem~\ref{thm:gwp} guarantees that the solution $u$ can be extended (in a unique way) to be global in time.  For $\gamma\in(0,\frac23)$, global existence follows from \cite[Theorem~1.3]{KOPV}, while uniqueness in the energy space follows from Theorem~\ref{thm:gwp} (see also Remark~\ref{R:unique}).

Next we show that \eqref{eq:scattering1} holds; indeed, we prove the stronger claim \eqref{E:scattering11}. We first note that Strichartz combined with \eqref{131} gives
$$
 \|z(t)-e^{-itH}Vu_+\|_{H_x^1} \lesssim t^{-1/4},
$$
which in turn implies
$$
\|V^{-1}z(t)-V^{-1}e^{-itH}Vu_+\|_{\Hr}  \lesssim t^{-1/4}.
$$ 
As $z=M(u)$, for $t>T$ we have
\begin{align*}
\|V^{-1}z(t)-u(t)\|_{\Hr}&\lesssim\|\jb^{-2}\vert u(t)\vert^2\|_{H_x^1}\lesssim \|\vert u(t)\vert^2\|_{L_x^{3/2}}\lesssim t^{-1}\|u\|_{X_T}^2,
\end{align*}
Therefore, by the triangle inequality we may conclude that
$$
\| u(t)-V^{-1}e^{-itH}Vu_+]\|_{\Hr} \lesssim t^{-1/4}.
$$

By the arguments presented so far, it is clear that $u$ is the unique solution in $B_1$ that obeys \eqref{eq:scattering1}.  This is slightly weaker than is claimed in Theorem~\ref{thm:wave ops2}, which places no restrictions on the $\Hr$ norm of alternate solutions $v(t)$, nor any restriction on the value of $T$ for which $\|v\|_{X_T} \leq 4\eta$; however, any solution $v(t)$ obeying \eqref{eq:scattering1} must have 
$$
\|v\|_{L_t^\infty ([T,\infty);\Hr)}\leq {4}\|u_+\|_{\Hr}
$$
for some $T$ large enough.  Thus the equality of $v(t)$ and $u(t)$ follows from the contraction mapping argument above with $T$ large enough combined with uniqueness in the energy space.
\end{proof}

Finally, we prove Corollary~\ref{cor}. 
 
\begin{proof}[Proof of Corollary~\ref{cor}] The proof consists of showing that smallness of the weighted norms implies the smallness condition \eqref{eq:scattering smallness condition}.  In view of \eqref{E:matrix prop}, it suffices to show
\begin{align*}
\| e^{\pm itH} u_+\|_{H_x^{1,3}}\lesssim \vert t\vert^{-\frac12}\eta \qtq{and} \| e^{\pm itH}U^{-1}\Re u_+\|_{H_x^{1,3}}\lesssim \vert t\vert^{-\frac12}\eta.
\end{align*}

By the dispersive estimate \eqref{eq:dispersive} and H\"older,
\begin{align*}
\|e^{\pm itH}u_+\|_{H_x^{1,3}}&
\lesssim \vert t\vert^{-\frac12}\|\langle\nabla\rangle u_+\|_{L_x^{3/2}}\lesssim \vert t \vert^{-\frac12}\|\langle x\rangle^{\frac12+}\langle\nabla\rangle u_+\|_{L_x^2}
\end{align*}
and
\begin{align*}
\|e^{\pm itH}U^{-1}\Re u_+\|_{H_x^{1,3}}&
\lesssim \vert t\vert^{-\frac12}\|U^{-\frac56}\langle\nabla\rangle \Re u_+\|_{L_x^{3/2}}.
\end{align*}
Using H\"older and Sobolev embedding, we obtain
\begin{align*}
\|\nabla U^{-\frac56}\Re u_+\|_{L_x^{3/2}}&\lesssim\|\langle\nabla\rangle u_+\|_{L_x^{3/2}}\lesssim \|\langle x\rangle^{\frac12+}\langle\nabla\rangle u_+\|_{L_x^2}\\
\|U^{-\frac56}\Re u_+\|_{L_x^{3/2}}&\lesssim \|\vert\nabla\vert^{\frac56}U^{-\frac56}\Re u_+\|_{L_x^{18/17}}
	\lesssim \|\langle x\rangle^{\frac43+}\langle\nabla\rangle^{\frac56}\Re u_+\|_{L_x^2}.
\end{align*} 

This completes the proof of the corollary.
\end{proof}

\end{document}